\documentclass[11pt,a4paper,reqno]{amsart}
\usepackage{vmargin,color}
\usepackage[latin1]{inputenc}
\usepackage{amsmath,amsfonts,amsthm,epsfig,graphicx,esint}
\usepackage{caption}

\def\B{\mathcal{B}}

\def\H{\mathcal{H}}
\def\I{\mathcal{I}}

\def\V{\mathcal{V}}

\def\QQ{\mathcal{Q}}
\def\NN{\mathcal{N}}

\def\N{\mathbb N}
\def\R{\mathbb R}

\def\Z{\mathbb Z}
\def\SS{\mathbb S}

\def\Om{\Omega}

\def\a{\alpha}

\def\g{\gamma}
\def\de{\delta}
\def\e{\varepsilon}
\def\k{\kappa}
\def\l{\lambda}
\def\s{\sigma}
\def\om{\omega}
\def\vphi{\varphi}

\def\Lip{{\rm Lip}}

\def\Div{{\rm div}\,}

\def\dist{{\rm dist}}
\def\loc{{\rm loc}}

\def\spt{{\rm spt}}

\def\weak{\rightharpoonup}

\def\pa{\partial}

\def\cc{\subset\subset}

\def\ac{\mathcal{AC}}
\def\AC{\mathcal{AC}_{\e}}

\newtheorem{theorem}{Theorem}[section]
\newtheorem{remark}{Remark}[section]

\numberwithin{equation}{section}
\numberwithin{figure}{section}

\title{Asymptotic behavior of a diffused interface \\ volume-preserving mean curvature flow}
\author{Matteo Bonforte}
\address{Departamento de Matem\'aticas, Universidad Aut\'onoma de Madrid Campus de Cantoblanco, 28049, Madrid}
\email{matteo.bonforte@uam.es}
\author{Francesco Maggi}
\address{Department of Mathematics, The University of Texas at Austin, 2515 Speedway, Stop C1200, Austin TX 78712-1202, United States of America}
\email{maggi@math.utexas.edu}
\author{Daniel Restrepo}
\address{Department of Mathematics, Johns Hopkins University, Baltimore, MD, United States of America}
\email{drestre1@jh.edu}

\begin{document}
	
\begin{abstract} {\rm We consider a diffused interface version of the volume-preserving mean curvature flow in the Euclidean space, and prove, in every dimension and under natural assumptions on the initial datum, exponential convergence towards single ``diffused balls''.}
\end{abstract}
		
\maketitle

\tableofcontents
	
\section{Introduction} \subsection{Overview}\label{section overview} In this paper we introduce a PDE reformulation of the classical volume-preserving mean curvature flow (VPMCF) in $\R^n$ where the role of the perimeter functional is played by the Allen--Cahn energy. From the physical viewpoint this reformulation seems well justified, since it consists in replacing the classical {\it sharp} interface model for surface tension based on perimeter minimization with the equally interesting and important {\it diffused} interface model based on the minimization of the Allen--Cahn energy. From the mathematical viewpoint working in the diffused setting clears up the analysis from the ambiguities brought in by the formation of singularities characteristic of geometric flows, which are directly reflected into the abundance of non-equivalent weak formulations of the VPMCF.

\medskip

Our main result proves, in every dimension and under a variety of natural assumptions on the initial datum, exponential convergence of the diffused VPMCF towards single ``diffused balls''. This result is indeed stronger than the presently known analogous results for the classical VPMCF, see Remark \ref{remark soa for vpmcf} below.

\medskip

The diffused interface model we adopt is based on the volume-constrained minimization of the classical Allen--Cahn functional
\begin{equation}
\label{AC eps}
\AC(u)=\e\,\int_{\R^n}|\nabla u|^2+\frac1\e\,\int_{ \R^n}W(u)\,,\qquad\e>0\,,
\end{equation}
defined on (dimensionless) density functions $u:\R^n\to[0,1]$, and requiring the choice of (dimensionless) double-well potential $W:[0,1]\to[0,\infty)$. We will always require that $W\in C^{2,1}[0,1]$ and that $W$ satisfies the standard non-degeneracy assumptions
\begin{equation}
\label{W basic}
\mbox{$W(0)=W(1)=0$}\,,\quad\mbox{$W>0$ on $(0,1)$}\,,\quad\mbox{$W''(0),W''(1)>0$}\,,
\end{equation}
as well as the normalization
\begin{equation}
\label{W normalization}
\int_0^1\sqrt{W}=1\,.
\end{equation}
As usual, $\e$ has the dimensions of length, so that $\AC(u)$ has the dimensions of (codimension one) area.

\medskip

Particular care must be put in the choice of the volume potential $V:[0,1]\to[0,\infty)$ used to impose the volume constraint on $u$. Indeed, while any choice of $V$ satisfying $V(1)>0$ and $V(0)=0$ will return the correct volume constraint in the sharp interface limit $\e\to0^+$ (and will thus be acceptable form the physical viewpoint), not every choice of $V$ will result in a mathematical model that is either well-posed or feasible of in-depth analysis.

\medskip

A natural choice for $V$ is suggested by the classical isoperimetric lower bound on $\AC(u)$, and consists in taking
\begin{equation}
\label{choice of V}
\mbox{$V(r)=\Phi(r)^{n/(n-1)}$, where $\Phi(r)=\int_{0}^{r} \sqrt{W}$ for $r\in[0,1]$}\,.
\end{equation}
Indeed, by a classical application of the Cauchy--Schwartz inequality and the chain rule, we have that
\[
\e\,|\nabla u|^2+\frac{W(u)}\e\ge 2\,|\nabla u|\,\sqrt{W(u)}=2\,|\nabla(\Phi\circ u)|\,,
\]
from which the isoperimetric lower bound
\begin{equation}
  \label{modicamortola}
  \AC(u)\ge2\,|D[\Phi\circ u]|(\R^n)\ge 2\,n\,\om_n^{1/n}\,\V(u)^{(n-1)/n}\,,\qquad\V(u):=\int_{\R^n}V(u)\,,
\end{equation}
follows, so that\footnote{Here, $|Dv|$ denotes the total variation measure of $v\in L^1_{\rm loc}(\R^n)$, $\om_n$ is the volume of the unit radius ball in $\R^n$, $|E|$ and $P(E)$ denote the volume and perimeter of $E\subset\R^n$, and $B^{(m)}$ stands for the ball of volume $m$ with center at the origin in $\R^n$.}
\[
n\,\om_n^{1/n}\,m^{(n-1)/n}=P(B^{(m)})=\inf\{P(E):|E|=m\}\,,\qquad m>0\,,
\]
is the optimal value of the {\bf (Euclidean) isoperimetric problem}.

\medskip

Thanks to \eqref{modicamortola}, our choice \eqref{choice of V} of $V$ is instrumental for obtaining a well-posed {\bf diffused interface (Euclidean) isoperimetric problem},
\begin{equation}
	\label{Psi eps m}
	\Psi(\e,m)=\inf \Big\{\AC(u):\V(u)=m\,,u\in L^1_{\rm loc}(\R^n;[0,1])\Big\}\,,\qquad \e\,,m>0\,.
\end{equation}
Indeed, by \eqref{modicamortola} and by the (sharp) isoperimetric inequality for functions of bounded variation, we see that
\begin{equation}
  \label{isop lb}
\Psi(\e,m)>2\,n\,\om_n^{1/n}\,m^{(n-1)/n}\,,\qquad\forall \e,m>0\,,
\end{equation}
while a simple comparison argument shows that this inequality is saturated in the limit as $\e\to0^+$. It is important to keep in mind that simpler choices of $V$, like $V(t)=t$ or $V(t)=t^2$, would have led\footnote{The size of $V(t)$ for $t\to 0^+$ plays a crucial role here. Our choice of $V$ satisfies $V(t)={\rm O}(t^{2\,n/(n-1)})$ as $t\to 0^+$. This is not the only property of $V$ that plays an important role in our analysis though, and the close relation between $V$ and $W$ will be repeatedly used.} to degenerate minimization problems where every competitor has positive energy but where the infimum of the energy is equal to zero.

\medskip

Problem $\Psi(\e,m)$ has been studied in great detail in \cite{maggi-restrepo}. Some of the results obtained therein will play an important role in the present paper, and will therefore be summarized in Section \ref{section preliminaries}. For the moment, with the sole intent of formulating our main result as quickly as possible, we just recall from \cite{maggi-restrepo} that, in the physical regime where $\e\in(0,\e_0\,m^{1/n})$ for some {\it universal constant}\footnote{By universal constant we mean a constant depending only on the dimension $n$ and on the double-well potential $W$. By $C(a,b,\dots)$ we denote a constant depending only on $n$, $W$, and the arguments $a$, $b$, etc..} $\e_0$, there is a unique minimizer $\zeta_{\e,m}$ in the class of the radially symmetric, strictly decreasing, and everywhere positive functions on $\R^n$ with maximum at the origin; and that $u$ is a minimizer of $\Psi(\e,m)$ if and only if $u=\tau_{x_0}[\zeta_{\e,m}]$ for some\footnote{We set $\tau_{x_0}[v](x)=v(x-x_0)$ for every $x,x_0\in\R^n$ and $v:\R^n\to\R^m$.} $x_0\in\R^n$. Each $\zeta_{\e,m}$ solves the {\bf diffused constant mean curvature equation}
\[
2\,\e^2\,\Delta\zeta_{\e,m}=W'(\zeta_{\e,m})-\e\,\Lambda_{\e,m}\,V'(\zeta_{\e,m})\qquad\mbox{on $\R^n$}\,,
\]
where $\Lambda_{\e,m}\to 2\,(n-1)\,\om_n^{1/n}\,m^{-1/n}$ as $\e\to 0^+$ (and is thus positive in the physical regime $\e\in(0,\e_0\,m^{1/n})$).

\medskip

The {\bf diffused VPMCF} is then defined as the $L^2$-gradient flow of $\AC$ with a Lagrange multiplier modification that preserves $\V$ along the flow: that is, we consider the parabolic initial value problem\footnote{Given $t\ge 0$ and $u:\R^n\times[0,\infty)\to\R$, we set $u(t):\R^n\to\R$ for the function defined by $u(t)(x):=u(x,t)$ ($x\in\R^n$).}
\begin{equation}
\label{diffused VPMCF}\tag{DF}
\left\{\begin{split}
&\e^2\,\pa_tu=2\, \e^2\, \Delta u - W'(u)+ \e\,\lambda_\e[u(t)]\, V'(u)\,,\qquad\mbox{on $\R^n\times(0,\infty)$}\,,
\\
&u(0)=u_0\,,
\end{split}
\right .
\end{equation}
where we have introduced the Lagrange multiplier functional\footnote{Notice that $\lambda_\e[v]$ is defined in $[0,\infty]$ on any $v:\R^n\to[0,1]$ with $|\{0<v<1\}|>0$ -- this condition guarantees indeed that $\int_{\R^n}V'(v)^2>0$.}
\begin{equation}
  \label{lambda intro}
  \lambda_\e[v]
= \frac{\int_{\R^n}2\,\e^2\, |\nabla v|^2\,V''(v) +  W'(v)\,V'(v)}{\e\,\int_{\R^n}V'(v)^2}\,,
\end{equation}
whose choice guarantees, on smooth solutions of the flow, that
\[
\frac{d}{dt}\,\V(u(t))=\int_{\R^n}V'(u(t))\,\pa_tu(t)=0\,,\qquad\mbox{i.e., $\V(u(t))=\V(u_0)$ for all $t>0$}\,.
\]
From the viewpoint of classical parabolic theory, \eqref{diffused VPMCF} presents some peculiar features since it is a non-autonomous semilinear PDE, where the non-autonomy is due to the Lagrange multiplier $\l_\e[u(t)]$, which, in turn, is non-local in space (its determination requires knowledge of $u(t)$ over the whole $\R^n$). The following theorem is our main result concerning the long-time behavior of \eqref{diffused VPMCF}.

\begin{theorem}\label{theorem convergence without bubbling} If $n\ge 2$ and $W\in C^{2,1}[0,1]$ satisfies \eqref{W basic} and \eqref{W normalization}, then there exists a universal constant $\e_0>0$ with the following property. If $\e \in (0, \e_0)$, $u_0\in W^{2,p}(\R^n;[0,1])$ for all $p\ge 2$, $\V(u_0)=1$, and
\begin{eqnarray}
\label{hp double energy}
&&\mbox{either $\AC(u_0)<2\,\Psi(\e,1/2)$}\,,
\\
\label{hp compact spt}
&&\mbox{or $\spt\,u_0$ is compact}\,,
\end{eqnarray}
then
\begin{equation}\label{eq convergence rate energy}
	0\le \AC(u(t))-\Psi(\e,1) \leq C(\e,u_0)\, e^{-t/C(\e)}\,,\qquad\forall t>\frac1{C(\e,u_0)}\,,
\end{equation}
and there exists a unique $x_0\in \R^n$ (depending on $\e$ and $u_0$) such that for all $p>2$ and $t>1/C(\e,u_0)$
\begin{equation}\label{eq convergence norms}
\big\|u(t)-\tau_{x_0}[\zeta_{\e,1}]\big\|_{(W^{2,p}\cap W^{1,2})(\R^n)} \leq C(u_0,\e,p)\, e^{-t/C(\e)}\,.
\end{equation}
\end{theorem}

\begin{remark}[Asymptotic analysis of the VPMCF]\label{remark soa for vpmcf}
  {\rm It is convenient to briefly review the state of the art concerning convergence to equilibrium for the classical VPMCF (with the disclaimer that, due to singularities formation, these various results may pertain to different  weak formulations of the VPMCF). First, convergence to a single ball has been proved under a variety of suitable geometric restrictions on the initial datum that can be shown to be preserved along the flow, and that {\it exclude} singularities formation: these are uniform convexity \cite{huisken_vpmcf}, $C^1$-proximity to a sphere \cite{escher-simonett}, star-shapedeness \cite{kim-kwon}, or integral pinching conditions \cite{haozhao}. In general, singularity formation may lead to convergence to multiple balls, a phenomenon called {\it bubbling} (see \cite[Theorem 1.4]{fjmo22} for an example) and it is thus unclear for which class of initial data one should expect convergence towards a single ball (of the same volume as the initial datum) or towards multiple balls (all with a same volume equal to a fraction of the initial one). Since perimeter decreases along the VPMCF, a natural {\bf conjecture} is that, if a unit volume initial datum has perimeter {\it strictly} less than {\it twice} the perimeter of a ball of volume $1/2$, then convergence to a single ball should follow (with exponential rates of convergence). This conjecture has been proved  for ``flat flow'' solutions of the VPMCF, and in dimensions $n=2$ and $n=3$ respectively, in the recent papers \cite{julin-morini-ponsiglione-spadaro,jmos24}. Moreover, again when $n=2,3$, in \cite{julin-niinikoski} it is proved that flat flow solutions always resolve, as $t\to\infty$, as finite union of balls with possibly moving centers. All these results are based on powerful quantitative bubbling results for sets with $L^2$-small mean curvature oscillation. The restriction to dimensions $n=2,3$ is strongly correlated, on the one hand, to the fact that the $L^2$-oscillation of the mean curvature is the dissipation of the VPMCF, and is thus the quantity to work with in this setting; and that, on the other hand, the critical $L^p$-space for the regularity theory of the mean curvature of a boundary in $\R^n$ is $p=n-1$. For these reasons, the further extension of these methods to dimensions $n\ge 4$ seems delicate.}
\end{remark}

\begin{remark}[On assumption \eqref{hp double energy}]
  {\rm Assumption \eqref{hp double energy} amounts to asking that the initial datum has strictly less energy than twice that of two diffused balls of volume $1/2$. Hence, by proving Theorem \ref{theorem convergence without bubbling} under \eqref{hp double energy} we have proved the validity, in every dimension, of the diffused analogue of the VPMCF-conjecture mentioned in Remark \ref{remark soa for vpmcf}.}
\end{remark}

\begin{remark}[On assumption \eqref{hp compact spt}]
  {\rm Proving Theorem \ref{theorem convergence without bubbling} under assumption \eqref{hp compact spt} is somehow more striking than doing so under assumption \eqref{hp double energy}, since \eqref{hp compact spt} allows for initial data with arbitrarily large  energy as well as for initial data that is arbitrarily close to the characteristic functions of {\it any} bounded set with finite perimeter and unit volume. Indeed, if $E\subset\R^n$ is a bounded set of finite perimeter with $|E|=1$, then by a standard construction, we can find a family $\{v_\e\}_{\e>0}$ of smooth and compactly supported functions on $\R^n$ such that $v_\e\to 1_E$ in $L^1(\R^n)$ and $\AC(v_\e)\to 2\,P(E)$ as $\e\to 0^+$, with $\spt\,v_\e\cc \{x:\dist(x,E)<1\}$ for all $\e>0$.}
\end{remark}

\subsection{Open problems and metastability} Before presenting the proof of Theorem \ref{theorem convergence without bubbling}, and the various intermediate results behind it, we briefly introduce some interesting problems related to Theorem \ref{theorem convergence without bubbling}.

\medskip

A first natural problem is addressing the existence of non-compactly supported initial data such that the resulting flow does not converge to a single diffused ball, but rather resolves into a superposition of time-drifting diffused balls (compare with conclusion \eqref{cmmm} in Theorem \ref{theorem bubbling one}). Since the physical or numerical relevance of non-compactly supported initial data is unclear, this is probably a question of very theoretical flavor; still, answering to it may be challenging.

\medskip

In the $\e\to 0^+$-limit, and for suitably prepared initial data, the diffused VPMCF should converge to a weak formulation of the VPMCF, and, indeed, this kind of convergence has been established, in absence of volume-preserving Lagrange multipliers, in  \cite{ilmanen,bethuelorlandismets}, and for the VPMCF but under spherical symmetry assumptions on $\Om$, in \cite{bronsardStoth}. For this reason, another natural problem related to Theorem \ref{theorem convergence without bubbling} would be understanding whether the $\e$-dependency of the decay rates \eqref{eq convergence rate energy} and \eqref{eq convergence norms} can be dropped off or not: the corresponding $\e$-independent decay rates could then be transferred to the VPMCF. It seems natural to conjecture that, if $u_0$ is compactly supported {\it and} is such that $\AC(u_0)<2\,\Psi(\e,1/2)$ (that is, if both \eqref{hp double energy} and \eqref{hp compact spt} hold), then $\e$-independent decay rates to a single diffused ball hold true, thus providing a strategy to extend the results of \cite{julin-niinikoski,julin-morini-ponsiglione-spadaro,jmos24} to arbitrary dimensions.

\medskip

We do not expect, however, that $\e$-independent decay rates to a single diffused ball should hold for a generic compactly supported initial datum $u_0$: such rates should hold, at best, only for $t>T_\e=T_{\e}(u_0)$ with $T_\e\to\infty$ as $\e\to 0^+$. Indeed, a flow starting from a superposition of two diffused balls truncated so to have compact support should spend a time $T_\e\to\infty$ as $\e\to 0^+$ close to its ``metastable'' initial datum, before eventually converging to a single diffused ball.

\medskip

Evidence in support of such metastable scenario can be found in a series of ``slow-motion'' results concerning the initial value problem
\begin{equation}
\label{parabolic AC}
\left\{\begin{split}
&\pa_tv=2\, \e^2\, \Delta v - W'(v)+ \e\,\mu_\e[v(t)]\,\,,\qquad\mbox{on $\Om\times(0,\infty)$}\,,
\\
&\nabla_{\nu_\Om}v(t)=0\,,\hspace{3.7cm}\qquad\mbox{on $\pa\Om$}\,,
\\
&v(0)=u_0\,,\qquad\hspace{4.2cm}\mbox{on $\Om$}\,,
\end{split}
\right .
\end{equation}
defined on a bounded open set $\Om\subset\R^n$ with regular boundary, and involving the Lagrange multiplier $\mu_\e[v]=\int_\Om W'(v)/(\e\,|\Om|)$, where the choice\footnote{The boundedness of $\Om$ allows one to work with the simpler volume potential $V(t)=t$, which also brings some remarkable simplifications in the form of the volume-preserving Lagrange multiplier. Indeed, comparing the definition of $\l_\e$ with that of $\mu_\e$ we see how the former choice, crucial in ensuring the well-posedness of $\Psi(\e,m)$ as a minimization problem on $\R^n$, leads to the presence of the possibly degenerate denominator $\e\,\int_{\R^n}V'(v)^2$ (in place of the constant denominator $\e|\Om|$), and of the additional term $|\nabla v|^2\,V''(v)$($\approx |\nabla v|^2\,v^{2/(n-1)}$ for $v$ small) at numerator.}
of $\mu_\e$ is such that $\int_\Om v(t)=\int_\Om u_0$ for all $t\ge0$. Problem \eqref{parabolic AC} has indeed been the object of study in several papers, as we are now going to informally\footnote{Our review is informal in the sense that we will gloss over specific assumptions made in the reviewed papers on the double well potential $W$ and the initial data $u_0$. Moreover, the range of $u$ is often assumed to be $\R$, while we have chosen to consider functions with range in $[0,1]$.} review. Before doing that, let us stress that in the slow-motion literature it is customary to work with the parabolic operator $\pa_tv-2\,\e^2\,\Delta v$ in place of the one used in \eqref{diffused VPMCF}, namely, $\e^2\,(\pa_t u-2\,\Delta u)$. In practice this is a small difference, since one can pass from one setting to the other by just rescaling solutions in time, according to the relation $v(t)=u(\e^2\,t)$. The choice made in \eqref{diffused VPMCF} seems more natural, since the resulting flow is the one that is expected to converge, in the $\e\to 0^+$-limit, to the VPMCF. It is important to keep this difference in mind when comparing Theorem \ref{theorem convergence without bubbling} to results from the slow-motion literature, which are typically formulated on \eqref{parabolic AC}.

\medskip

In dimension $n=1$, when $\Om$ is an interval, building up on the pioneering work \cite{bronkohnCPAM,bronkohnJDE}, in \cite{grant} (see also \cite{bellettiniHNnovaga}) the following result is proved: if $u_0:\Om\to\{0,1\}$ has $N$-many jumps and $\{u_0^\e\}_\e$ is a family of initial data for \eqref{parabolic AC} such that $\|u_0^\e-u_0\|_{L^1(\Om)}\to 0$ as $\e\to 0^+$, then solutions $\{v^\e(t)\}_\e$ of \eqref{parabolic AC} are such that
\begin{equation}
  \label{grant}
\lim_{\e\to 0^+}\sup_{t<C\,e^{C/\e}}\,\int_\Om|v^\e(t)-u_0|=0\,.
\end{equation}
In terms of $u^\e(t)=v^\e(t/\e^2)$, we thus have $\sup_{t<C\,\e^2\,e^{C/\e}}\,\int_\Om|u^\e(t)-u_0|=0$ as $\e\to 0^+$, which is a non-trivial information since $\e^2\,e^{C/\e}\to\infty$ as $\e\to 0^+$. When $N\ge 2$ we can interpret this statement as a metastability result, since the expectation is that $u^\e(t)$ will eventually converge to a single transition layer.

\medskip

Concerning dimensions $n\ge2$, it is shown in \cite{alifusco,aliBfusco} that for initial data resembling the characteristic function of a sufficiently small ball contained in $\Om$ and located near $\pa\Om$, solutions $\{v^\e(t)\}_\e$ to \eqref{parabolic AC} will remain close to the characteristic function of a ball contained in $\Om$ on an interval of times $(0,T_\e)$ with $T_\e={\rm O}(e^{C/\e})$ as $\e\to 0^+$. This is another result supporting the metastability scenario: indeed, the expected attractors of \eqref{parabolic AC} as $t\to\infty$ should be close, for $\e$ small, to minimizers of $\inf\{P(E;\Om):E\subset\Om\,,|E|=m\}$; however, such minimizers, at small volume $m$, look like half-balls centered at the point of highest mean curvature of $\pa\Om$ (see \cite{fall,maggimihaila}), and thus will {\it never} be balls contained in $\Om$.

\medskip

Another slow-motion result in higher dimension has been obtained in \cite{murrayrinaldi,leonimurrayPAMS}. Its interest lies in the fact that it assumes the proximity of the initial data to generic perimeter minimizing sets (thus, not necessarily to balls/intervals); the drawback is that proximity of solutions $v^\e(t)$ is shown to be preserved only on a time interval $(0,T_\e)$ with $T_\e=C/\e$ so that, in terms of $u^\e(t)=v^\e(t/\e^2)$, proximity preservation is shown only on a time interval $(0,C\,\e)$, and no information survives in the limit $\e\to 0^+$.

\subsection{Analysis of the diffused VPMCF}\label{section main theorems} We now provide a detailed breakdown of a series of results leading to Theorem \ref{theorem convergence without bubbling}.

\medskip

The starting point of our analysis is to establish the existence of solutions of the diffused VPMCF and their basic regularity properties. This is a nontrivial task due to the presence of the Lagrange multiplier functional $\l_\e$. In addition to being non-local and requiring integrating on a non-compact set, this term brings two main technical complications into our analysis: 1) the possible degeneration of the flow because of the smallness of the $\int_{\R^n}V'(u)^2$-term at the denominator, and 2) the lack of regularity because of the $|\nabla u|^2\,V''(u)$-term at numerator, since $V''(u)$ is only H\"older continuous. These difficulties are addressed by first approximating $V$ with regularized volume potentials $V_\de$ such that $\Lip(V_\de'';[0,1])<\infty$, and by then discussing the delicate limit as $\de\to 0^+$. Boundedness, Lipschitz and H\"older continuity estimates on $\l_\e$, and on its variant $\l_{\e,\de}$ associated with $V_\delta$, are discussed in Section \ref{section estimates for lambda}, and are obtained by borrowing the geometric viewpoint of ``nucleation lemmas'' from the theory of isoperimetric clusters \cite{Almgren76}. With these estimates at hand we can implement a fixed point argument to show the existence of mild solutions to the modification of \eqref{diffused VPMCF} obtained by replacing $V$ with $V_\de$. We can then start bootstrapping regularity and monotonicity properties for the $\de$-approximating flows, up to the point where enough information is obtained and the $\de\to 0^+$ limit can be addressed, thus establishing the following theorem (proved in Section \ref{section proof of theorem 1}).

\begin{theorem}[Existence, regularity, and entropies of the flow]\label{thm existence and regularity} If $\e>0$, $n\ge 2$, $W\in C^{2,1}[0,1]$ satisfies \eqref{W basic} and \eqref{W normalization}, $u_0\in W^{2,p}(\R^n;[0,1])$ for every $p\ge2$, and $\V(u_0)=1$, then there exists a unique $u\in C^0(\R^n\times[0,\infty))\cap C^2_{\rm loc}(\R^n\times(0,\infty))$ that is a classical solution of \eqref{diffused VPMCF}. Moreover:

\medskip

\noindent {\bf (i):} for every $p\ge 2$ and $t_0>0$ we have
\begin{eqnarray}
  \label{eq reg bounds uniform}
  \sup_{t>0}\big\{\| \pa_t u(t)\|_{L^p(\R^n)}\,,\| u(t)\|_{W^{2,p}(\R^n)}\big\} \leq  C(\e,u_0,p)\,,
\end{eqnarray}
\begin{eqnarray}
\label{eq reg bounds}
\sup_{t>t_0}\big\{
\|\pa_{tt}u(t)\|_{L^p(\R^n)}\,,\| \pa_t u(t)\|_{W^{2,p}(\R^n)}\,,\| u(t)\|_{W^{3,p}(\R^n)}\big\} \leq  C(\e,u_0,p,t_0)\,;
\end{eqnarray}

\medskip

\noindent {\bf (ii):} $t\mapsto\lambda_{\e}[u(t)]$ is Lispchitz continuous on $(t_0,\infty)$ for every $t_0>0$, and is such that
\[
\sup_{t>0}\,\e\,|\l_\e[u(t)]|\le C\,\AC(u_0)^{2n+2}\,;
\]

\medskip

\noindent {\bf (iii):} $\V(u(t))=1$ for every $t\ge 0$;

\medskip

\noindent {\bf (iv)}: $t\mapsto \AC(u(t))$ is continuous and decreasing on $[0,\infty)$ with
\begin{equation}\label{entropy derivative}
\AC(u(t_2))-\AC(u(t_1)) = -\e\,\int_{t_1}^{t_2}\,dt\int_{\R^n} (\pa_tu(t))^2\,,\qquad\forall t_2\ge t_1>0\,;
\end{equation}

\medskip		

\noindent {\bf (v):} $0<u<1$ on $\R^n\times(0,\infty)$;

\medskip

\noindent {\bf (vi):} the function $t\mapsto \int_{\R^n}(\pa_tu(t))^2$ belongs to $W^{1,1}(t_0,\infty)$ for every $t_0>0$, with
\begin{equation}\label{dissipation formula}
\frac{d}{dt}\, \int_{\R^n} (\pa_tu(t))^2 = - \int_{\R^n} \Big\{4\,|\nabla(\pa_tu)|^2
+\frac{2}{\e}\,\Big( \frac{W''(u)}\e-\lambda_\e[u(t)]\,V''(u)\Big)\,(\pa_tu)^2\Big\}(t)\,.
\end{equation}
In particular,
\[
\lim_{t\to\infty}\int_{\R^n}(\pa_tu(t))^2=0\,.
\]
\end{theorem}

It is well-know that the asymptotic behavior of a semilinear parabolic PDE can be ``subsequentially resolved'' into a bubbling of mutually drifting stationary states: a good exemplification of this kind of result, whose fundamental idea is rooted into Lions' concentration-compactness principle itself, is found, for example, in \cite[Theorem 1.1]{feireisl1997long}. Another powerful idea found in theory of semilinear parabolic PDE is that when the initial datum is compactly supported, then a sort of star-shapedness of the flow (see \eqref{also seen}) can be established by means of the parabolic maximum principle, thus excluding bubbling phenomena; for a general exemplification of this idea, see \cite[Theorem 1.2]{feireisl1997long}. In Section \ref{section bubbling general} we adapt these general methods to our specific problem, which, again, does not follow in the range of application of the general theory because of the non-autonomy of \eqref{diffused VPMCF} and because of the presence of the non-local Lagrange multiplier $\l_\e$. Notice that we are not considering (yet) the physical regime when $\e$ is small: in particular, there is no way to attribute any geometric meaning to the stationary states $\xi_i$ appearing in the statement, e.g. by relating them to the minimizers $\zeta_{\e,m}$ of $\Psi(\e,m)$.

\begin{theorem}[Subsequential bubbling, general $\e$]\label{thm bubbling general}
Let $\e>0$, $n\ge 2$, $W\in C^{2,1}[0,1]$ satisfy \eqref{W basic} and \eqref{W normalization}, $u_0 \in W^{2,p}(\R^n;[0,1])$ for every $p\ge 2$, $\V(u_0)=1$, and let $\{u(t)\}_{t\ge0}$ be the diffused VPMCF with $u(0)=u_0$.

\medskip

Then, for every sequence $t_j\to \infty$ as $j\to \infty$, up to extracting a subsequence, there are $M\in \N$, $\ell_\e>0$, sequences $(x^i_j)_j$ ($i=1,...,M$) satisfying $|x^i_j- x^k_j|\to \infty$ as $j\to\infty$ ($i\ne k$), and radial solutions $\{\xi_i\}_{i=1}^M$ of
\begin{equation}\label{eq PDE for bubbles0}
	2\,\e^2\,\Delta \xi_i = W'(\xi_i)-\e\,\ell_\e\, V'(\xi_i)\quad\mbox{on $\R^n$}\,,
\end{equation}
such that
\begin{equation}\label{eq vlume}
	\sum_{i=1}^M \V(\xi_i)=1\,,\qquad \sum_{i=1}^M\,\AC(\xi_i)\le\AC(u_0)\,,
\end{equation}
and, for all $p> 2$,
\begin{equation*}
	\lim_{j\to\infty}\Big\| u(t_j)- \sum_{i=1}^M \tau_{x^j_i}[\xi_i]\Big\|_{(W^{2,p}\cap W^{1,2})(\R^n)}
+\big|\l_\e[u(t_j)]-\ell_\e\big| =0\,.
\end{equation*}
Moreover, if $u_0$ has compact support in $\R^n$, then $M=1$ and $x_j^1\to x_*$ as $j\to \infty$ for some $x_*\in\R^n$.
\end{theorem}

From this point on we work exclusively in the geometric regime when $\e<\e_0$ for some small universal constant $\e_0$ and relate the subsquential bubbling established in Theorem \ref{thm bubbling general} to the diffused isoperimetric problem $\Psi(\e,m)$. A key tool in achieving this result is the diffused Alexandrov theorem proved in \cite[Theorem 1.1-(iv)]{maggi-restrepo}, whose statement is recalled in Section \ref{section preliminaries}, and which asserts, roughly speaking, that any solution of $2\,\e^2\,\Delta \xi = W'(\xi)-\e\,\ell\, V'(\xi)$ on $\R^n$ with $\xi(x)\to 0$ as $|x|\to\infty$ and $\e\,\ell\in(0,\nu_0)$ for a sufficiently small universal constant $\nu_0$, must satisfy, up to translations, $\xi=\zeta_{\e,m}$ and $\ell=\Lambda_{\e,m}$, for some $m$ such that $\e\in(0,\e_0\,m^{1/n})$.

\begin{theorem}[Subsequential bubbling, small $\e$]\label{theorem bubbling one}
If $n\ge 2$, $W\in C^{2,1}[0,1]$ satisfies \eqref{W basic} and \eqref{W normalization}, $u_0\in W^{2,p}(\R^n;[0,1])$ for all $p\ge 2$, and $\V(u_0)=1$, then there is $\e_0^*=\e_0^*(n,W,u_0)$ with the following property.

\medskip

If $\e\in(0,\e_0^*)$ and $\{u(t)\}_{t\ge0}$ is the corresponding solution of \eqref{diffused VPMCF} with $u(0)=u_0$, then, for every $t_j\to\infty$ as $j\to\infty$, up to extracting a subsequence in $j$, there are $M \in \N$ and sequences $\{(x_j^i)_j\}_{i=1}^M$ with $|x_j^i-x_j^k|\to\infty$ as $j\to\infty$ ($i\ne k$), such that
\begin{equation}\label{cmmm}
\lim_{j\to\infty}\Big\| u(t_j)-  \sum_{i=1}^M \tau_{x_j^i}\big[\zeta_{\e,1/M}\big]\Big\|_{(W^{2,p}\cap W^{1,2})(\R^n)}
+\Big|\l_\e[u(t_j)]-\Lambda_{\e,1/M}\Big|=0\,,
\end{equation}
for all $p>2$. Moreover, $M$ is uniquely characterized by the relation
\begin{equation}
  \label{uniquely M}
M\,\Psi\Big(\e,\frac1{M}\Big)=\lim_{t\to\infty}\AC(u(t))\,.
\end{equation}
In particular, $M$ and the limit Lagrange multiplier $\Lambda_{\e,1/M}$ depend only on $u_0$ and not on the specific sequence $(t_j)_j$ under consideration.
\end{theorem}

\begin{remark}
  {\rm For the quantification of $\e_0^*$ in terms of $u_0$, see \eqref{e0star} below.}
\end{remark}

Under the assumption \eqref{hp double energy} or \eqref{hp compact spt}, Theorem \ref{theorem bubbling one} holds with $M=1$. Thus to go from Theorem \ref{theorem bubbling one} to Theorem \ref{theorem convergence without bubbling} we have to upgrade subsequential convergence to full convergence as $t\to\infty$. The natural approach to this problem consists in proving the differential inequality
\[
-\frac{d}{dt}\,\I_\e(t)\ge \frac{\I_\e(t)}{C(\e)}\,,\qquad\forall t>1/C(\e,u_0)\,,
\]
for the Fisher information/dissipation $\I_\e(t)=\e\,\int_{\R^n}(\pa_t u(t))^2$ of the flow (compare with Theorem \ref{thm existence and regularity}-(iv,vi)). This is the well-known Bakry-\'Emery method, which requires establishing the strict stability of the $\V$-constrained second variation
\begin{equation}
  \label{qqe}
  \QQ_\e[\zeta_\e](\vphi)=\int_{\R^n}2\,\e\,|\nabla\vphi|^2
  +\Big\{\frac{W''(\zeta_\e)}\e-\Lambda_\e\,V''(\zeta_\e)\Big\}\,\vphi^2\,,
  \qquad\vphi\in W^{1,2}(\R^n)\,,
\end{equation}
of the Allen--Cahn functional at the minimizer $\zeta_\e=\zeta_{\e,1}$ of $\Psi(\e,1)$ (where we are also setting $\Lambda_\e=\Lambda_{\e,1}$ for brevity). This strict stability result is established in the following theorem, proved in Section \ref{section stability second variation} (whereas the proof of Theorem \ref{theorem bubbling one} is finally discussed in Section \ref{section convergence without bubbling}).

\begin{theorem}[Strict stability of $\QQ_\e$ at $\zeta_\e$]\label{theorem eps stability}
If $n\ge 2$ and $W\in C^{2,1}[0,1]$ satisfies \eqref{W basic} and \eqref{W normalization}, then there exists a positive universal constant $\e_0$ with the following property.

\medskip

For every $\e\in(0,\e_0)$, if $\vphi\in W^{1,2}(\R^n)$ satisfies
\begin{eqnarray}
\label{stability scnd var ortho conditions}
\int_{\R^n}\varphi\,V'(\zeta_\e)=0\,,\qquad \int_{\R^n}\vphi\,\nabla\zeta_\e=0\,,
\end{eqnarray}
then
\begin{equation}
\label{stability scnd var vphi only}
\QQ_\e[\zeta_\e](\vphi)\geq \frac{\e}C\,\int_{\R^n}\varphi^2\,;
\end{equation}
and, moreover, for a constant $C(\e)>0$,
\begin{equation}
  \label{stability scnd var W12}
  \QQ_\e[\zeta_\e](\vphi)\geq \frac1{C(\e)}\,\int_{\R^n}|\nabla\vphi|^2+\varphi^2\,.
\end{equation}
\end{theorem}

It had already been proved in \cite[Lemma 4.4]{maggi-restrepo} that if $\vphi$ is {\it radial} (and satisfies \eqref{stability scnd var ortho conditions}), then
\begin{equation}
  \label{sssss}
  C\,\QQ_\e[\zeta_\e](\vphi)\geq \int_{\R^n} \e\,|\nabla\vphi|^2+ \frac{\varphi^2}\e\,,
\end{equation}
thus showing that the stability of $\QQ_\e[\zeta_\e]$ among radial variations is much stronger than among generic variations (the costant $\e/C$ in \eqref{stability scnd var vphi only} is sharp as seen in the proof of the theorem itself). The proof of Theorem \ref{theorem eps stability} combines the decomposition of $\vphi$ as a Fourier series in angular/radial variables with a geometric change of variables that allows to relate (on special modes of such decomposition) $\QQ_\e[\zeta_\e]$ with the second variation of the perimeter functional at a ball.

\subsection{Organization of the paper}\label{section organization} Section \ref{section preliminaries} contains a recap of our basic notation, main results from \cite{maggi-restrepo}, and useful properties of the various potentials involved in our analysis. The other sections of the paper are organized as described in the previous section.

\medskip

\noindent {\bf Acknowledgements:} FM has been supported by NSF Grant DMS-2247544. FM and DR have been supported by NSF Grant DMS-2000034 and NSF FRG Grant DMS-1854344.

\section{Background material on the diffused isoperimetric problem}\label{section preliminaries}  In this section we collect some background material that will be repeatedly used in the sequel. In Section \ref{section conv} we set some basic conventions and notation, while in Section \ref{section recap} we recall some key results contained in \cite{maggi-restrepo} and concerning the diffused isoperimetric problem $\Psi(\e,m)$. Finally, in Section \ref{section inequalities} we collect some elementary inequalities concerning the various potentials $W$, $\Phi$, and $V$ that will be often referred to in later proofs, and in Section \ref{section Vdelta} we introduced some regularized volume potentials $V_\de$ which will be employed as a technical device in the proof of Theorem \ref{thm existence and regularity}.

\subsection{Basic conventions}\label{section conv} Throughout the paper $n\in\N$, $n\ge 2$, denotes the dimension of the Euclidean space we are working in, and $W\in C^{2,1}[0,1]$ is a ``double-well potential'' satisfying \eqref{W basic} and \eqref{W normalization}. By {\bf universal constant} we mean a positive real number depending on $W$ and the value of $n$ under consideration. We denote by $C$ a generic universal constant whose value may increase at each of its subsequent appearances. Universal constants that may also depend on $\e>0$ are denoted by $C(\e)$, with the idea that $C(\e)$ may diverge to $+\infty$ as $\e\to 0^+$. Given $k\in\N$, we will write ``$f(\e)={\rm O}(\e^k)$ as $\e\to 0^+$'' if there exists a universal constant $C$ such that $|f(\e)|\le C\,\e^k$ for every $\e\in(0,1/C)$; similar definitions are given for ``${\rm O}(t)$ as $t\to\infty$'', etc.

\subsection{The diffused isoperimetric problem}\label{section recap} Let us recall that for every $\e,m>0$ we set
\[
\Psi(\e,m)=\inf \Big\{ \AC(u):  \V(u)=m\,, u \in W^{1,2}(\R^n;[0,1])\Big\}\,.
\]
Among the basic properties of $\Psi(\e,m)$ we have the scaling law
\begin{equation}
\label{Psi scaling}
\Psi(\e,m)=m^{(n-1)/n}\,\Psi\Big(\frac{\e}{m^{1/n}},1\Big)\,,\qquad\forall\e,m>0\,,
\end{equation}
and the isoperimetric lower bound and $\e\to 0^+$ limit
\begin{eqnarray}
\label{Psi isop lower bound}
\Psi(\e,m)\!\!&>&\!\!2\,c_{\rm iso}(n)\,m^{(n-1)/n}\,,\qquad\forall\e,m>0\,,
\\
\label{Psi limit as eps to zero}
\lim_{\e\to 0^+}\Psi(\e,m)\!\!&=&\!\!2\,c_{\rm iso}(n)\,m^{(n-1)/n}\,,\qquad\forall m>0\,,
\end{eqnarray}
where we have set $c_{\rm iso}(n)=n\,\om_n^{1/n}$. A simple concentration compactness argument (see \cite[Proof of Theorem A.1, steps one and two]{maggi2023hierarchy}) shows that $\Psi(\e,m)$ admits radially symmetric decreasing minimizers for each $\e$ and $m$. Stronger properties are proved in \cite[Theorem 1.1, Theorem 6.1]{maggi-restrepo} under the geometric regime $\e<\!\!<m^{1/n}$:

\medskip

\noindent {\bf Theorem} ${\bf \Psi}$ \cite{maggi-restrepo} {\it If $n\ge 2$ and $W\in C^{2,1}[0,1]$ satisfies \eqref{W basic} and \eqref{W normalization}, then there are positive universal constants $\e_0$ and $\nu_0$ with the following properties:}

\medskip

\noindent ${\bf (i):}$ {\it we have}
\begin{eqnarray}
\label{Psi strict increasing in eps}
&&\mbox{$\Psi(\cdot,m)$ is strictly increasing on $(0,\e_0\,m^{1/n})$}\,,
\\
\label{Psi strict concave in m}
&&\mbox{$\Psi(\e,\cdot)$ is concave on $(0,\infty)$ and strictly concave on $((\e/\e_0)^{n},\infty)$}\,;
\end{eqnarray}
{\it moreover, if $0<\e<\e_0\,m^{1/n}$, then there is a radially symmetric strictly decreasing minimizer $\zeta_{\e,m}$ of $\Psi(\e,m)$ with maximum at the origin having the property that $u$ is a minimizer of $\Psi(\e,m)$ if and only if $u=\tau_{x_0}[\zeta_{\e,m}]$ for some $x_0\in\R^n$; and, for some $\Lambda_{\e,m}>0$, $\zeta_{\e,m}$ satisfies}
\[
2\,\e^2\,\Delta \zeta_{\e,m}=W'(\zeta_{\e,m})-\e\,\Lambda_{\e,m}\,V'(\zeta_{\e,m})\qquad\mbox{on $\R^n$}\,,
\]
{\it where}
\begin{equation}
  \label{limit of Lambda eps}
  \lim_{\e\to 0^+}\,m^{1/n}\,\Lambda_{\e,m}=2\,(n-1)\,\om_n^{1/n}\,.
\end{equation}

\medskip

\noindent {\bf (ii):} {\it if $v\in C^2(\R^n;[0,1])$ satisfies $v(x)\to 0^+$ as $|x|\to\infty$ and solves}
\[
2\,\e^2\,\Delta v=W'(v)-\e\,\l\,V'(v)\qquad\mbox{on $\R^n$}\,,
\]
{\it for some positive $\l$ such that $\e\,\l<\nu_0$, then there exist $x_0\in\R^n$ and $m>0$ such that}
\[
v=\tau_{x_0}[\zeta_{\e,m}]\,,\qquad \e<\e_0\,m^{1/n}\,,\qquad \l=\Lambda_{\e,m}\,.
\]

\subsection{Technical properties of $W$ and related potentials}\label{section inequalities} In the linearization of \eqref{diffused VPMCF} we shall tacitly use the fact that, being $W\in C^{2,1}[0,1]$, we have
\[
\Big|W(r)-W(s)-W'(s)(r-s)-W''(s)\frac{(r-s)^2}2\Big|\le C\,|r-s|^3\,,\qquad\forall r,s\in[0,1]\,.
\]
We also observe that $1/C\le W''\le C$ on $[0,1]$, and that there is a universal constant $\de_0<1/2$ such that
\begin{equation}
	\label{W near the wells}
	\begin{split}
		&\frac{1}C\le \frac{W(r)}{r^2}\,,\frac{W'(r)}r\leq C\quad\hspace{1.5cm}\forall r\in(0,\de_0]\,,
		\\
		&\frac{1}C\le\frac{W(r)}{(1-r)^2}\,, \frac{-W'(r)}{1-r}\leq C\quad\hspace{0.4cm}\forall r\in [1-\de_0,1)\,.
	\end{split}
\end{equation}
Recalling that $\Phi(r)=\int_0^r\sqrt{W}$ and $V(r)=\Phi(r)^{n/(n-1)}$ for $r\in[0,1]$, we use \eqref{W near the wells} to quantify the behaviors of $\Phi$ and $V$ near $r=0$ and $r=1$. By \eqref{W basic}, $\Phi\in C^3_{{\rm loc}}(0,1)$, with
\[
\Phi'=\sqrt{W}\,,\quad\Phi''=\frac{W'}{2\,\sqrt{W}}\,,\quad\Phi'''=\frac{W''}{2\,\sqrt{W}}-\frac{(W')^2}{4\,W^{3/2}}\,,\qquad\mbox{on $(0,1)$}\,.
\]
By \eqref{W near the wells} and \eqref{W normalization} we thus see that $\Phi$ satisfies
\begin{equation}
	\label{Phi near the wells}
	\begin{split}
		&\frac{1}C\le \frac{\Phi(r)}{r^2}\,,\frac{\Phi'(r)}r\,,\Phi''(r)\leq C\,,\qquad\forall r\in(0,\de_0]\,,
		\\
		&\frac{1}C\le \frac{1-\Phi(r)}{(1-r)^2}\,,\frac{\Phi'(r)}{1-r}\,,-\Phi''(r)\le C\,,\qquad\forall r\in[1-\de_0,1)\,.
	\end{split}
\end{equation}
By exploiting \eqref{Phi near the wells} and setting for brevity $a=W''(0)$, we see that, as $r\to 0^+$,
\begin{eqnarray*}
	\Phi'''&=&\frac{2\,W''\,W-(W')^2}{4\,W^{3/2}}
	=\frac{2\,(a+{\rm O}(r))\,(a\,(r^2/2)+{\rm O}(r^3))-(a\,r+{\rm O}(r^2))^2}{4\,(a\,(r^2/2)+{\rm O}(r^3))^{3/2}}
	\\
	&=&\frac{{\rm O}(r^3)}{4\,a^{3/2}\,r^3+{\rm o}(r^3)}\,,
\end{eqnarray*}
and by a similar computation for $r\to 1^-$, we find
\begin{equation}
	\label{Phi third}
	|\Phi'''|\le C\quad\mbox{on $(0,\de_0)\cup(1-\de_0,1)$}\,.
\end{equation}
By \eqref{Phi near the wells} and \eqref{Phi third} we see that $\Phi\in C^{2,1}[0,1]$ with a universal estimate on its $C^{2,1}[0,1]$-norm: in particular,
\begin{equation}\label{Phi second order taylor}
	\Big|\Phi(r)-\Phi(s)-\Phi'(s)(r-s)-\Phi''(s)\frac{(r-s)^2}2\Big|\le C\,|r-s|^3\,,\qquad\forall r,s\in(0,1)\,.
\end{equation}
Since $V=\Phi^{1+\a}$ for $\a=1/(n-1)\in(0,1]$ (recall that $n\ge 2$) and $\Phi(r)=0$ if and only if $r=0$, we easily see that $V\in C^3_{{\rm loc}}(0,1)$, with
\begin{eqnarray*}
	&&V'=(1+\a)\,\Phi^\a\,\Phi'\,,\qquad V''=(1+\a)\Big\{\a\,\frac{(\Phi')^2}{\Phi^{1-\a}}+\Phi^\a\,\Phi''\Big\}\,,
	\\
	&&|V'''|\le C(\a)\,\Big\{\frac{(\Phi')^3}{\Phi^{2-\a}}+\frac{\Phi'\,|\Phi''|}{\Phi^{1-\a}}+\Phi^\a\,|\Phi'''|\Big\}\,.
\end{eqnarray*}
By \eqref{Phi second order taylor}, and keeping track of the sign of $\Phi''$ and of the fact that negative powers of $\Phi(r)$ are large only near $r=0$, but are bounded near $r=1$, we find that
\begin{equation}
	\label{V near the wells}
	\begin{split}
		&\frac1C\le \frac{V(r)}{r^{2+2\a}}\,,\frac{V'(r)}{r^{1+2\a}}\,,\frac{V''(r)}{r^{2\a}}\leq C\,,\hspace{0.7cm}\quad |V'''(r)|\le\frac{C}{r^{1-2\a}}\qquad\forall r\in(0,\de_0]\,,
		\\
		&\frac1C\le  \frac{1-V(r)}{(1-r)^2}\,,\frac{V'(r)}{1-r}\le C\,,\quad |V''(r)|\,,|V'''(r)|\le C\,,\qquad\forall r\in[1-\de_0,1)\,.
	\end{split}
\end{equation}
In particular, $V''(r)\to\infty$ explodes as $r\to 0^+$. However, $V\in C^{2,\g(n)}[0,1]$ (with $\g(n)=\min\{1,2/(n-1)\}\in(0,1]$) and we have the second order Taylor expansion
\[
\Big|V(r)-V(s)-V'(s)\,(r-s)-V''(s)\frac{(r-s)^2}2\Big|\le C\,|r-s|^{2+\g(n)}\,,\qquad\forall r,s\in(0,1)\,.
\]
As much as the analogous expansion for $W$, this formula will be repeatedly used in linearizing \eqref{diffused VPMCF}.

\subsection{Regularized volume potentials}\label{section Vdelta} As already discussed at the beginning of Section \ref{section main theorems}, the non-Lipschitzianity of $V''$ near $0$ causes several technical problems, that call for the introduction of regularized volume potentials $V_\de:[0,1]\to[0,1]$ ($\de>0$) such that $V_\delta \in C^{2,1}[0,1]$, with
\[
V_\delta(0)=0\,,\qquad V_\delta'(0)= V_\delta'(1)=0\,,\qquad\lim_{\de\to 0^+}\|V_\de-V\|_{C^2[0,1]}=0\,.
\]
We will of course have
\[
\lim_{\de\to 0^+}\Lip(V_\de'';[0,1])=+\infty\,,\qquad \sup_{\de>0}\,[V_\de'']_{C^{0,\g(n)}}<\infty\,.
\]
Such potentials $V_\de$ can be defined by first considering a family $\{\rho_\de\}_{\de>0}$ of smooth mollifiers on $\R$ such that $\spt\rho_\de\cc (-\de^2,\de^2)$, and then by setting
\begin{equation}\label{eq delta approximation}
V_\delta=\rho_{\delta}\star\Big(L_\delta\circ\Phi\Big)^{n/(n-1)}\,,
\end{equation}
where
\begin{equation*}
	L_\delta(r)=
\left\{
\begin{split}
  &0\,,\qquad\hspace{1cm} r\in[0,\de]\,,
  \\
  &\frac{r-\de}{1-2\,\de}\,,\qquad r\in[\de,1-\de]\,,
  \\
  &1\,,\qquad\hspace{1cm} r\in[1-\de,1]\, .
\end{split}
\right .
\end{equation*}
By $W>0$ on $(0,1)$, \eqref{W near the wells} and \eqref{V near the wells}, and up to further decreasing the value of $\delta_0$ introduced above, we have
\[
V_\delta(r)\le C\,r^2\le C\, W(r)\,, \qquad V_\delta'(r)\leq C\, r\,,\qquad\forall r\in(0,1-\de_0)\,,
\]
for every $\delta \in \big[0,\delta_0]$, as well as (compare with \eqref{V near the wells})
\begin{equation}
	\label{V limitata dal basso via da zero}
	V_\delta(r)\ge\frac1C\,,\qquad V_\delta'(r)\leq C\,(1-r)\, \qquad\forall r\in(\de_0,1)\,.
\end{equation}

\section{Estimates for the Lagrange multiplier functional}\label{section estimates for lambda} This section is devoted to the analysis of the Lagrange multiplier functional $\l_\e$, defined with values in $[0,\infty)$ on any given function $v\in W^{1,2}(\R^n;[0,1])$ with $|\{0<v<1\}|>0$ (and assumption that guarantees $\int_{\R^n}V'(u)^2>0$) by setting
\begin{equation}
\label{lambda eps u proof}
\lambda_\e[v]
=  \frac{\int_{\R^n}2\,\e^2\, |\nabla v|^2\,V''(v) +  W'(v)\,V'(v)}{\e\,\int_{\R^n}V'(v)^2}\,.
\end{equation}
In particular, we address the Lipschitz continuity properties of $\l_\e$ in the Banach space  $(X,\|\cdot\|)$ defined by
\[
X=(C^0\cap W^{1,2})(\R^n; [0,1])\,,\qquad \|u\|_{X}= \|u\|_{W^{1,2}(\R^n)}+\|u\|_{C^0(\R^n)}\,,
\]
that is the space we shall use to construct mild solutions of \eqref{diffused VPMCF}. In fact, we shall also need to consider the Lagrange multiplier functionals
\begin{equation}
  \label{lambda eps delta u proof}
\lambda_{\e,\de}[u]=\frac{\int_{\R^n}2\,\e^2\,|\nabla u|^2\,V_\de''(u) + W'(u)\,V_\de'(u)}{\e\,\int_{\R^n}V_\de'(u)^2}\,,
\end{equation}
obtained by replacing $V$ with the regularized volume potentials $V_\de$ introduced in the previous section. We shall also set $V_0=V$, $\V_0=\V$, and use the notation
\[
\ac_\e(u;\Om)=\int_\Om\e\,|\nabla u|^2+\frac{W(u)}\e\,,\qquad \V_\de(u;\Om)=\int_\Om V_\de(u)\,,
\]
for the localization to a Borel set $\Om\subset\R^n$ of the functionals $\AC$ and $\V_\de$.

\begin{theorem}
If $n\ge 2$ and $W\in C^{2,1}[0,1]$ satisfies \eqref{W basic} and \eqref{W normalization}, then there exist  positive universal constants $\e_0$ and $\de_0$ with the following properties:

\medskip

\noindent {\bf (i):} if $u\in W^{1,2}(\R^n;[0,1])$, $\e>0$, and $\de\in[0,\de_0]$, then
\begin{equation}
  \label{eq l2 bound}
\int_{\R^n}u^2\le C\,\big\{\e\, \AC(u)+ \V_\delta(u)\big\}\,,
\end{equation}
and, with an $\e$-dependent universal constant $C(\e)$,
\begin{eqnarray}
\label{vdelta prime squared below}
C(\e)\,\int_{\R^n}|\nabla u|^2\,\int_{\R^n}V_\de'(u)^2&\ge&\min\Big\{1,\frac{\V_\de(u)}{\AC(u)}\Big\}^{2\,n}\,,
\\
\label{eq upper bound lambda}
|\l_{\e,\de}[u]|&\leq&C(\e)\,\frac{\AC(u)^{2n+1}}{\V_\delta(u)^{2n}}\,\int_{\R^n}|\nabla u|^2\,.
\end{eqnarray}
Furthermore, if $\delta=0$ and $\e\in(0,\e_0)$, then \eqref{vdelta prime squared below} and \eqref{eq upper bound lambda} hold with $C$ in place of $C(\e)$.

\medskip

\noindent {\bf (ii):} if $u, v \in X \setminus \{0\}$, $\e>0$, and $\max\{\|u\|_{X}\,,\|v\|_X\}\leq C(\e)$, then
\begin{eqnarray}
\label{eq Lip lambda}
|\l_{\e,\de}[u]-\l_{\e,\de}[v]|\!\!&\leq&\!\!C(\e,\delta)\, \frac{\| u-v\|_{X}}
{\min\{1,\V_\de(u)^{2n}\}\,\min\{1,\V_\de(v)^{2n}\}}\,,\quad\forall\de\in(0,\de_0]\,,\,\,\,
\\
\label{eq holder}
|\l_{\e,\de}[u]-\l_{\e,\de}[v]|\!\!&\leq&\!\!C(\e)\, \frac{\| u-v\|_{X}^{\gamma(n)}}{\min\{1,\V_\de(u)^{2n}\}\,\min\{1,\V_\de(v)^{2n}\}}\,,\qquad\forall\de\in[0,\de_0]\,,
\\
\label{eq holder 2}
|\l_{\e,\de}[u]-\l_{\e}[u]|\!\!&\leq&\!\!C(\e)\, \frac{\|V_\de-V\|_{C^2[0,1]}}{\min\{1,\V_\de(u)^{2n}\}\,\min\{1,\V(u)^{2n}\}}\,,\qquad\forall\de\in[0,\de_0]\,,
\end{eqnarray}
where $\gamma(n)= \min\{1, 2/(n-1)\}$; and if $u, v \in W^{2,2}(\R^n;[0,1])\setminus\{0\}$, then
\begin{equation}\label{eq lipschitz}
	|\l_{\e,\de}[u]-\l_{\e,\de}[v]|\leq
C(\e) \frac{\max\{\|u\|_{W^{2,2}},\|v\|_{W^{2,2}}\}\,\| u-v\|_{W^{1,2}(\R^n)}}{\min\{1,\V_\de(u)^{2n}\}\,\min\{1,\V_\de(v)^{2n}\}}\,,\qquad\forall\de\in[0,\de_0]\,.
\end{equation}
\end{theorem}

\begin{proof} {\it Step one, diffused relative isoperimetry and nucleation}: We prove two relative isoperimetric inequalities in balls  in the diffused setting and a consequent nucleation type lemma modeled after \cite[Lemma 2.1]{morini-ponsiglione-spadaro}. This kind of result is in turn inspired by a tool introduced by Almgren \cite{Almgren76} in the study of isoperimetric clusters, see \cite[Lemma 29.10]{maggiBOOK}. More precisely, we prove the existence of universal constants $\eta_0$, $\s_0$ and $C$ with the following properties:

\medskip

\noindent {\bf (a):} if $\de\in[0,\de_0]$, $\e,r>0$, and if $u\in W^{1,2}(\R^n;[0,1])$ satisfies
\begin{equation}\label{eq initial volume bound}
\fint_{B_r} V_\delta(u)\leq \eta_0\,\Big(\frac{\e}{r}\Big)^{2\,n},
\end{equation}
then
\begin{equation}\label{eq sobolevballs}
C\,\AC(u;B_r) \geq  \frac{\e}{r}\,\V_\delta(u;B_r)^{(n-1)/n}\,;
\end{equation}

\medskip

\noindent {\bf (b):} if $\e,r>0$ are such that $\e/r<\s_0$, and if $u\in W^{1,2}(\R^n;[0,1])$ satisfies
\begin{equation}\label{eq initial volume bound zero}
\fint_{B_r} V(u)\leq \frac12\,,
\end{equation}
then
\begin{equation}\label{eq sobolevballs zero}
C\,\AC(u;B_r)\geq \V(u;B_r)^{(n-1)/n}\,.
\end{equation}

\medskip

\noindent {\bf (c):} if $\e>0$ and
\[
\B=\Big\{B_{\sqrt{n+1}\,R/2}(R\,z):z\in\Z^n\Big\}\,,\qquad R=\frac{\max\{1,\e\}}{\s_0\,(\sqrt{n+1}/2)}\,,
\]
then for every $u\in W^{1,2}(\R^n;[0,1])$ we have
\begin{eqnarray}
\label{non vanishing estimate}
C\,\frac{\max\{1,\e\}}{\min\{1,\e^{2\,n}\}}\,\sup_{B\in \,\mathcal{B}} \fint_{B} V_\delta(u)
&\geq&
\min\Big\{1,\Big(\frac{\V_\delta(u)}{\AC(u)}\Big)^{n}\Big\}\,,\qquad\forall \de\in(0,\de_0]\,,
\\
\label{non vanishing estimate 2}
C\,\max\{1,\e\}\,\sup_{B\in\, \mathcal{B}} \fint_{B} V(u)
&\geq&
\min\Big\{1,\Big(\frac{\V(u)}{\AC(u)}\Big)^{n}\Big\}\,.
\end{eqnarray}

\medskip

\noindent {\it We first derive conclusion (c) from conclusions (a) and (b)}: Let $u\in W^{1,2}(\R^n;[0,1])$. In proving \eqref{non vanishing estimate 2} we can assume without loss of generality that
\[
\sup_{B\in\B}\fint_B V(u)\le\frac12\,.
\]
In particular, since the choice of $R$ is such that $\e/(\sqrt{n+1}R/2)<\s_0$ for every $\e>0$, we deduce from conclusion (b) that \eqref{eq sobolevballs zero} holds for every $B\in\B$. The corresponding bounds can be used together with the fact that $\B$ is a covering of $\R^n$ with finite overlapping (depending only on the dimension $n$) to conclude that
\begin{eqnarray}\label{as in}
C(n)\,\AC(u)&\ge&\sum_{B\in\B}\AC(u;B)\ge\frac1C\, \sum_{B\in\B}\V(u;B)^{(n-1)/n}
\\\nonumber
&\ge& \frac1C\, \frac{\sum_{B\in\B}\V(u;B)}{\sup_{B\in\B}\V(u;B)^{1/n}}
\ge\frac1C\, \frac{\V(u)}{\sup_{B\in\B}\Big(\int_B V(u)\Big)^{1/n}}\,.
\end{eqnarray}
Thanks to $|B|^{1/n}\le C\,\max\{1,\e\}$ for every $B\in\B$, \eqref{as in} implies \eqref{non vanishing estimate 2}. To prove \eqref{non vanishing estimate}, let $\de\in(0,\de_0]$, and let us assume without loss of generality that
\begin{equation}
  \label{hold}
  \sup_{B\in\B}\fint_B V_\de(u)\le\eta_0\,\Big(\frac{\e}{\sqrt{n+1}R/2}\Big)^{2\,n}\,,
\end{equation}
then we can apply \eqref{eq sobolevballs} to each $B\in\B$, and conclude as in \eqref{as in} that
\[
C(n)\,\AC(u)\ge\frac{\e}{C\,R}\, \frac{\V_\de(u)}{\sup_{B\in\B}\V_\de(u;B)^{1/n}}\ge\frac{\min\{1,\e\}}C\,\frac{\V_\de(u)}{\sup_{B\in\B}\V_\de(u;B)^{1/n}}\,.
\]
In particular, \eqref{non vanishing estimate} follows by taking again into account that $\e/R=C\,\max\{1,\e\}$.

\medskip

\noindent {\it We now prove conclusions (a) and (b)}: Up to a rescaling, we can take $r=1$ in both conclusions. To prove conclusion (a), we notice that,
%
by \eqref{eq initial volume bound} and \eqref{V limitata dal basso via da zero},
\begin{equation}\label{eq upper bound level set}
\eta_0\,\e^{2\,n} \geq 	\frac1{\om_n}\,\int_{B_1 \cap \{u \geq 1/2\}} V_\delta(u)\geq \frac{|B_1 \cap \{u \geq 1/2\}|}{C}\,.
\end{equation}
Since $W(r)\ge r^2/C$ for $r\in[0,1/2]$, by \eqref{V limitata dal basso via da zero}, \eqref{eq upper bound level set}, and the H\"older inequality, we get
\begin{eqnarray}\notag
\int_{B_1} u^2 &=& \int_{B_1\cap\{u\leq 1/2\}} u^2 +\int_{B_1\cap \{u\geq  1/2\}} u^2\\\notag
&\leq& C\int_{B_1} W(u) + C\,\big|B_1\cap \{u\geq  1/2\}\big|^{1/n}\,\Big(\int_{B_1} u^{2 n/(n-1)} \Big)^{(n-1)/n}\\\label{eq upper bound l2 norm}
&\leq& C\int_{B_1} W(u) + C\,\eta_0^{1/n}\,\e^2\,\Big(\int_{B_1} u^{2 n/(n-1)} \Big)^{(n-1)/n}\,.
\end{eqnarray}
By combining \eqref{eq upper bound l2 norm} with the embedding of $L^{2n/(n-1)}(B_1)$ into $W^{1,2}(B_1)$, we find that
\begin{eqnarray*}
	\int_{B_{1}} \e\, |\nabla u|^2 + \frac{W(u)}{\e} &\geq& \int_{B_1} \e\,
|\nabla u|^2+\frac{u^2}{C \e}
-C\,\eta_0^{1/n}\,\e\,\Big(\int_{B_1} u^{2 n/(n-1)} \Big)^{(n-1)/n}
\\
&\geq& \frac{\e}C\,\Big(\int_{B_1} u^{2 n/(n-1)} \Big)^{(n-1)/n}-
C\,\eta_0^{1/n}\,\e\,\Big(\int_{B_1} u^{2 n/(n-1)} \Big)^{(n-1)/n}
\\
	&\geq& \frac{\e}C\,\Big(\int_{B_1} u^{2 n/(n-1)} \Big)^{(n-1)/n}\ge \frac{\e}C\,\Big(\int_{B_1}V_\de(u)\Big)^{(n-1)/n}\,,
\end{eqnarray*}
provided $\eta_0$ is a sufficiently small universal constant, and where we have used that $V_\delta(r) \leq C\, r^{2n/(n-1)}$ for every $r\in[0,1]$ and $\delta\in[0,\de_0]$. Having proved conclusion (a), we now prove conclusion (b). Arguing by contradiction, we can assume the existence of $\e_k\to 0^+$ and $\{u_k\}_k$ in $W^{1,2}(B_1;[0,1])$ such that, for all $k\in\N$ and setting for brevity
\[
M_k=\int_{B_1}V(u_k)=\int_{B_1}\Phi(u_k)^p\,,\qquad p=\frac{n}{n-1}\,,
\]
we have $M_k\le \om_n/2$ and
\begin{eqnarray}
\label{eq contradictionballs}
\int_{B_{1}} \e_k |\nabla u_k|^2 + \frac{W(u_k)}{\e_k}\le \frac{M_k^{1/p}}k\,,
\end{eqnarray}
for every $k\in\N$. Combining  $|\nabla(\Phi\circ u)|=|\nabla u|\,\sqrt{W(u)}$ with Young's inequality as in \eqref{modicamortola}, we deduce from \eqref{eq contradictionballs} and the $BV$-Poincar\'e inequality \cite[(3.41)]{AFP} that
\begin{equation}\label{eq seminorm}
\frac{1}{C}\Big(\int_{B_1}|\Phi(u_k)- t_k|^p\Big)^{1/p}\le\int_{B_{1}}|\nabla (\Phi\circ u_k)|\le\frac{C}k\,M_k^{1/p}\,,
\end{equation}
where $t_k=\om_n^{-1}\int_{B_1}\Phi(u_k)$. In particular, there is $c\in[0,1]$ such that, up to extracting a subsequence, $u_k\to c$ in $L^1(B_1)$ and a.e. in $B_1$ as $k\to\infty$. Since \eqref{eq contradictionballs} implies $\int_{B_1}W(u_k)\to 0$ as $k\to\infty$, by Fatou's lemma we find
\[
\int_{B_1} W(c)\leq \liminf_{k\to \infty} \int_{B_1} W(u_k)=0\,.
\]
In particular, $c\in\{0,1\}$. Since $c=1$ would contradict $M_k\le\om_n/2$ for every $k$, we conclude that $c=0$, and hence, thanks also to $0\le u_k\le 1$, that $\Phi(u_k)\to 0$ in $L^p(B_1)$ (i.e., $M_k\to 0$) as $k\to\infty$. On noticing that $0\le t_k\le C\,M_k^{1/p}$, we deduce by \eqref{eq seminorm} that
\begin{eqnarray}\label{eq convergence constant}
\|\Phi(u_k)-(M_k/\om_n)^{1/p}\|_{L^p(B_1)}\leq\|t_k-(M_k/\om_n)^{1/p}\|_{L^p(B_1)}+\frac{C}k\,M_k^{1/p}\le C\,M_k^{1/p}\,,
\end{eqnarray}
and, in particular, that
\[
\lim_{k\to\infty}\Big|B_1\cap\Big\{\Phi(u_k)\le(M_k/2\om_n)^{1/p}\Big\}\Big|=0\,.
\]
Since $\Phi(u_k)\to 0$ in $L^p(B_1)$ implies $|B_1\cap\{\Phi(u_k)>1/2\}|\to 0$ as $k\to\infty$, we conclude that
\begin{equation}\label{eq measuredensity}
\Big|B_1\cap\Big\{\frac12\ge\Phi(u_k)\ge(M_k/2\om_n)^{1/p}\Big\}\Big|\ge\frac{|B_1|}2\,,
\end{equation}
for $k$ large enough. In particular, thanks to $\Phi\le C\, W$ on $[0,1/2]$ and to \eqref{eq contradictionballs}, for $k$ large enough we have
\begin{equation}
\frac{M_k^{1/p}}{C\, \e_k} \leq \frac{1}{\e_k} \int_{B_1\cap\{1/2\ge\Phi(u_k)\ge(M_k/2\om_n)^{1/p}\}}\Phi(u_k)\le
\frac1{\e_k}\int_{B_1} W(u_k) \leq \frac{M_k^{1/p}}k\,,
\end{equation}
that leads to a contradiction as $k\to\infty$.

\medskip

\noindent {\it Step two:} We prove statement (a). To prove \eqref{eq l2 bound} it suffices to recall that $W(r)\geq  r^2/C$ for $r\in [0,1/2]$ and $V_\delta(r)\geq r^2/C$ for $r\in[1/2,1]$, so that
\[
\int_{\R^n}u^2=\int_{\{u\le 1/2\}}u^2+\int_{\{u> 1/2\}}u^2\le C\,\int_{\R^n}W(u)+V_\de(u)\,.
\]
To prove \eqref{eq upper bound lambda} we first notice that $|V_\de''(r)|\le C$ and $|W'(r)\,V'_\delta(r)|\leq C\, W(r)$ for every $r\in[0,1]$ (and every $\de\in[0,\de_0])$, so that
\begin{equation}
  \label{kurt 1}
  \Big|\int_{\R^n}2\e\,|\nabla u|^2\,V_\de''(u) +  \frac{1}{\e}\,W'(u)\,V_\de'(u)\Big|\le C\,\AC(u)\,.
\end{equation}
Since $V_\de(r)\le C\,r^2$ for $r\in[0,1]$ and $u\in L^2(\R^n)$ we can find a sequence $R_j\to\infty$ such that $\int_{\pa B_{R_j}}V_\de(u)\to 0$ as $j\to\infty$, and thus apply the divergence theorem to deduce that, for every $x_0\in\R^n$,
\begin{eqnarray}
\label{eq upper bound V}
(n-1)\, \int_{\R^n} \frac{V_\delta(u)}{|x-x_0|}\,dx&=& \int_{\R^n} V_\delta(u)\,\Div\Big(\frac{x-x_0}{|x-x_0|}\Big)\,dx
\\
\nonumber
&=& - \int_{\R^n} V'_\delta(u)\,\nabla u\cdot \frac{x-x_0}{|x-x_0|}\,dx
\leq \| V'_\delta(u)\|_{L^2(\R^n)}\, \| \nabla u\|_{L^2(\R^n)}\,.
\end{eqnarray}
Setting $R=2\,\max\{1,\e\}/\s_0\,\sqrt{n+1}$, we can apply \eqref{non vanishing estimate} to find $x_0\in\R^n$ such that
\[
C(\e)\,\fint_{B_R(x_0)}V_\de(u)\ge\min\Big\{1,\frac{\V_\de(u)}{\AC(u)}\Big\}^n
\]
which combined with \eqref{eq upper bound V} gives
\begin{equation}
  \label{kurt 2}
  \| V'_\delta(u)\|_{L^2(\R^n)}\, \| \nabla u\|_{L^2(\R^n)}
\ge\frac{(n-1)}{R}\,\int_{B_R(x_0)}V_\de(u)\ge \frac1{C(\e)}\,\min\Big\{1,\frac{\V_\de(u)}{\AC(u)}\Big\}^n\,,
\end{equation}
that is \eqref{vdelta prime squared below}. In summary, by combining \eqref{kurt 1} and \eqref{kurt 2} we find
\begin{eqnarray}
|\l_{\e,\de}[u]|&\le&C\,\frac{\AC(u)}{\int_{\R^n}V_\de'(u)^2}\le C(\e)\,\frac{\AC(u)}{\min\Big\{1,\frac{\V_\de(u)}{\AC(u)}\Big\}^{2n}}\,\int_{\R^n}|\nabla u|^2
\end{eqnarray}
that is \eqref{eq upper bound lambda}. When $\e\in(0,\e_0]$ and $\de=0$ we can replace the constant $C(\e)$ in \eqref{eq upper bound lambda} with a plain universal constant $C$ by exploiting the fact that \eqref{non vanishing estimate 2} can be used in place of \eqref{non vanishing estimate} (notice indeed that $\max\{1,\e\}\le C$ when $\e\in(0,\e_0]$).

\medskip

\noindent {\it Step three}: We prove statement (b). Setting for brevity,
\[
\NN_{\e,\de}[u]=\int_{\R^n}2\,\e\,|\nabla u|^2\,V_\de''(u)+\frac{W'(u)\,V_\de'(u)}\e\,,\qquad\de\in[0,\de_0]\,,
\]
we notice that for every $\de,\de_*\in[0,\de_0]$ and $u,v\in X\setminus\{0\}$,
\begin{equation}\label{eq difference lambda}
\l_{\e,\de}[u]-\l_{\e,\de_*}[v]=
\frac{\NN_{\e,\de}[u]-\NN_{\e,\de_*}[v]}{\int_{\R^n}V_\delta'(u)^2}
+
\frac{\NN_{\e,\de_*}[v]\,\big\{\int_{\R^n}V_{\delta_*}'(v)^2-\int_{\R^n}V_\delta'(u)^2\big\}}{\int_{\R^n}V_\delta'(u)^2 \int_{\R^n}V_{\delta_*}'(v)^2}\,.
\end{equation}

\medskip

{\it We first work on \eqref{eq difference lambda} with $\de=\de_*$}, with the goal of proving \eqref{eq Lip lambda}, \eqref{eq holder}, and \eqref{eq lipschitz}. Since $\AC(u)\le\|u\|_{W^{1,2}}^2/\e$, we deduce from \eqref{vdelta prime squared below} that, if $\|u\|_{W^{1,2}}\le C(\e)$, then
\begin{equation}
  \label{estimetta}
  \frac1{\int_{\R^n}V_\de'(u)^2}\le \frac{C(\e)}{\min\{1,\V_\de(u)^{2n}\}}\,.
\end{equation}
Recalling from \eqref{kurt 1} that $|\NN_{\e,\de}[u]|\le C\,\AC(u)\le C(\e)\,\|u\|_{W^{1,2}}$, and using again first $\Lip(V_\de',[0,1])\le C$, and then $|V_\de'(r)|\le C\,t$ for $r\in[0,1]$, we find that, if $\|u\|_{W^{1,2}}\le C(\e)$, then for every $v\in W^{1,2}(\R^n;[0,1])$
\begin{eqnarray}
  \nonumber
  &&\Big|\NN_{\e,\de}[v]\,\int_{\R^n}\big(V_\delta'(v)^2-V_\delta'(u)^2\big)\Big|
  \le C(\e)\,\int_{\R^n}|V_\de'(u)-V_\de'(v)|\,\big(|V_\de'(u)|+|V_\de'(v)|\big)
  \\\nonumber
  &&
  \le C(\e)\,\|u-v\|_{L^2}\,\Big(\int_{\R^n}|V_\de'(u)|^2+|V_\de'(v)|^2\Big)^{1/2}
  \\
  \label{eq Nbound}
  &&\le C(\e)\,\max\{\|u\|_{L^2},\|v\|_{L^2}\}\,\|u-v\|_{L^2}\,.
\end{eqnarray}
In summary, by \eqref{eq difference lambda}, \eqref{estimetta} and \eqref{eq Nbound} for every $\e>0$ and $u,v\in W^{1,2}(\R^n;[0,1])\setminus\{0\}$ with $\max\{\|u\|_{W^{1,2}(\R^n)},\| v\|_{W^{1,2}(\R^n)}\}\leq C(\e)$ we have proved that
\begin{equation}
  \label{partial}
  \big|\l_{\e,\de}[u]-\l_{\e,\de}[v]\big|
\le C(\e)\,\frac{|\NN_{\e,\de}[u]-\NN_{\e,\de}[v]|}{\min\{1,\V_\de(u)^{2n}\}}
+\frac{C(\e)\,\|u-v\|_{L^2}}{\min\{1,\V_\de(u)^{2n}\}\,\min\{1,\V_\de(v)^{2n}\}}\,.
\end{equation}
We first estimate that
\begin{eqnarray}\nonumber
  |\NN_{\e,\de}[u]-\NN_{\e,\de}[v]|&\le& 2\,\e\,\int_{\R^n}|\nabla(u-v)|\,\big(|\nabla u|+|\nabla v|)\,|V_\de''(u)|
  \\\label{kurt 3}
  &&+2\,\e\,\int_{\R^n}|\nabla v|^2\,|V_\de''(u)-V_\de''(v)|
  \\\nonumber
  &&+\frac1\e\,\int_{\R^n}|W'(u)-W'(v)|\,|V_\de'(u)|+\frac1\e\,\int_{\R^n}|W'(v)|\,|V_\de'(u)-V_\de'(v)|\,.
\end{eqnarray}
By recalling that $\Lip(W';[0,1])\le C$, $\Lip(V_\de';[0,1])\le C$, $\Lip(V_\de'';[0,1])\le C(\de)$ (when $\de>0$), and $\max\{|V_\de'(r)|,|W'(r)|\}\le C\,r$ for $r\in[0,1]$ we find
\begin{eqnarray}\nonumber
&&|\NN_{\e,\de}[u]-\NN_{\e,\de}[v]|
\le
C\,\max\big\{\|\nabla u\|_{L^2},\|\nabla v\|_{L^2}\big\}\,\|\nabla u-\nabla v\|_{L^2}
\\\nonumber
&&+C(\e,\de)\,\|u-v\|_{C^0}\,\int_{\R^n}|\nabla v|^2+\frac{C}\e\,\|u-v\|_{L^2}\,\max\{\|u\|_{L^2},\|v\|_{L^2}\}
\\\label{kurt 4}
&\le& C(\e)\,\max\big\{\|u\|_{W^{1,2}},\|v\|_{W^{1,2}}\big\}\,\|u-v\|_{W^{1,2}}+C(\de)\,\|\nabla v\|_{L^2}^2\,\|u-v\|_{C^0}\,;
\end{eqnarray}
while, using $[V_\de'']_{C^{0,\g(n)}[0,1]}\le C$ (which also holds when $\de=0$) in place of $\Lip(V_\de'';[0,1])\le C(\de)$, we find instead
\begin{eqnarray}\nonumber
&&|\NN_{\e,\de}[u]-\NN_{\e,\de}[v]|
\\\label{kurt 5}
&\le& C(\e)\,\max\big\{\|u\|_{W^{1,2}},\|v\|_{W^{1,2}}\big\}\,\|u-v\|_{W^{1,2}}
+C\,\|\nabla v\|_{L^2}^2\,\|u-v\|_{C^0}^{\g(n)}\,;
\end{eqnarray}
thanks to $\max\{\|u\|_X,\|v\|_X\}\le C(\e)$; with the convention that $C(\e,\de)=+\infty$ if $\de=0$, we thus conclude from \eqref{kurt 3}, \eqref{kurt 4} and \eqref{kurt 5} that
\begin{eqnarray}\label{stima per Neps}
  |\NN_{\e,\de}[u]-\NN_{\e,\de}[v]|\le \min\Big\{C(\e,\de)\,\|u-v\|_{X},C(\e)\|u-v\|_X^{\g(n)}\Big\}\,.
\end{eqnarray}
By means of \eqref{partial} and \eqref{stima per Neps} we find that, if $\max\{\|u\|_X,\|v\|_X\}\le C(\e)$, then \eqref{eq Lip lambda} and \eqref{eq holder} hold. We now assume that $u,v\in W^{2,2}(\R^n;[0,1])\setminus\{0\}$ and prove \eqref{eq lipschitz}. To this end let us first notice that an integration by parts gives
\begin{eqnarray*}
&&\int_{\R^n}|\nabla u|^2\,V_\de''(u)-\int_{\R^n}|\nabla v|^2\,V_\de''(v)
\\
&=&\int_{\R^n}(\nabla u-\nabla v)\cdot\nabla v\,\,V_\de''(v)+\int_{\R^n}\nabla u\cdot\nabla u \,\,V_\de''(u)-\nabla u\cdot\nabla v\,\,V_\de''(v)
\\
&=&\int_{\R^n}(\nabla u-\nabla v)\cdot\nabla v\,V_\de''(v)-\int_{\R^n}(\Delta u)\,\big(V_\de'(u)-V_\de'(v)\big)\,,
\end{eqnarray*}
which can be used to replace \eqref{kurt 3} with
\begin{eqnarray}\label{kurt 6}
  |\NN_{\e,\de}[u]-\NN_{\e,\de}[v]|&\le& 2\,\e\,\int_{\R^n}|\Delta u|\,|V_\de'(u)-V_\de'(v)|+|\nabla (u-v)|\,|\nabla v|\,|V_\de''(v)|
  \\\nonumber
  &&+\frac1\e\,\int_{\R^n}|W'(u)-W'(v)|\,|V_\de'(u)|+\frac1\e\,\int_{\R^n}|W'(v)|\,|V_\de'(u)-V_\de'(v)|\,.
\end{eqnarray}
By $\Lip(W';[0,1])\le C$, $\Lip(V_\de';[0,1])\le C$ and $\max\{|V_\de'(r)|,|W'(r)|\}\le C\,r$ for $r\in[0,1]$ we thus find
\[
|\NN_{\e,\de}[u]-\NN_{\e,\de}[v]|\le C(\e)\,\max\{\|u\|_{W^{2,2}},\|v\|_{W^{1,2}}\}\,\|u-v\|_{W^{1,2}}
\]
which, combined with \eqref{partial} and with the assumption $\max\{\|u\|_{W^{2,2}},\|v\|_{W^{2,2}}\}\le C(\e)$, gives \eqref{eq lipschitz}.

\medskip

{\it We now work on \eqref{eq difference lambda} with $u=v$ and $\de_*=0$ to prove \eqref{eq holder 2}}. To this end, we first notice that
\begin{eqnarray*}
  \big|\NN_{\e,\de}[u]-\NN_{\e}[u]\big|&\le& 2\,\e\int_{\R^n}|\nabla u|^2\,|V_\de''(u)-V''(u)|+\frac1\e\int_{\R^n}|W'(u)|\,|V_\de'(u)-V'(u)|
\\
&\le&2\,\AC(u)\,\|V_\de-V\|_{C^2[0,1]}\,,
\end{eqnarray*}
while $V_\de'(0)=V'(0)=0$ gives $|V_\de'(u)-V'(u)|\le \|V_\de''-V''\|_{C^0[0,1]}\,|u|$ on $\R^n$, and thus, arguing as in \eqref{eq Nbound}, that
\begin{eqnarray*}
  \Big|\int_{\R^n}V_\de'(u)^2-\int_{\R^n}V'(u)^2\Big|\le C\,\|V_\de-V\|_{C^2[0,1]}\,\|u\|_{L^2(\R^n)}^2\,.
\end{eqnarray*}
Combining this last two estimates with \eqref{eq difference lambda}, \eqref{estimetta} and $|\NN_{\e,\de}[u]|\le C(\e)\,\|u\|_X$ we immediately prove \eqref{eq holder 2}.
\end{proof}

\section{Existence, regularity and entropies of the flow (Proof of Theorem \ref{thm existence and regularity})}\label{section proof of theorem 1}

\begin{proof}[Proof of Theorem \ref{thm existence and regularity}]
\noindent {\it Step one, existence of mild solutions}: For every $\e,\de>0$ we introduce the regularized flows
\begin{equation}\tag{DF$_\de$}
\label{eq auxiliary flow}
	\begin{cases}
	\e^2\,\partial_t u=  2 \,\e^2\, \Delta u - W'(u)+ \e\,\l_{\e,\de}[u(t)]\, V_\delta'(u)\,,\qquad\mbox{on $\R^n\times(0,\infty)$}\,,
    \\
	u(0)=u_0\,,
	\end{cases}
\end{equation}
that are obtained by replacing $V$ with $V_\de$ in \eqref{diffused VPMCF}. If we set
\begin{equation}
  \label{gxt}
  G(x,t)=\frac{e^{-|x|^2/8\,t}}{(8\,\pi\,t)^{n/2}}\,,\qquad (x,t)\in\R^n\times(0,\infty)\,,
\end{equation}
and $S_tv=v\star G(t)$ for $v:\R^n\to\R$ and $t>0$, then a solution $u$ of
\begin{equation}
  \label{duhamel equation}
  \mbox{$u_t-2\,\Delta u=f$ on $\R^n\times(0,\infty)$}\,,\qquad  \mbox{$u(0)=u_0$ on $\R^n$}\,,
\end{equation}
with data $f:\R^n\times(0,\infty)\to\R$ and $u_0:\R^n\to\R$, is formally given by the {\it Duhamel formula} (see, for example, \cite[Section 2.3.1]{EvansPDE}),
\begin{eqnarray}\label{duhamel formula}
u(x,t)&=&S_tu_0(x)+\int_0^t\,S_{t-s}[f(s)](x)\,ds
\\\nonumber
&=&\int_{\R^n}G(x-y,t)\,u_0(y)\,dy+\int_0^tds\int_{\R^n}G(x-y,t-s)\,f(y,s)\,dy\,.
\end{eqnarray}
Solutions to the integral equation \eqref{duhamel formula} are usually called {\it mild solutions} of the parabolic PDE \eqref{duhamel equation}, and can be constructed by fixed points arguments. We now set up the stage to prove the short-time existence of a unique mild solution to \eqref{eq auxiliary flow}.

\medskip

From now on we fix $u_0$ with
\begin{equation}
  \label{data}
  u_0\in\bigcap_{p\ge 2}W^{2,p}(\R^n;[0,1])\,,\qquad\V(u_0)=1\,.
\end{equation}
By the properties of $V_\de$ and $W$ it is easily seen that $\V(u_0)=1$ implies  $\V_\de(u_0)\in(0,\infty)$ for every $\de>0$. Moreover, we also have
\begin{equation}
  \label{limit of vdelta}
  \lim_{\de\to 0^+}\V_\de(u_0)=\V(u_0)=1\,.
\end{equation}

\medskip

Given parameters $\tau>0$ and $\sigma\in(0,1)$ to be chosen in a moment, we then introduce the vector space
\begin{equation}
  \label{def of Y}
  Y=\big\{u\in C^0\big([0,\tau);B^X_\s(u_0)\big):u(0)=u_0\big\}\,,
\end{equation}
where $B^X_\s(u_0)$ denotes the ball in $X$ of radius $\s$ and center $u_0$. If we pick $\s$ small enough depending on $u_0$, and since $\V_\de(u_0)>0$ for every $\de>0$, we find that (see \eqref{volume vde ut} below) $\V_\de(u(t))>0$ for every $t\in[0,\tau)$: in particular, $\l_{\e,\de}[u(t)]$ is well defined for every $t\in[0,\tau)$. Hence, for each $\e>0$, $\de\in(0,\de_0]$, and $u\in Y$ we can define $F_\de[u]:\R^n\times[0,\tau)\to\R$ and $T_\de[u]:\R^n\times[0,\tau)\to\R$ by setting
\begin{eqnarray}\label{def of Fu}
F_\de[u]&=&-\frac{W'(u)}{\e^2}+\l_{\e,\de}[u(t)]\,\frac{V_\de'(u)}\e\,,
\\\label{def of Tu}
T_\de[u]&=&S_tu_0+\int_0^tS_{t-s}[F_\de[u](s)]\,ds\,.
\end{eqnarray}
{\it We claim that}, if $\tau$ and $\s$ (introduced in the definition of $Y$) are small enough with respect to $\e$, $\de$, and $u_0$, then $u\mapsto T_\de[u]$ defines a contraction of the Banach space $Y$. By the Banach fixed point theorem we will then deduce the existence of a unique $u\in C^0([0,\tau);B^X_\s(u_0))$ such that $u(0)=u_0$ and
\begin{equation}
  \label{mild solution of}
  u(t)=S_tu_0+\int_0^tS_{t-s}[F_\de[u](s)]\,ds\,,
\end{equation}
for every $t\in(0,\tau)$. In particular, this claim will prove the existence of a unique, short-time, mild solution $u$ of \eqref{eq auxiliary flow}.

\medskip

To prove our claim, we begin by showing that if $\s$ is small enough in terms of $u_0$ and $\de$, then for every $u\in Y$ we have
\begin{equation}
  \label{volume vde ut}
  \V_\de(u(t))\ge\frac{\V_\de(u_0)}2\,,\qquad\forall t\in[0,\tau)\,.
\end{equation}
Indeed, by $|V_\de'(r)|\le C\,|r|$ for every $r\in(0,1)$ and by $\|u_0-u(t)\|_X<\s<1$ we find that
\begin{eqnarray*}
\V_\de(u(t))-\V_\de(u_0)&=&\int_{\R^n}\,\int_0^1 V_\de'\big(s\,u_0+(1-s)\,u(t)\big)\,(u_0-u(t))\,ds
\\
&\ge&-C\,\int_{\R^n}\,|u_0-u(t)|\,\big(|u_0|+|u(t)|\big)
\\
&\ge& -C\,\|u(t)-u_0\|_{L^2}\,\big(\|u_0\|_{L^2}+\|u(t)\|_{L^2}\big)
\\
&\ge&-C\,\s\,\big(2\,\|u_0\|_{L^2}+\s\big)\ge-\frac{\V_\de(u_0)}2\,,
\end{eqnarray*}
provided $\s$ is small enough in terms of $\de$ and $u_0$ (recall that $\V_\de(u_0)>0$). Having proved \eqref{volume vde ut} we notice that, combined with \eqref{eq upper bound lambda}, \eqref{eq Lip lambda}, and $\AC(u)\le C(\e)\,\|u\|_{W^{1,2}}$, it implies
\begin{eqnarray}
  \label{base 1}
  \big|\l_{\e,\de}[u(t)]\big|&\le& C(\e,\de,u_0)
  \\
  \label{base 2}
  \big|\l_{\e,\de}[u(t)]-\l_{\e,\de}[v(t)]\big|&\le& C(\e,\de,u_0)\,\|u(t)-v(t)\|_X\,,
\end{eqnarray}
for every $u,v\in Y$ and every $t\in[0,\tau)$. Next we notice that if $u\in Y$, then for every $t\in(0,\tau)$ we have
\begin{equation}
  \label{Fu absolute value}
  |F_\de[u(t)]|\le C(\e)\,\max\big\{1,|\l_{\e,\de}[u(t)]\big\}\,|u(t)|\,,\qquad\mbox{on $\R^n$}\,;
\end{equation}
moroever, since $|W'(r)|\le C\,|r|$ and $|V_\de'|\le C\,|r|$ for $r\in(0,1)$, $W''$ and $V_\de''$ are bounded on $[0,1]$, and
\begin{equation}
  \label{base 3}
  \nabla\big(F_\de[u](t))=\Big\{-\frac{W''(u)}{\e^2}+\l_{\e,\de}[u(t)]\,\frac{V_\de''(u)}\e\Big\}\,\nabla u(t)\,,
\end{equation}
then
\begin{equation}
  \label{DFu absolute value}
  |\nabla F_\de[u(t)]|\le C(\e)\,\max\big\{1,|\l_{\e,\de}[u(t)]\big\}\,|\nabla u(t)|\,,\qquad\mbox{on $\R^n$}\,.
\end{equation}
By combining \eqref{Fu absolute value} and \eqref{DFu absolute value} with \eqref{base 1}, we see that if $u\in Y$, then $F_\de[u]\in C^0([0,\tau);X)$ and
\begin{eqnarray}
\label{utile 1}
\|F_\de[u](t)\|_{C^0}\!\!&\le&\!\! C(\e)\,\max\big\{1,|\l_{\e,\de}[u(t)]\big\}\,\|u(t)\|_{C^0}\,,
\\\label{utile 2}
\|F_\de[u](t)\|_{L^2}\!\!&\le&\!\! C(\e)\,\max\big\{1,|\l_{\e,\de}[u(t)]\big\}\,\|u(t)\|_{L^2}\,,
\\\label{utile 3}
\|\nabla(F_\de[u](t))\|_{L^2}\!\!&\le&\!\! C(\e)\,\max\big\{1,|\l_{\e,\de}[u(t)]\big\}\,\|\nabla u(t)\|_{L^2},\qquad\forall t\in[0,\tau)\,.
\end{eqnarray}
By combining these estimates with \eqref{base 1} we thus find
\begin{equation}
  \label{base 4}
  \|F_\de[u(t)]\|_X\le C(\e,\de,u_0)\,\|u\|_X\,,\qquad\forall t\in[0,\tau)\,.
\end{equation}
Since $\|S_tv\|_{C^0}\le \|v\|_{C^0}$, $\|S_tv\|_{L^2}\le \|v\|_{L^2}$, and $\|\nabla(S_tv)\|_{L^2}\le \|\nabla v\|_{L^2}$ for every $v\in X$ and every $t>0$, we deduce that if $u\in Y$, then $T_\de[u]\in C^0([0,\tau);X)$. Moreover, since $\|S_tu_0-u_0\|_X\to 0$ as $t\to 0^+$, if we pick $\tau$ small enough in terms of $\e$, $\de$, $u_0$ and $\s$ (where $\s$ has already been chosen small enough in terms of $u_0$ and $\de$), then we find
\begin{eqnarray*}
\|T_\de[u](t)-u_0\|_X&\le& \|S_tu_0-u_0\|_X+t\,\sup_{0<s<t}\|F_\de[u](t)\|_X
\\
&\le&\|S_tu_0-u_0\|_X+C(\e,\de,u_0)\,t\,\|u(t)\|_X<\s\,,
\end{eqnarray*}
for every $t\in[0,\tau_0)$, so that $T_\de[u]\in Y$ for every $u\in Y$. Having proved that $T_\de$ is a self-map of $Y$, we now pick $u,v\in Y$ and notice that for every $t\in[0,\tau)$ we have
\begin{equation}
  \label{base 0}
  \|T_\de[u]-T_\de[v]\|_Y=\sup_{0<t<\tau}\|T_\de[u](t)-T_\de[v](t)\|\le \tau\,\sup_{0<t<\tau}\|F_\de[u](t)-F_\de[v](t)\|\,.
\end{equation}
Now for every $u,v\in Y$ and $t\in[0,\tau)$ we have, pointwise on $\R^n$,
\begin{eqnarray}\nonumber
  &&|F_\de[u](t)-F_\de[v](t)|\le\Big\{\frac{\Lip(W')}{\e^2}+\big|\l_{\e,\de}[u(t)]\big|\,\frac{\Lip(V_\de')}\e\Big\}\,|u(t)-v(t)|
  \\\label{utile 4}
  &&\hspace{4cm}+ \big|\l_{\e,\de}[u(t)]-\l_{\e,\de}[v(t)]\big|\,\frac{|V_\de'(v(t))|}\e
  \\\nonumber
  &&\hspace{1cm}\le C(\e)\,\Big\{\max\big\{1,\big|\l_{\e,\de}[u(t)]\big|\big\}\,|u(t)-v(t)|+\big|\l_{\e,\de}[u(t)]-\l_{\e,\de}[v(t)]\big|\,|v(t)|\Big\}\,,
\end{eqnarray}
where we have used $|V_\de'(r)|\le C\,|r|$ for $r\in[0,1]$; then, by \eqref{base 1},  \eqref{base 2}, and \eqref{utile 4}, we find that
\begin{eqnarray}\label{base 41}
  &&\|F_\de[u](t)-F_\de[v](t)\|_X\le C(\e,\de,u_0)\,\|u(t)-v(t)\|_X\,,\qquad\forall t\in[0,\tau)\,.
\end{eqnarray}
Similarly, starting from \eqref{base 3}, we find that, pointwise on $\R^n$,
\begin{eqnarray}\nonumber
  &&|\nabla(F_\de[u](t))-\nabla (F_\de[v](t))|\le \Big\{\frac{\Lip(W')}{\e^2}+\big|\l_{\e,\de}[u(t)]\big|\,\frac{\Lip(V_\de')}\e\Big\}\,|\nabla u(t)-\nabla v(t)|
  \\\nonumber
  &&\hspace{4cm}+\Big\{\frac{\Lip(W'')}{\e^2}\,+\frac{\Lip(V_\de'')}\e\,\,|\l_{\e,\de}[v(t)]\Big\}\,|\nabla v(t)|\,|u(t)-v(t)|
  \\\nonumber
  &&\hspace{4cm}+\Big|\l_{\e,\de}[u(t)]-\l_{\e,\de}[v(t)]\Big|\,|\nabla v(t)|\,\frac{|V_\de''(u(t))|}\e
  \\\label{utile 5}
  &&\le C(\e,\de)\,\Big\{\max\big\{1,\big|\l_{\e,\de}[u(t)]\big|\big\}\,\big\{|\nabla v(t)|\,|u(t)-v(t)|+|\nabla u(t)-\nabla v(t)|\big\}
  \\\nonumber
  &&\hspace{8cm}+\Big|\l_{\e,\de}[u(t)]-\l_{\e,\de}[v(t)]\Big|\,|\nabla v(t)|\Big\}\,,
\end{eqnarray}
where we have made crucial use of the regularization $V_\de$ of $V$ to assert that $\Lip(V_\de'')\le C(\de)$. By \eqref{base 1}, \eqref{base 2}, \eqref{utile 5}, and $\s<1$,
\begin{eqnarray}\label{base 6}
   \|\nabla(F_\de[u](t))-\nabla (F_\de[v](t))\|_{L^2}\le C(\e,\de,u_0)\,\|u(t)-v(t)\|_{X}\,,\qquad\forall t\in[0,\tau)\,.
\end{eqnarray}
By combining \eqref{base 41} and \eqref{base 6} with \eqref{base 0} we conclude that
\[
\|T_\de[u]-T_\de[v]\|_Y\le C(\e,\de,u_0)\,\tau\,\|u-v\|_Y\,.
\]
In particular, up to further decrease $\tau$ depending on $\e$, $\de$ and $u_0$, we can ensure that $\Lip(T_\de;Y)<1$, and that $T_\de$ is a contraction of $Y$.

\medskip

\noindent {\it Step two, regularity and global-in-time existence for \eqref{eq auxiliary flow}}: Given $\e>0$, $\de\in(0,\de_0]$ and $u_0$ as in \eqref{data}, by step five we can define $\tau_*\in(0,\infty]$ as the supremum of those $\tau>0$ such that there exists $u\in C^0([0,\tau);X)$ such that $u(0)=u_0$ and \eqref{mild solution of} holds for every $t\in(0,\tau)$ and
\begin{equation}
  \label{bounds 0}
  \sup_{0<t<\tau}\big\{\AC(u(t)),\V_\de(u(t))\big\}<\infty\,,\qquad\inf_{0<t<\tau}\V_\de(u(t))>0\,.
\end{equation}
In this step we prove that $\tau_*=+\infty$, that $\AC(u(t))$ is Lipschitz continuous and decreasing on $[0,\infty)$, that $\V_\de(u(t))=\V_\de(u_0)$ for every $t\in[0,\infty)$, and that $u(t)\in W^{3,p}(\R^n)$ and $\pa_t u(t)\in W^{1,p}(\R^n)$ for every $t>0$ and $p\ge 2$ with
\begin{eqnarray}
\label{full estimate delta}
&&\max\big\{\|u(t)\|_{W^{3,p}(\R^n)},\|\pa_tu(t)\|_{W^{1,p}}\big\}
\\\nonumber
&&\le C(\e,p)\,M\big(\AC(u_0),\V_\de(u_0),1/\V_\de(u_0),\|\nabla u_0\|_{W^{2,p}},t\big)\,,
\end{eqnarray}
where $M$ denotes a generic constant which is {\it increasing and continuous} in its arguments.

\medskip

We first notice that by combining $0\le u\le 1$, \eqref{eq l2 bound}, \eqref{eq upper bound lambda} and \eqref{bounds 0} we find that
\begin{equation}
  \label{bounds 1}
  \sup_{0<s<t}\|u(s)\|_{L^p}<\infty\,,\qquad\sup_{0<s<t}\big|\l_{\e,\de}[u(s)]\big|<\infty
\end{equation}
for every $t\in(0,\tau_*)$ and $p\ge 2$. Next, setting $G(t)=G(\cdot,t)$ (recall \eqref{gxt}), we notice that for every $t>0$ we have
\begin{equation}
  \label{gt L1}
  \|G(t)\|_{L^1}=1\,,\qquad
  \max\big\{t^{1/2}\|\nabla G(t)\|_{L^1},\,t^{3/2}\|\nabla^2 G(t)\|_{L^1},\,t\|\pa_tG(t)\|_{L^1}\big\}\le C(n)\,.
\end{equation}
Combining $u(t)=T_\de[u(t)]$ with \eqref{gt L1}, \eqref{Fu absolute value} and standard applications of Fubini's theorem and H\"older's inequality, we find that, if $p\ge 2$, then, for every $t\in(0,\tau_*)$,
\begin{eqnarray}\label{gradient one}
\|\nabla u(t)\|_{L^p}&\le&\|\nabla u_0\|_{L^p}\,\|G(t)\|_{L^1}
\\\nonumber
&&+C(\e)\,\sup_{0<s<t}\,\max\big\{1,\big|\l_{\e,\de}[u(s)]\big|\big\}\,\|u(s)\|_{L^p}\,\,\int_0^t\|\nabla G(t-s)\|_{L^1}\,ds
\\\nonumber
&\le&\|\nabla u_0\|_{L^p}+C(\e)\,t^{1/2}\,\sup_{0<s<t}\,\max\big\{1,\big|\l_{\e,\de}[u(s)]\big|\big\}\,\|u(s)\|_{L^p}\,,
\end{eqnarray}
so that \eqref{bounds 1} gives
\begin{equation}
  \label{gradient one pre}
  \sup_{0<s<t}\|\nabla u(s)\|_{L^p}<\infty\,,\qquad\forall t\in(0,\tau_*)\,,p\ge 2\,.
\end{equation}
Similarly, again from $u(t)=T_\de[u(t)]$ and \eqref{gt L1}, we obtain that
\begin{eqnarray}\label{gradient two}
\|\nabla^2 u(t)\|_{L^p}&\le&\|\nabla^2u_0\|_{L^p}\,\|G(t)\|_{L^1}
\\\nonumber
&&+C(\e)\,\sup_{0<s<t}\,\max\big\{1,\big|\l_{\e,\de}[u(s)]\big|\big\}\,\|\nabla u(s)\|_{L^p}\,\,\int_0^t\|\nabla G(t-s)\|_{L^1}\,ds
\\\nonumber
&\le&\|\nabla^2 u_0\|_{L^p}+C(\e)\,t^{1/2}\,\sup_{0<s<t}\,\max\big\{1,\big|\l_{\e,\de}[u(s)]\big|\big\}\,\|\nabla u(s)\|_{L^p}\,
\end{eqnarray}
which, combined with \eqref{base 1}, \eqref{bounds 1}, and \eqref{gradient one pre}, gives
\begin{equation}
  \label{summary w2p}
  \sup_{0<s<t}\|u(s)\|_{W^{2,p}}<\infty\,,\qquad\forall t\in(0,\tau_*)\,,p\ge 2\,.
\end{equation}
In particular, $u\in L^\infty_{\rm loc}((0,\tau);W^{2,2}(\R^n))$, so that, by standard properties of mild solutions (see e.g. \cite[Proposition 4.1.9]{cazehara}) \eqref{mild solution of} implies that
\begin{equation}
  \label{w11}
  u\in W^{1,1}_{\rm loc}((0,\tau_*);X)\,,
\end{equation}
with
\begin{equation}
  \label{strong solution}
  \partial_t u=  2 \, \Delta u +f_\de\qquad\mbox{a.e. on $\R^n\times(0,\tau_*)$}\,,
\end{equation}
where we have set $f_\de(x,t)=F_\de[u](x,t)$, and where, by \eqref{Fu absolute value} and \eqref{DFu absolute value},
\begin{eqnarray}
  \label{fdelta Lp}
  \|f_\de(t)\|_{L^p}\!\!&\le&\!\!
  C(\e)\,\max\big\{1,|\l_{\e,\de}[u(t)]\big\}\,\|u(t)\|_{L^p}\,,
  \\
  \label{fdelta W1p}
  \|\nabla f_\de(t)\|_{L^p}\!\!&\le&\!\!
  C(\e)\,\max\big\{1,|\l_{\e,\de}[u(t)]\big\}\,\|\nabla u(t)\|_{L^p}\,,\qquad\forall p\ge 2\,.
\end{eqnarray}
Moreover, by \eqref{strong solution}, we see that for every $t\in(0,\tau_*)$ and $p\ge 2$,
\begin{equation}
  \label{patu with D2u and f}
  \|\pa_t u(t)\|_{L^p}\le C(\e)\,\Big\{\|\nabla^2u(t)\|_{L^p}+\|f_\de(t)\|_{L^p}\Big\}\,.
\end{equation}

We now differentiate the flow in space to obtain $L^p$-estimates on $\nabla^3 u$ and on $\nabla(\pa_tu)$. Given $e\in\R^n$ with $|e|=1$ and $v=v(x,t)$ we set $e_hv(x,t)=(v(x+h\,e,t)-v(x,t))/h$ for the (spatial) incremental ratio of $v$ in the direction $v$ of step $h$. In this way \eqref{strong solution} implies that $e_hu$ solves
\begin{eqnarray}
  \label{PDE for ehu}
  \pa_t(e_hu)-2\,\Delta(e_hu)=e_hf_\de\qquad\mbox{on $\R^n\times(0,\infty)$},\qquad\mbox{$e_hu(0)=e_hu_0$ on $\R^n$}\,.
\end{eqnarray}
Next we consider the decomposition $e_hu=u_{1,h}+u_{2,h}$ where $u_{1,h}(t)=S_t[e_hu_0]$ and thus
\begin{eqnarray}
  \label{PDE for u2h}
  \pa_tu_{2,h}-2\,\Delta u_{2,h}=e_hf_\de\qquad\mbox{on $\R^n\times(0,\infty)$},\qquad\mbox{$u_{2,h}(0)=0$ on $\R^n$}\,.
\end{eqnarray}
Since $u_{2,h}(0)=0$ on $\R^n$ we can apply \cite[Corollary 7.31]{lieberman} to deduce that for every $(a,b)\cc(0,\tau_*)$ and $p\ge 2$
\begin{equation}
  \int_a^b\,dt\int_{\R^n}|\pa_tu_{2,h}(t)|^p+|\nabla^2u_{2,h}(t)|^p\le C(p)
  \int_a^b\,dt\int_{\R^n}|e_hf_\de(t)|^p\,.
\end{equation}
Since $\|e_hf_\de(t)\|_{L^p}\le \|\nabla f_\de(t)\|_{L^p}$ and, thanks to \eqref{fdelta W1p}, \eqref{bounds 1}, and \eqref{gradient one pre},
\[
\sup_{0<s<t}\|\nabla f_\de(s)\|_{L^p}<\infty\,,\qquad\forall t\in(0,\tau_*)\,,
\]
we conclude that, for every $(a,b)\subset(0,t)$, $t<\tau_*$, and $p\ge 2$,
\begin{equation}
  \int_a^b\,dt\int_{\R^n}|\pa_tu_{2,h}(t)|^p+|\nabla^2u_{2,h}(t)|^p\le C(p)
  (b-a)\,\sup_{a<s<b}\|\nabla f_\de(s)\|_{L^p}^p\,,
\end{equation}
so that, by arbitrariness of $(a,b)$, for every $(a,b)\cc(0,\tau_*)$ and $p\ge 2$,
\begin{equation}
  \label{cz para u2h}
  \sup_{a<s<b}\big\{\|\pa_tu_{2,h}(s)\|_{L^p},\|\nabla^2u_{2,h}(s)\|_{L^p}\big\}
  \le C(p)\, \sup_{a<s<b}\|\nabla f_\de(s)\|_{L^p}\,.
\end{equation}
At the same time, recalling that $u_{1,h}(s)=S_s(e_hu_0)$, by \eqref{gt L1} we find that for every $s>0$
\begin{eqnarray*}
  \|\pa_tu_{1,h}(s)\|_{L^p}\le\|\pa_tG(s)\|_{L^1}\,\|e_hu_0\|_{L^p}\le \frac{C}{s}\,\|\nabla u_0\|_{L^p}\,,
  \\
  \|\nabla^2u_{1,h}(s)\|_{L^p}\le \|\nabla^2 G(s)\|_{L^1}\,\|e_hu_0\|_{L^p}\le \frac{C}{s^{3/2}}\,\|\nabla u_0\|_{L^p}\,.
\end{eqnarray*}
By combining these last two estimates with \eqref{cz para u2h} and $e_hu=u_{1,h}+u_{2,h}$, by the uniformity in $h>0$ and $e$ with $|e|=1$, we conclude that for every $p\ge 2$ and $(a,b)\cc(0,\tau_*)$ we have
\begin{eqnarray}\label{useful inter}
  \sup_{a<s<b}\big\{\|\nabla(\pa_t u)(s)\|_{L^p},\|\nabla^3u(s)\|_{L^p}\big\}
  \le C(p)\sup_{a<s<b}\Big\{\frac{\|\nabla u_0\|_{L^p}}{a^{3/2}},\|\nabla f_\de(s)\|_{L^p}\Big\}\,.
\end{eqnarray}
In combination with \eqref{summary w2p}, \eqref{patu with D2u and f}, \eqref{fdelta Lp}, and \eqref{bounds 1} this implies that, for every $p\ge 2$ and $(a,b)\cc(0,\tau_*)$ we have
\begin{eqnarray}\label{smooth udelta}
  \sup_{a<t<b}\big\{\|\pa_tu(t)\|_{W^{1,p}},\|u(t)\|_{W^{3,p}}\big\}<\infty\,.
\end{eqnarray}
In particular, $u(t)\in C^{2,\a}(\R^n)$ and $\pa_tu(t)\in C^{0,\a}(\R^n)$ for every $\a\in(0,1)$ and $t\in(0,\tau_*)$.

\medskip

The regularity of $u(t)$ established in \eqref{smooth udelta} is sufficient to prove that
\begin{eqnarray}
  \label{vd1}
  &&\V_\de(u(t))=\V_\de(u_0)\,,\qquad\forall t\in[0,\tau_*)\,,
  \\
  \label{vd2}
  &&\mbox{$t\mapsto\AC(u(t))$ is decreasing, continuous on $[0,\tau_*)$,}
  \\\nonumber
  &&\hspace{2.6cm}\mbox{and locally Lipschitz on $(0,\tau_*)$}\,.
\end{eqnarray}
Indeed, let us set $u_\eta=(\psi_\eta\,u)\star\rho_\eta$, where, for each $\eta>0$, $\rho_\eta$ is a mollifying kernel on $\R^n\times\R$, and where $\psi_\eta\in C^\infty_c(\R^n\times(0,\infty))$ is monotonically converging to $1$ as $\eta\to 0^+$. If we set $D(t)=\int_{\R^n}|\nabla u(t)|^2$ and $D_\eta(t)=\int_{\R^n}|\nabla u_\eta(t)|^2$, then by \eqref{gradient one pre} we find that, for every $\vphi\in C^\infty_c(0,\tau)$,
\begin{eqnarray*}
  &&\int_0^\tau D(t)\,\vphi'(t)\,dt=\lim_{\eta\to 0^+}\int_0^\tau\,D_\eta(t)\,\vphi'(t)\,dt
  \\
  &&=-2\,\lim_{\eta\to0^+}\int_0^\tau\vphi(t)\,dt\int_{\R^n}\nabla u_\eta(t)\cdot\nabla(\pa_tu_\eta(t))
  \\
  &&=-2\,\int_0^\tau\,\vphi(t)\,dt\int_{\R^n}\,\nabla u(t)\cdot\nabla(\pa_tu(t))=2\,\int_0^\tau\,\vphi(t)\,dt\int_{\R^n}\Delta u(t)\,\pa_tu(t)\,,
\end{eqnarray*}
where, in computing the second limit, we have used \eqref{smooth udelta}. We have thus proven that $D$ is locally Lipschitz continuous on $(0,\tau_*)$, with
\begin{equation}
  \label{nn1}
  D'(t)=\frac{d}{dt}\int_{\R^n}|\nabla u(t)|^2=-\int_{\R^n}2\,\Delta u(t)\,\pa_t u(t)\,.
\end{equation}
for a.e. $t\in(0,\tau_*)$. By an analogous approximation argument we see that, for a.e. $t\in(0,\tau_*)$,
\begin{equation}
  \label{nn2}
  \frac{d}{dt}\int_{\R^n}W(u(t))=\int_{\R^n}W'(u)\,\pa_tu\,,\qquad \frac{d}{dt}\int_{\R^n}V_\de(u(t))=\int_{\R^n}V_\de'(u)\,\pa_tu\,.
\end{equation}
In particular, by \eqref{strong solution}, for a.e. $t\in(0,\tau_*)$,
\begin{eqnarray*}
  &&\int_{\R^n}V_\de'(u)\,\pa_tu=\int_{\R^n}\Big(2\,\Delta u-\frac{W'(u)}{\e^2}+\l_{\e,\de}[u(t)]\,\frac{V_\de'(u)}\e\Big)\,V_\de'(u)
  \\
  &&=-\int_{\R^n}2\,\nabla u\cdot\nabla(V_\de'(u))-\frac1{\e^2}\int_{\R^n}W'(u)\,V_\de'(u)+\frac{\l_{\e,\de}[u(t)]}\e\,\int_{\R^n}V_\de'(u)^2=0\,,
\end{eqnarray*}
where the last identity follows from the definition \eqref{lambda eps delta u proof} of $\l_{\e,\de}[u(t)]$. This proves that $\V_\de(u(t))$ is constant on $(0,\tau_*)$, and since $t\mapsto\V_\de(u(t))$ is continuous on $[0,\tau_*)$, we deduce \eqref{vd1}. Finally, by \eqref{nn1} and \eqref{nn2} we have that $t\mapsto\AC(u(t))$ is locally Lipschitz continuous on $(0,\tau_*)$ with
\begin{eqnarray*}
  \frac{d}{dt}\AC(u(t))&=&\e\,\int_{\R^n}\pa_tu\,\Big\{-2\,\Delta u+\frac{W'(u)}{\e^2}\Big\}
\\
&=&\e\,\int_{\R^n}\pa_tu\,\Big\{-\pa_tu+\l_{\e,\de}[u(t)]\,\frac{V_\de'(u)}\e\Big\}=-\e\,\int_{\R^n}(\pa_tu)^2\,.
\end{eqnarray*}
where we have taken into account $\int_{\R^n}V_\de'(u)\,\pa_tu=0$. The continuity of $t\mapsto \AC(u(t))$ on $[0,\tau_*)$ is of course immediate from $u\in C^0([0,\tau_*);X)$. This proves \eqref{vd2}.

\medskip

We now prove that $\tau_*=+\infty$. We argue by contradiction, and assume that $\tau_*<\infty$. By combining \eqref{vd1} and \eqref{vd2} (which implies $\AC(u(t))\le\AC(u_0)$ for every $t\in(0,\tau_*)$) with \eqref{eq l2 bound} and \eqref{eq upper bound lambda} we deduce that, for every $p\ge 2$,
\begin{eqnarray}
  \label{hey 1}
\sup_{0<t<\tau_*}\big\{\|u(t)\|_{L^p},|\l_{\e,\de}[u(t)]|\big\}\le C(\e)\,M\big(\e,\AC(u_0),\V_\de(u_0),1/\V_\de(u_0)\big)\,.
\end{eqnarray}
Then, by \eqref{hey 1} and \eqref{gradient one} we find that, for every $p\ge 2$,
\begin{equation}
  \label{hey 2}
  \sup_{0<t<\tau_*}\|\nabla u(t)\|_{L^p}\le C(\e)\,M\big(\AC(u_0),\V_\de(u_0),1/\V_\de(u_0),\|\nabla u_0\|_{L^p},\tau_*\big)\,,
\end{equation}
which, combined with \eqref{gradient two} gives that, for every $p\ge 2$,
\begin{equation}
  \label{hey 4}
  \sup_{0<t<\tau_*}\|\nabla^2 u(t)\|_{L^p}\le C(\e,p)\,M\big(\e,\AC(u_0),\V_\de(u_0),1/\V_\de(u_0),\|\nabla u_0\|_{W^{1,p}},\tau_*\big)\,.
\end{equation}
By combining \eqref{fdelta Lp} and \eqref{fdelta W1p} with \eqref{hey 1} and \eqref{hey 2} we find that, for every $p\ge 2$,
\begin{equation}
  \label{summary h}
  \sup_{0<t<\tau_*}\|f_\de(t)\|_{W^{1,p}}\le C(\e)\,M\big(\e,\AC(u_0),\V_\de(u_0),1/\V_\de(u_0),\|\nabla u_0\|_{L^p},\tau_*\big)\,,
\end{equation}
so that by \eqref{patu with D2u and f} and \eqref{useful inter} with \eqref{hey 4} and \eqref{summary h} we find
\begin{eqnarray*}
  &&\sup_{\tau_*/2<t<\tau_*}\big\{\|\nabla^3u(t)\|_{L^p},\|\pa_tu(t)\|_{W^{1,p}}\big\}
  \\
  &&\hspace{3cm}\le C(\e,p)\,M\big(\AC(u_0),\V_\de(u_0),1/\V_\de(u_0),\|\nabla u_0\|_{W^{2,p}},\tau_*,1/\tau_*\big)\,,
\end{eqnarray*}
which combined with \eqref{hey 1}, \eqref{hey 2} and \eqref{hey 4} finally gives, for every $p\ge 2$,
\begin{eqnarray}
\label{hey3}
&&\sup_{\tau_*/2<t<\tau_*}\big\{\|u(t)\|_{W^{3,p}},\|\pa_tu(t)\|_{W^{1,p}}\big\}
\\\nonumber
&&\hspace{2cm}\le C(\e,p)\,M\big(\AC(u_0),\V_\de(u_0),1/\V_\de(u_0),\|\nabla u_0\|_{W^{2,p}},\tau_*,1/\tau_*\big)\,.
\end{eqnarray}
By the $W^{1,2}$-estimate on $\pa_tu$ contained in \eqref{hey3}, we can deduce that for every $t,s\in(\tau_*/2,\tau_*)$ it holds
\begin{equation}
  \label{lip}
  \|u(t)-u(s)\|_X\le C(\e,\de,u_0,\tau_*)\,|t-s|\,.
\end{equation}
Since $X$ is a Banach space, this means that there is $u_*\in X=C^0\cap W^{1,2}(\R^n;[0,1])$ such that
\[
\lim_{t\to\tau_*^-}\|u(t)-u_*\|_X=0\,.
\]
By combining this last fact with \eqref{hey3} we see that $u_*\in W^{3,p}(\R^n)$ for every $p\ge 2$, with
\[
\V_\de(u_*)=\lim_{t\to\tau_*^-}\V_\de(u(t))=\V_\de(u_0)>0\,.
\]
Therefore, for some $\eta>0$, we can repeat the argument of step five to extend $u$ as an element of $C^0([0,\tau_*+\eta);X)$ such that \eqref{mild solution of} holds for every $t\in(0,\tau_*+\eta)$. We can also repeat the proof of \eqref{vd1} and \eqref{vd2} to show that $\AC(u(t))$ is decreasing on $[\tau_*,\tau_*+\eta)$ and $\V_\de(u(t))$ is constant on $t\in[\tau_*,\tau_*+\eta)$. Since $\tau_*$ was  introduced as the supremum of those $\tau>0$ such that there exists $u\in C^0([0,\tau);X)$ with $u(0)=u_0$, solving \eqref{mild solution of} for every $t\in(0,\tau)$, and such that
\begin{equation}
  \label{contra}
  \sup_{0<t<\tau}\big\{\AC(u(t)),\V_\de(u(t))\big\}<\infty\,,\qquad\inf_{0<t<\tau}\V_\de(u(t))>0\,.
\end{equation}
and since we have just proved that the bounds stated in \eqref{contra} holds with $\tau=\tau_*+\eta$, thus reaching a contradiction with the maximality of $\tau_*$. This proves that $\tau_*=+\infty$. The monotonicity of $\AC(u(t))$ and constancy of $\V_\de(u(t))$ then hold on $[0,\infty)$, and the validity of \eqref{full estimate delta} is deduced by arguing as in the proof of  \eqref{hey3}.

\medskip

\noindent {\it Step three}: We prove, for an initial datum $u_0$ as in \eqref{data}, the existence of $u\in C^0(\R^n\times[0,\infty))$ with $u(0)=u_0$ which is a classical solution of \eqref{diffused VPMCF} on $\R^n\times(0,\infty)$, with $\V(u(t))=\V(u_0)=1$ for every $t\ge0$ and such that $t\mapsto\AC(u(t))$ is continuous and decreasing on $[0,\infty)$.

\medskip

Given $\e>0$ and $\de\in(0,\de_0]$, let us denote by $u^\de$ the unique global solution of \eqref{eq auxiliary flow} with $u^\de(0)=u_0$ constructed in step six. Recalling that $u^\de$ satisfies \eqref{full estimate delta}, and keeping in mind that $\V_\de(u_0)\to\V(u_0)=1$ by \eqref{limit of vdelta}, up to make $\de_0$ depend on $u_0$ too, we can ensure that $1/2\le\V_\de(u_0)\le 2$ for every $\de\in(0,\de_0]$, and thus deduce from \eqref{full estimate delta} that for every $(a,b)\subset(0,\infty)$ and $p\ge 2$, uniformly on $\de\in(0,\de_0]$,
\begin{eqnarray}
\label{full estimate delta bis}
&&\sup_{a<s<b}\big\{\|u^\de(s)\|_{W^{3,p}},\|\pa_tu^\de(s)\|_{W^{1,p}}\big\}\le C\big(\e,p,\AC(u_0),\|\nabla u_0\|_{W^{2,p}},a,b\big)\,.\hspace{1cm}
\end{eqnarray}
By Morrey's embedding theorem we can obtain a $C^{0,1/2}$-estimate on $\pa_tu^\de$ from \eqref{full estimate delta bis}, which combined with the $W^{1,2}$-estimate on $\pa_tu^\de$ contained in \eqref{full estimate delta bis} leads to prove that, for every $(a,b)\subset(0,\infty)$,
\begin{eqnarray}
\label{full estimate delta 2}
\|u^\de(r)-u^\de(s)\|_X\le C(\e,u_0,a,b)\,|r-s|,\qquad\forall r,s\in(a,b)\,,
\end{eqnarray}
uniformly on $\de\in(0,\de_0]$.

\medskip

If we now consider a sequence $\de_j\to 0^+$, then, up to extracting a subsequence, we deduce from \eqref{full estimate delta 2} and \eqref{full estimate delta bis} that there is $u\in \Lip_{\rm loc}((0,\infty);X)$ with
\begin{eqnarray}
\label{full estimate delta 2 u}
\sup_{r,s\in(a,b)}\frac{\|u(r)-u(s)\|_X}{|r-s|}\!\!&\le&\!\! C(\e,u_0,a,b)\,,
\\
\label{full estimate delta bis u}
\sup_{a<s<b}\big\{\|u(s)\|_{W^{3,p}},\|\pa_tu(s)\|_{W^{1,p}}\big\}
\!\!&\le&\!\!C(\e,p,\AC(u_0),\|\nabla u_0\|_{W^{2,p}},a,b\big)\,,
\end{eqnarray}
for every $p\ge 2$ and $(a,b)\subset(0,\infty)$, and such that
\begin{eqnarray}
 \label{full convergence}
  &&\lim_{j\to\infty}\sup_{a<s<b}\big\{\|u^{\de_j}(s)-u(s)\|_{W^{2,p}},\|\pa_tu^{\de_j}(s)-\pa_t u(s)\|_{L^p\cap C^0}\big\}=0\,,
\end{eqnarray}
for every $p\ge 2$ and $(a,b)\subset(0,\infty)$. By \eqref{full estimate delta bis u}, $u(t)\in C^{2,\a}(\R^n)$, $\pa_tu(t)\in C^{0,\a}(\R^n)$ for every $\a\in(0,1)$ and $t>0$. By \eqref{full convergence}, and since
\[
\e^2\,\pa_tu^{\de_j}=2\e^2\,\Delta u^{\de_j}- W'(u^{\de_j})+\e\,\l_{\e,\de_j}[u^{\de_j}(t)]\,V'_{\de_j}(u^{\de_j})\,,
\]
holds in classical sense on $\R^n\times(0,\infty)$, we can deduce that
\begin{equation}
  \label{limit equation}
  \e^2\,\pa_tu=2\,\e^2\,\Delta u-W'(u)+\e\,\l_{\e}[u(t)]\,V'(u)\,,
\end{equation}
also holds in classical sense on on $\R^n\times(0,\infty)$, provided we show that
\begin{equation}
  \label{provided}
  \lim_{j\to\infty}\l_{\e,\de_j}[u^{\de_j}(t)]=\l_\e[u(t)]\,,\qquad\forall t>0\,.
\end{equation}
To prove \eqref{provided} we first notice that, by $0\le V_\de(r)\le C\,r^2$ for $r\in[0,1]$ any by dominated convergence, for every $t>0$ it holds that $\V_{\de_j}(u^{\de_j}(t))\to\V(u(t))$ as $j\to\infty$. At the same time, for every $t>0$, $\V_{\de_j}(u^{\de_j}(t))=\V_{\de_j}(u_0)\to\V(u_0)$ as $j\to\infty$ so that, in summary,
\begin{equation}
  \label{volumeeee}
  \V(u(t))=\V(u_0)=1\,,\qquad\forall t>0\,.
\end{equation}
Now, by \eqref{eq holder}, \eqref{eq holder 2}, $\V_\de(u^\de(s))=\V_\de(u_0)$, \eqref{volumeeee}, and \eqref{full estimate delta bis u}, we have that
\begin{eqnarray*}
|\l_{\e,\de}[u^\de(t)]-\l_{\e,\de}[u(t)]|\!\!&\le&\!\!
C(\e,u_0,b)\,\frac{\|u^\de(t)-u(t)\|_{X}^{\gamma(n)}}{\min\{1,\V_\de(u_0)^{2n}\}\,\min\{1,\V_\de(u(t))\}^{2n}}\,,
\\
|\l_{\e,\de}[u(t)]-\l_{\e}[u(t)]|\!\!&\le&\!\!
C(\e,u_0,a,b)\, \frac{\|V_\de-V\|_{C^2[0,1]}}{\min\{1,\V_\de(u(t))^{2n}\}}\,,
\end{eqnarray*}
for all $t\in(a,b)\subset(0,\infty)$. Since $\V_\de(u(t))\in(0,\infty)$ for every $t>0$ and $\de\in(0,\de_0]$, by letting $\de=\de_j$ and $j\to\infty$ in the above two estimates we obtain \eqref{provided}, and thus \eqref{limit equation}. The constancy of $\V(u(t))$ and the monotonicity of $\AC(u(t))$ are then immediate to prove.

\medskip

We are left to prove that $u(0)=u_0$. To begin with, we notice that, by \eqref{eq upper bound lambda}, \eqref{vd1} and \eqref{vd2}, for all $t>0$ we have
\[
|\l_{\e,\de}[u^\de(t)]|\leq C(\e)\,\frac{\AC(u^\de(t))^{2n+2}}{\V_\delta(u^\de(t))^{2n}}
\leq C(\e)\,\frac{\AC(u_0)^{2n+2}}{\V_\delta(u_0)^{2n}}\,,
\]
so that \eqref{provided} and \eqref{limit of vdelta} imply
\begin{equation}
  \label{final bound on lambda eps pre}
|\l_\e[u(t)]|\le C(\e,u_0)\,,\qquad\forall t>0\,.
\end{equation}
Next, if we set
\begin{eqnarray}\label{def of Fu zero}
F[u](t)&=&-\frac{W'(u(t))}{\e^2}+\l_{\e}[u(t)]\,\frac{V'(u(t))}\e\,,
\\\label{def of Tu zero}
T[u](t)&=&S_tu_0+\int_0^tS_{t-s}[F[u](s)]\,ds\,,
\end{eqnarray}
then \eqref{limit equation} implies $u(t)=T[u(t)]S_tu_0+\int_0^tS_{t-s}[F[u](s)]\,ds$ for every $t>0$, and by the contraction properties of the heat flow, \eqref{utile 1}, \eqref{utile 2}, \eqref{utile 3}, and \eqref{final bound on lambda eps pre}, we find that
\begin{eqnarray*}
\|u(t)-u_0\|_X&\le&\|S_tu_0-u_0\|_X+C(\e)\,\int_0^t\big(1+|\l_\e[u(s)]\big)\,\|u(s)\|_X\,ds
\\
&\le& \|S_tu_0-u_0\|_X+C(\e,u_0)\,t\,\sup_{0<s<t}\|u(s)\|_X\,.
\end{eqnarray*}
By $u_0\in X$ and \eqref{full estimate delta bis u} we find that $\|u(t)-u_0\|_X\to 0$ as $t\to 0^+$, and thus that $u\in C^0(\R^n\times[0,\infty))$ with $u(0)=u_0$, as claimed. This completes the proof of step seven.

\medskip

\noindent {\it Step four}: We now prove that
\begin{equation}
\label{final bound on lambda eps}
|\l_\e[u(t)]|\le C\,\AC(u_0)^{2n+2}\,,\qquad\forall t>0\,,
\end{equation}
and that, for every $t_0>0$ and $p\ge 2$ we have
\begin{eqnarray}
\label{final hilbert estimate 1}
  \sup_{t\ge 0}\big\{\|u(t)\|_{W^{2,p}},\|\pa_tu(t)\|_{L^p}\big\}\le C(\e,p,u_0)\,,
  \\
\label{final hilbert estimate 2}
  \sup_{t\ge t_0}  \max\big\{\|u(t)\|_{W^{3,p}},\|\pa_tu(t)\|_{W^{2,p}},\|\pa_{tt}u(t)\|_{L^p}\big\}\le C(\e,p,u_0,t_0)\,.
\end{eqnarray}
Indeed, recalling that in step seven we have proved $\V(u(t))=1$ and $\AC(u(t))\le\AC(u_0)$ for every $t\ge0$, and recalling that the constant $C(\e)$ in \eqref{eq upper bound lambda} can be taken independent from $\e$ when \eqref{eq upper bound lambda} is applied with $\de=0$, we find that, for all $t>0$,
\[
|\l_\e[u(t)]|\leq C\,\frac{\AC(u(t))^{2n+1}}{\V(u(t))^{2n}}\,\int_{\R^n}|\nabla u|^2\le \frac{C}\e\,\AC(u_0)^{2n+2}\,,
\]
that is \eqref{final bound on lambda eps}. Similarly, we deduce from \eqref{eq l2 bound} (with $\de=0$) and $0\le u\le 1$ that
\begin{equation}
  \label{final Lp estimate}
  \sup_{t\ge0}\|u(t)\|_{L^p}\le C(\e,u_0)\,,\qquad\forall p\ge 2\,,
\end{equation}
while by combining \eqref{gradient one} with \eqref{final bound on lambda eps} and \eqref{final Lp estimate} we get
\begin{equation}
  \label{temp}
  \sup_{0\le t\le 1}\|\nabla u(t)\|_{L^p}\le C(\e,p,u_0)\,,\qquad\forall p\ge 2\,.
\end{equation}
Now, using the semigroup property of the heat flow we see that for every $t>s\ge0$ we have
\begin{equation}
  \label{representation st}
  u(t)=S_{t-s}u(s)+\int_s^tS_{t-r}\big[F[u](r)\big]\,dr\,.
\end{equation}
By differentiating \eqref{representation st}, and by using \eqref{gt L1}, \eqref{final bound on lambda eps} and \eqref{final Lp estimate}, we find that
\begin{eqnarray}\nonumber
  \|\nabla u(t)\|_{L^p}&\le&
  \|\nabla G(t-s)\|_{L^1}\,\|u(s)\|_{L^p}
  +\sup_{s<r<t}\|F[u](r)\|_{L^p}\,\int_s^t\|\nabla G(t-r)\|_{L^1}\,dr
  \\\label{temp 2}
  &\le& \frac{C(\e,u_0)}{(t-s)^{1/2}}+C(\e,u_0)\,(t-s)^{1/2}\,,
\end{eqnarray}
where we have used the analog to \eqref{Fu absolute value} with $\de=0$ in estimating $\|F[u](r)\|_{L^p}$. If $t\ge 1$, then we can apply \eqref{temp 2} with $s=t-1\ge0$ to deduce that $\|\nabla u(t)\|_{L^p}\le C(\e,u_0)$, which, combined with \eqref{temp}, gives
\begin{equation}
  \label{final grad Lp estimate}
  \sup_{t\ge0}\|\nabla u(t)\|_{L^p}\le C(\e,p,u_0)\,,\qquad\forall p\ge 2\,.
\end{equation}
Similarly, combining \eqref{gradient two} with \eqref{final bound on lambda eps} and \eqref{final grad Lp estimate} we see that, on the one hand
\begin{equation}
  \label{temp2}
  \sup_{0\le t\le 1}\|\nabla^2 u(t)\|_{L^p}\le C(\e,p,u_0)\,,\qquad\forall p\ge 2\,;
\end{equation}
on other hand, using again \eqref{representation st}, \eqref{gt L1}, and \eqref{final bound on lambda eps}, this time in combination with \eqref{final Lp estimate}, \eqref{DFu absolute value}, and \eqref{final grad Lp estimate}, we find that, if $t>s\ge0$, then
\begin{eqnarray*}
  \|\nabla^2u(t)\|_{L^p}&\le&
  \|\nabla^2G(t-s)\|_{L^1}\,\|u(s)\|_{L^p}
  +\sup_{s<r<t}\|\nabla F[u](r)\|_{L^p}\,\int_s^t\|\nabla G(t-r)\|_{L^1}\,dr
  \\
  &\le& \frac{C(\e,u_0)}{(t-s)^{3/2}}+C(\e,u_0)\,(t-s)^{1/2}\,,
\end{eqnarray*}
and using this last estimate for $t\ge 1$ (and with $s=t-1\ge0$), we find $\|\nabla^2u(t)\|_{L^p}\le C(\e,u_0)$ for every $t\ge 1$ and $p\ge 2$. By combining this last fact with \eqref{temp2} we have thus proved
\begin{equation}
  \label{final w22 estimate}
  \sup_{t>0}\big\{\|u(t)\|_{W^{2,p}},\|\pa_tu(t)\|_{L^p}\big\}\le C(\e,p,u_0)\,,\qquad\forall p\ge 2\,.
\end{equation}
Here the $L^p$-estimate on $\pa_t u(t)$ has been obtained by combining the $W^{2,p}$-estimate for $u(t)$ with $\pa_tu=2\,\Delta u+f$, where $f(t)=F[u(t)]$, and thus $\|f(t)\|_{L^p}\le C(\e,p,u_0)$ thanks to \eqref{Fu absolute value} and \eqref{final Lp estimate}. In fact, thanks also to \eqref{DFu absolute value}, \eqref{final bound on lambda eps}, and \eqref{final grad Lp estimate} for every $p\ge 2$, we have
\[
\sup_{t\ge 0}\|f(t)\|_{W^{1,p}}\le C(\e,p,u_0)\,.
\]
We can now repeat the argument based on the incremental ratios method and on the parabolic Calderon--Zygmund theorem used in the proof of \eqref{useful inter} to deduce that
\begin{equation*}
  \sup_{t>t_0}\big\{\|\nabla(\pa_tu)(t)\|_{L^p},\|\nabla^3u(t)\|_{L^p}\big\}
  \le C(p)\, \max\Big\{\frac{\|\nabla u_0\|_{L^p}}{t_0^{3/2}},\sup_{t>t_0}\|\nabla f(t)\|_{L^p}\Big\}\,,
\end{equation*}
and conclude, in summary, that, for every $p\ge 2$ and $t_0>0$,
\begin{equation}
  \label{final w32 estimate}
  \sup_{t>t_0}\big\{\|u(t)\|_{W^{3,p}},\|\pa_tu(t)\|_{W^{1,p}}\big\}\le C(\e,p,u_0,t_0)\,.
\end{equation}
The $W^{1,2}$-estimate on $\pa_tu(t)$ allows one to deduce by elementary means that
\begin{equation}
  \label{final ut us}
  \|u(t)-u(s)\|_{W^{1,2}(\R^n)}\le C(\e,u_0,t_0)\,|t-s|\,,\qquad\forall t,s>t_0\,.
\end{equation}
Combining \eqref{final ut us} with \eqref{eq lipschitz} (with $\de=0$) and \eqref{final w22 estimate} with $p=2$ we conclude that
\begin{equation}
  \label{final lambda eps Lip}
  |\l_\e[u(t)]-\l_\e[u(s)]|\le C(\e,u_0,t_0)\,|t-s|\,,\qquad\forall t,s>t_0\,,
\end{equation}
that is $t\mapsto\l_\e[u(t)]$ is Lipschitz continuous on $(t_0,\infty)$ for every $t_0>0$.

\medskip

To obtain $L^p$-estimates for $\pa_{tt}u(t)$ and $\nabla^2(\pa_tu)$ we need to differentiate $\pa_tu=2\,\Delta u+f$ in time. To this end, given $t_0>0$, we introduce that incremental ratio operator $T_h$ that acts on $v=v(x,t)$ by taking $T_hv(t)=(v(t_0+t+h)-v(t_0+t))/h$ for every $t\ge0$ and $h\in(-t_0,t_0)\setminus\{0\}$. With this notation, $\pa_tu=2\,\Delta u+f$ on $\R^n\times(0,\infty)$ implies that
\begin{eqnarray}
 \label{PDE for incratios in time}
  &&\pa_t(T_hu)-2\,\Delta(T_hu)=T_hf\,,\qquad\mbox{on $\R^n\times(0,\infty)$}\,,
  \\
  \nonumber
  &&T_hu(0)=\frac{u(t_0+h)-u(t_0)}h\,,\qquad\mbox{on $\R^n$}\,.
\end{eqnarray}
Setting
\[
u_{1,h}=S_t\Big[\frac{u(t_0+h)-u(t_0)}h\Big]\,,\qquad u_{2,h}=T_hu-u_{1,h}\,,
\]
we find that $u_{2,h}$ satisfies
\begin{equation}
  \label{PDE for u2h time}
  \pa_t u_{2,h}-2\,\Delta u_{2,h}=T_hf\,,\qquad\mbox{on $\R^n\times(0,\infty)$}\,,\qquad\mbox{$u_{2,h}(0)=0$ on $\R^n$}\,.
\end{equation}
By \cite[Corollary 7.31]{lieberman}, for every $(a,b)\subset(0,\infty)$ we have
\begin{eqnarray}\nonumber
  \int_a^b\,dt\int_{\R^n}|\nabla^2u_{2,h}(t)|^p+|\pa_tu_{2,h}(t)|^p\le
  C(p)\,\int_a^b\,dt\int_{\R^n}|T_hf(t)|^p\,.
\end{eqnarray}
Setting
\begin{equation}
  \label{that was g}
  g(t)=\Big\{-\frac{W''(u)}{\e^2}+\l_\e[u(t)]\,\frac{V''(u)}\e\Big\}\,v+\frac{V'(u)}\e\,\frac{d}{dt}\,\l_\e[u(t)]\,,
\end{equation}
by \eqref{final bound on lambda eps} and \eqref{final lambda eps Lip} we find that
\begin{eqnarray*}
\int_a^b\,dt\,\int_{\R^n}|T_hf(t)|^p\le\int_{a-|h|}^{b+|h|}dt\,\int_{\R^n}|g(t)|^p
\le C(\e,u_0,a)\,(b-a+2\,|h|)\,,
\end{eqnarray*}
so that, in summary,
\begin{eqnarray}\label{lieb x}
  \int_a^b\,dt\int_{\R^n}|\nabla^2u_{2,h}(t)|^p+|\pa_tu_{2,h}(t)|^p\le C(\e,p,u_0,a)\,(b-a+2\,|h|)\,.
\end{eqnarray}
At the same time, thanks to \eqref{final w22 estimate}, we have, for $|h|<t_0/2$ and $t>t_0$,
\begin{eqnarray*}
  \|\nabla^2 u_{1,h}(t)\|_{L^p}&\le&
  \|\nabla^2 G(t)\|_{L^1}\,\Big\|\frac{u(t_0+h)-u(t_0)}h\Big\|_{L^p}
  \\
  &\le&\frac{C}{t^{3/2}}\,\sup_{t_0-|h|<s<t_0+|h|}\|\pa_t u(s)\|_{L^p}\le \frac{C(\e,p,u_0)}{t_0^{3/2}}\,,
  \\
  \|\pa_t u_{1,h}(t)\|_{L^p}&\le&\|\pa_t G(t)\|_{L^1}\Big\|\frac{u(t_0+h)-u(t_0)}h\Big\|_{L^p}
  \le \frac{C(\e,p,u_0)}{t_0}\,,
\end{eqnarray*}
and thus, by $T_hu=u_{1,h}+u_{2,h}$ and \eqref{lieb x}
 \begin{eqnarray}
\label{lieb 2}
  \int_a^b\,dt\int_{\R^n}|\nabla^2(T_hu)(t)|^p+|\pa_t(T_hu)(t)|^p\le C(\e,p,u_0,a)\,(b-a+2\,|h|)\,.
\end{eqnarray}
Letting $h\to 0^+$ in \eqref{lieb 2} we obtain
\begin{eqnarray}
\nonumber
  \int_a^b\,dt\int_{\R^n}|\nabla^2(\pa_tu)(t)|^p+|\pa_{tt}u(t)|^p\le C(\e,p,u_0,a)\,(b-a)\,,
\end{eqnarray}
from which we easily conclude that, for every $t_0>0$ and $p\ge 2$,
\begin{equation}
  \sup_{t\ge t_0}\big\{\|\nabla^2(\pa_tu)(t)\|_{L^p},\|\pa_{tt}u(t)\|_{L^p}\big\}\le C(\e,p,u_0,t_0)\,.
\end{equation}
This last estimate, combined with \eqref{final w32 estimate} concludes the proof of \eqref{final hilbert estimate 2}.

\medskip

\noindent {\it Step five}: We prove the uniqueness of the solution $u$ of \eqref{diffused VPMCF} constructed in step seven. Indeed, let $v$ be another solution of \eqref{diffused VPMCF} with $v(0)=u_0$. By combining \eqref{eq lipschitz} with $\V(u(t))=\V(v(t))=\V(u_0)=1$ and with \eqref{final w22 estimate} with $p=2$ for both $u$ and $v$, we find that, for all $t>0$,
\begin{equation}\label{eq contraction lagrange}
	|\l_\e[u(t)]-\l_\e[v(t)]|\leq C(\e,u_0)\,\|u(t)-v(t)\|_{W^{1,2}}\,.
\end{equation}
Since $w=u-v$ satisfies $\pa_tw-2\,\Delta w=F[u]-F[v]$ with $w(0)=0$, by the Duhamel formula \eqref{duhamel formula}, for all $t\ge0$ we can represent $w(t)$ as
\begin{equation}
  \label{represent w}
  w(t)=\int_0^tS_{t-s}\big[F[u](s)-F[v](s)\big]\,ds\,.
\end{equation}
Notice that by using, in the order, \eqref{final bound on lambda eps}, $0\le V'(r)\le C\,r$ for $r\in[0,1]$, and \eqref{eq contraction lagrange}, we find
\begin{eqnarray*}
  |F[u](s)-F[v](s)|&\le& C(\e)\,\Big\{\Lip(W')+|\l_\e[u(s)]|\,\Lip(V')\Big\}\,|w(s)|
  \\
  &&+V'(v(s))\,|\l_\e[u(s)]-\l_\e[v(s)]|
  \\
  &\le& C(\e,u_0)\,\big\{|w(s)|+ \|w(s)\|_{W^{1,2}}\,|v(s)|\big\}\,,
\end{eqnarray*}
so that, by \eqref{final grad Lp estimate} with $p=2$ applied to $v$,
\begin{equation}
  \label{FusFvs}
  \|F[u](s)-F[v](s)\|_{L^2}\le C(\e,u_0)\,\|w(s)\|_{W^{1,2}}\,,\qquad\forall s>0\,.
\end{equation}
Setting
\[
a(t)=\sup_{0<s<t}\|w(s)\|_{W^{1,2}}
\]
and combining \eqref{FusFvs} with \eqref{gt L1} we deduce from \eqref{represent w} that, if $0<s<t$, then
\begin{eqnarray*}
  \|w(s)\|_{L^2}&\le&\int_0^s\,\|F[u](r)-F[v](r)\|_{L^2}\,dr\le  C(\e,u_0)\,\int_0^s \|w(r)\|_{W^{1,2}}\,dr
  \\
  &\le&C(\e,u_0)\,t\,a(t)\,,
  \\
  \|\nabla w(s)\|_{L^2}&\le&\sup_{0<r<s}\|F[u](r)-F[v](r)\|_{L^2}\,\int_0^s\|\nabla G(s-r)\|_{L^1}\,ds
  \\
  &\le& C(\e,u_0)\,\sqrt{s}\,\sup_{0<r<s}\|w(r)\|_{W^{1,2}}\le C(\e,u_0)\,\sqrt{t}\,a(t)\,.
\end{eqnarray*}
Combining this last two estimates, we find that $a(t)\le C(\e,u_0)\,\sqrt{t}\,a(t)$ for every $t>0$. In particular, setting $t_0=1/4\,C(\e,u_0)^2$, if $a(t)>0$ for some $t\in(0,t_0)$, then a contradiction follows. We have thus proved that $u(t)=v(t)$ on $\R^n$ for every $t\in(0,t_0)$. The argument can of course be iterated (using $t_0$ and then integer multiples of $t_0$ as initial times) to prove that $u(t)=v(t)$ on $\R^n$ for every $t>0$.

\medskip

\noindent {\it Conclusion}: Statement (i) was proved in \eqref{eq reg bounds uniform} and \eqref{eq reg bounds}. Statement (ii) was proved in \eqref{final bound on lambda eps} and \eqref{final lambda eps Lip}. Statement (iii) was proved in \eqref{volumeeee}. Moreover, we have
\begin{eqnarray*}
\frac{d}{dt}\AC(u(t))&=&\int_{\R^n}\Big\{2\,\e\,\nabla u\cdot\pa_t(\nabla u)+\frac{W'(u)}\e\,\pa_tu\Big\}(t)
\\
&=&-\int_{\R^n}\pa_t u(t)\,\Big\{2\,\e\,\Delta u-\frac{W'(u)}\e\Big\}(t)
\\
&=&-\int_{\R^n}\pa_t u(t)\,\Big\{\e\,\pa_tu+\l_\e[u(t)]\,V'(u)\Big\}(t)
\\
&=&-\e\,\int_{\R^n}(\pa_tu(t))^2+\l_\e[u(t)]\,\int_{\R^n}\pa_tu(t)\,V'(u(t))=- \e \, \int_{\R^n}(\pa_tu(t))^2\,,
\end{eqnarray*}
where we have used the regularity properties of $u$ to differentiate in time and to apply the divergence theorem, \eqref{diffused VPMCF}, and, in the last identity, statement (iii), to deduce that $0=(d/dt)\V(u(t))=\int_{\R^n}\pa_tu(t)\,V'(u(t))$. This proves statement (iv).

\medskip

We prove that $0<u<1$ on $\R^n\times(0,\infty)$, i.e. statement (v). Since $V, W \in C^{2}[0,1]$ with $V'(1)=W'(1)=0$ and $|\l_\e[u(t)]|\le C(\e,u_0)$ for all $t>0$, we can find a positive constant $K=K(\e,u_0)$ such that $r\mapsto K\,r-W'(r)+\e \lambda_\e[u(t)]\,V'(r)$ is strictly increasing on $[0,1]$. Correspondingly,
\[
K\,u-W'(u)+\e \lambda_\e[u(t)]\,V'(u)\leq K\,,\qquad\mbox{on $\R^n\times(0,\infty)$}\,,
\]
which, combined with \eqref{diffused VPMCF}, implies
\[
\e^2\partial_t u=2\,\e^2\,\Delta u-W'(u)+\e\,\l_\e[u(t)]\,V'(u)\le 2\,\e^2\,\Delta u-K\,u+K
\]
that is, $v=1-u$ is a non-negative solution of $\e^2(\pa_t-2\,\Delta)v+K\,v\ge0$ on $\R^n\times(0,\infty)$. By the strong maximum principle, either $v\equiv 0$ or $v>0$ on $\R^n\times(0,\infty)$, where the first option is excluded a priori since $V(1)=1$ and $\V(u(t))$ is finite for every $t>0$. We conclude that $u<1$ on $\R^n\times(0,\infty)$, and argue analogously for proving $u>0$ on $\R^n\times(0,\infty)$.

\medskip

We finally prove statement (vi). The argument of step eight shows that $v=\pa_tu$ satisfies $\pa_tv-2\,\Delta v=g$ for $g$ as in \eqref{that was g}. By testing this equation with $\pa_tu$ and integrating by parts,
\begin{eqnarray}\nonumber
\e^2\, \int_{\R^n} \big(\pa_tu\,\pa_{tt}u\big)(t)\!\!\!&=&\!\!\! - 2\,\e^2\,\int_{\R^n} |\nabla(\pa_tu)|^2 -\int_{\R^n}\Big\{W''(u)-\e\,\lambda_\e[u(t)]\,V''(u)\Big\}(t)\,(\pa_tu(t))^2
\\\label{eq bak0}
&&+\e\,\Big(\frac{d}{dt}\,\l_\e[u(t)]\Big)\int_{\R^n}V'(u(t))\,\pa_tu(t)\,,
\end{eqnarray}
where the last integral is equal to zero since $\V(u(t))=0$ for every $t>0$. This implies the validity of \eqref{dissipation formula}. Setting $b(t)=\int_{\R^n}(\pa_tu(t))^2$, \eqref{dissipation formula} combined with \eqref{eq reg bounds} implies that $b'\in L^1(t_0,\infty)$ for every $t_0>0$. Since \eqref{entropy derivative} and the monotonicity of $\AC(u(t))$ imply $b\in L^1(0,\infty)$, we have proved that $b\in W^{1,1}(t_0,\infty)$ for every $t_0>0$. This completes the proof of the theorem.
\end{proof}

\section{Subsequential bubbling resolution (Proof of Theorem \ref{thm bubbling general})} \label{section bubbling general}
\begin{proof}[Proof of Theorem \ref{thm bubbling general}] Let $u_0 \in W^{2,p}(\R^n;[0,1])$ for all $p\ge 2$, with $\V(u_0)=1$. By Theorem \ref{thm existence and regularity}, there is a unique solution $u$ to the diffused VPMCF \eqref{diffused VPMCF} with initial datum $u_0$, satisfying the various statements (i)--(vi) listed therein.

\medskip

Given $\{t_j\}_{j\in \N}$ with $t_j\to \infty$ as $j\to \infty$, we now want to prove that, up to extracting a subsequence, there are $M\in \N$, $\ell_\e>0$ such that
\begin{equation}
\label{thm21 leps utk to leps}
\ell_\e=\lim_{j\to \infty} \lambda_\e[u(t_j)]\,,
\end{equation}
sequences $\{x^i_j\}_{j\in \N}$ ($i=1,...,M$) satisfying $|x^i_j- x^k_j|\to \infty$ as $j\to\infty$ ($i\ne k$), and strictly radially decreasing solutions $\{\xi_i\}_{i=1}^M$ of
\begin{equation}\label{thm21 PDE xi}
	2\,\e^2\,\Delta \xi_i =W'(\xi_i)-\e\,\ell_\e\, V'(\xi_i)\quad\mbox{on $\R^n$}\,,
\end{equation}
with
\begin{eqnarray}
\label{thm21 xi vol and energy}
\sum_{i=1}^M \V(\xi_i)=1\,,\qquad \sum_{i=1}^M\AC(\xi_i) \leq \AC(u_0)\,.
\end{eqnarray}
such that, for all $p>2$,
\begin{eqnarray}
\label{thm21 utk limit to xii}
&&\lim_{j\to\infty}\Big\| u(t_j)- \sum_{i=1}^M \tau_{x_i^j}[\xi_i]\Big\|_{(W^{2,p}\cap W^{1,2})(\R^n)}=0\,.
\end{eqnarray}
This will be proved in step one through five. Finally, in step six, we shall prove that, if $\spt\,u_0\cc\R^n$, then $M=1$ and $x_j^1\to x_*$ as $j\to \infty$ for some $x_*\in\R^n$.

\medskip

\noindent {\it Step one}: We start by proving that, if $\xi \in W^{1,2}(\R^n;[0,1])\setminus \{0\}$ and $\ell\in\R$ satisfy
\begin{equation}\label{PDE uinfty}
	2\,\e ^2\,\Delta \xi = W'(\xi)-\e\,\ell\, V'(\xi)\quad \mbox{on $\R^n$},
\end{equation}
then $\ell>0$, $\xi$ is strictly radially decreasing with respect to some $x_0\in\R^n$, and
\begin{equation}\label{eq lower bound}
	\V(\xi) \geq \frac{1}{C\, \ell^n}\,.
\end{equation}
Indeed, $\xi\in W^{2,p}(\R^n)$ for every $p\ge 2$ thanks to the Calderon--Zygmund theorem, $\|\xi\|_{L^p}\le\|\xi\|_{L^2}<\infty$ (as $0\le \xi\le 1$ on $\R^n$), and $|W'(r)|,|V'(r)|\le C\,t$ for all $r\in[0,1]$. Since $\xi\in W^{2,p}(\R^n)$ we have enough regularity to test \eqref{PDE uinfty} with $\vphi=X\cdot\nabla \xi$ for $X(x)=\eta(x/R)\,x$, $\eta$ a cut-off function between $B_1$ and $B_2$, and $R>0$, and deduce that (see \cite{luckhausmodica89}, or \cite[Step five, Proof of Theorem 2.1]{maggi-restrepo}, for the details)
\begin{equation}\label{eq formulalagrange}
n\,\ell\,\V(\xi) =n\, \AC(\xi)-2\,\e\,\int_{\R^n} |\nabla \xi|^2\,.
\end{equation}
Since $\xi \neq 0$ and $n\ge 2$, the right-hand side of \eqref{eq formulalagrange} is strictly positive, thus proving that $\ell>0$. Moreover, since $\xi\in W^{2,p}(\R^n)$, we have that $\xi(x)\to 0$ as $|x|\to \infty$, so that the moving planes method of \cite{gidas1981symmetry} can be applied (see for example \cite[Theorem 6.2-(i)]{maggi-restrepo}) to deduce that $\xi$ is strictly radially decreasing with respect to some $x_0\in\R^n$.

\medskip

We finally prove \eqref{eq lower bound}. If $n\ge 3$, then we can combine \eqref{eq formulalagrange} with \eqref{modicamortola} to find
\[
n\,\ell\,\V(\xi)\ge (n-2)\,\AC(\xi)>(n-2)\,n\,\om_n^{1/n}\,\V(\xi)^{(n-1)/n}\,,
\]
which immediately gives \eqref{eq lower bound}. In the case $n=2$, we argue as follows. We test \eqref{PDE uinfty} with $\vphi_k\,\xi$ for $\vphi_k\in C^\infty_c(B_{k+2};[0,1])$ with $\vphi_k=1$ on $B_k$ and $|\nabla\vphi_k|\le 1_{B_{k+2}\setminus B_k}$ for each $k\in\N$. In this way we find
\[
2\,\e^2\,\int_{\R^n}\vphi_k\,|\nabla \xi|^2\le2\,\e^2\,\int_{\R^n}\xi\,|\nabla \xi|\,|\nabla\vphi_k|
+\int_{\R^n}\vphi_k\,\xi\,W'(\xi)+\e\,\ell\,\vphi_k\,\xi\,V'(\xi)\,.
\]
Letting $k\to\infty$ (and using monotone convergence for all the integrals but the one involving $W'$, which may be negative, and is dealt with by dominated convergence), we find that
\[
2\,\e^2\,\int_{\R^n}|\nabla \xi|^2\le\int_{\R^n}\xi\,W'(\xi)+\e\,\ell\,\xi\,V'(\xi)\,.
\]
Thanks to \eqref{W basic} there is $\de_0>0$ such that $W'(r)\le 0$ for all $r\in[1-\de_0,1]$ and $r\,|W'(r)|\le C\,r^2\le C\,W(r)$ for all $r\in[0,1]$. Therefore, $r\,W'(r)\le C\,W(r)$ for all $r\in[0,1]$, and, taking also into account that $r\,V'(r)\le C\,V(r)$ for all $r\in[0,1]$, we deduce that
\begin{equation}
  \label{comboo}
  2\,\e^2\,\int_{\R^n}|\nabla \xi|^2\le C\,\Big\{\e\,\ell\,\V(\xi)+\int_{\R^n}W(\xi)\Big\}\,.
\end{equation}
Since, when $n=2$, \eqref{eq formulalagrange} boils down to $\e\,\ell\,\V(\xi)=\int_{\R^n}W(\xi)$, we have finally proved $C\,\ell\,\V(\xi)\ge \AC(\xi)$. Again by \eqref{modicamortola}, as in the case $n\ge 3$, we deduce \eqref{eq lower bound}.

\medskip

\noindent {\it Step two}: We now begin the proof of \eqref{thm21 leps utk to leps}, \eqref{thm21 PDE xi}, \eqref{thm21 xi vol and energy}, and \eqref{thm21 utk limit to xii}  by discussing a compactness argument aimed at ``extracting one bubble'' from $u(t_j)$ as $t_j\to\infty$. In later steps we will of course discuss the iteration of this argument.

\medskip

We begin by recalling that, by Theorem \ref{thm existence and regularity}-(i,ii), for all $p\ge 2$ and $\a\in(0,1)$ we have
\begin{eqnarray}
  \label{reg1}
  &&\sup_{t\ge0}\big\{\|u(t)\|_{W^{2,p}(\R^n)},\|u(t)\|_{C^{1,\a}(\R^n)},|\l_\e[u(t)]|\big\}<\infty\,,
  \\
  \label{reg2}
  &&\sup_{t\ge1}\big\{\|u(t)\|_{W^{3,p}(\R^n)},\|u(t)\|_{C^{2,\a}(\R^n)}\big\}<\infty\,.
\end{eqnarray}
By \eqref{reg1}, $\sup_{\R^n}u(t)$ is achieved for every $t>0$, and we actually claim that
\begin{equation}
  \label{one point}
  \max_{\R^n}u(t)\ge\min\Big\{\frac12,\Big(\frac1{C\,\e\,\AC(u_0)}\Big)^{(n-1)/2}\Big\}\,,\qquad\forall t>0\,.
\end{equation}
Indeed, setting for brevity $\beta(t)=\max_{\R^n}u(t)$, if $\beta(t)\le 1/2$, then we can use the elementary estimate
\[
V(r)\le C\,r^{2n/(n-1)}\le C\,r^{2/(n-1)}\,r^2\le C\,r^{2/(n-1)}\,W(r)\,,\qquad\forall r\in[0,1/2]\,,
\]
to deduce from Theorem \ref{thm existence and regularity}-(iii,iv) that, for all $t>0$,
\begin{eqnarray*}
1&=&\V(u(t))\le C\,\int_{\R^n}u(t)^{2/(n-1)}\,W(u(t))\le C\,\beta(t)^{2/(n-1)}\,\e\,\AC(u(t))
\\
&\le&C\,\beta(t)^{2/(n-1)}\,\e\,\AC(u_0)\,.
\end{eqnarray*}
This proves \eqref{one point}. Combining \eqref{one point} with \eqref{reg1} we find $\beta_0=\beta_0(\e,u_0)\in(0,1/2]$ and $r_0=r_0(\e,u_0)>0$ such that for each $t>0$ there is $x_t\in\R^n$ with the property that
\begin{equation}
  \label{one point as a ball}
  u(t)\ge \beta_0\qquad\mbox{on $B_{r_0}(x_t)$}\,.
\end{equation}
In particular, given $t_j\to\infty$ as $j\to\infty$, then, thanks to \eqref{reg1}, \eqref{reg2} and \eqref{one point as a ball}, and up to extracting subsequences, we can find $\ell_\e\in\R$, $\xi_1\in\cap_{p\ge 2}W^{3,p}(\R^n;[0,1])\setminus\{0\}$, and a sequence $(x_j^1)_j$ in $\R^n$ such that, as $j\to\infty$,
\begin{eqnarray}
  \label{conv1}
  &&\lim_{j\to\infty}\l_\e[u(t_j)]=\ell_\e\,,
  \\
  &&\label{weak conv 1}
  \mbox{$\tau_{(-x_j^1)}[u(t_j)]\weak\xi_1$ weakly in $W^{3,p}(\R^n)$ as $j\to\infty$}\,,
  \\
  \label{conv2}
  &&\lim_{j\to\infty}\big\|\tau_{(-x_j^1)}[u(t_j)]-\xi_1\big\|_{W^{2,p}(B_R)}=0\,,\qquad\forall p\ge 2\,,\forall R>0\,.
\end{eqnarray}
By Theorem \ref{thm existence and regularity}-(vi), if we first take $t=t_j$ in
\begin{equation}
  \label{t eq}
  \pa_tu(t)=2\,\e^2\,\Delta u(t)-W'(u(t))+\e\,\l_\e[u(t)]\,V'(u(t))\quad\mbox{on $\R^n$}\,,
\end{equation}
and then let $j\to\infty$, we deduce that
\[
2\,\e^2\,\Delta\xi_1=W'(\xi_1)-\e\,\ell_\e\,V'(\xi_1)\quad\mbox{on $\R^n$}\,.
\]
By step one, we conclude that $\ell_\e>0$, that $\xi_1$ is strictly radially decreasing, and that
\begin{equation}\label{eq lower bound xi1}
	\V(\xi_1) \geq \frac{1}{C\, \ell_\e^n}\,.
\end{equation}

\medskip

\noindent {\it Step three}: We now iterate the construction of step two. Let us consider the following statement, depending on $k\in\N$:

\medskip

\noindent $(S)_k$: There are $\{(x^i_j)_j\}_{i=1}^k$ sequences in $\R^n$ with $|x^i_j-x_j^\ell|\to\infty$ as $j\to\infty$ ($i\ne\ell$), and there are $\{\xi_i\}_{i=1}^k$ radially symmetric decreasing functions solving \eqref{thm21 PDE xi}, such that
\begin{equation}
  \label{equality in}
  \sum_{i=1}^k\V(\xi_i)\le 1\,,
\end{equation}
and, for each $i=1,...,k$,
\begin{eqnarray}
&&\nonumber
  \mbox{$\tau_{(-x_j^i)}[u(t_j)]\weak\xi_i$ weakly in $W^{3,p}(\R^n)$ as $j\to\infty$}\,,
  \\
  \label{strong conv loc i}
  &&\lim_{j\to\infty}\big\|\tau_{(-x_j^i)}[u(t_j)]-\xi_i\big\|_{W^{2,p}(B_R)}=0\,,\qquad\forall R>0\,,\forall p\ge 2\,.
\end{eqnarray}

\medskip

Notice that if $(S)_k$ holds with equality in \eqref{equality in} for some $k\in\N$, then the theorem is proved (up to showing the validity of \eqref{thm21 utk limit to xii}) with $M=k$. We also notice that, in step two, we have proved that $(S)_1$ holds. In this step we prove that if $(S)_k$ holds for some $k\in\N$ with strict sign in \eqref{equality in}, that is, with
\[
\sum_{i=1}^k\V(\xi_i)<1\,,
\]
then, up to extracting a subsequence in $j$, we can find $(x_j^{k+1})_j$ and $\xi_{k+1}$ such that $(S)_{k+1}$ holds. Since, by step one, $\V(\xi_i)\ge(1/C\,\ell_\e)^n$ for each $i=1,...,k+1$, we conclude that, for some $M\in\N$, $(S)_M$ must hold with equality in \eqref{equality in}.

\medskip

So let us consider $k\in\N$ such that $(S)_k$ holds with strict sign in \eqref{equality in}, and let $m=\sum_{i=1}^k\V(\xi_i)$. Given $h\in\N$, let
\[
A_j^h=\bigcup_{i=1}^k B_h(x_j^i)\,.
\]
By \eqref{strong conv loc i}, for every $h\in\N$ there is $j(h)\in\N$ such that
\[
\V(u(t_j);B_h(x_j^i))\ge \V(\xi_i;B_h)-\frac{1-m}{2k}\,,\qquad\forall i=1,...,k,\forall j\ge j(h)\,,
\]
so that if, $j\ge j(h)$,
\begin{eqnarray*}
  1-m&\le&1-\sum_{i=1}^k\V(\xi_i;B_h)\le\frac{1-m}2+1-\sum_{i=1}^k\V\big(u(t_j);B_h(x_j^i)\big)
  \\
  &=&\frac{1-m}2+\V(u(t_j);\R^n\setminus A_j^h)\,,
\end{eqnarray*}
that is
\[
\frac{1-m}2\le\V(u(t_j);\R^n\setminus A_j^h)\,,\qquad\forall h\in\N\,,\forall j\ge j(h)\,,
\]
If we now set
\[
\beta_j^h=\max_{\R^n\setminus A_j^h}u(t_j)\,,
\]
then, either $\beta_j^h\ge 1/2$ or, by arguing as in step two, for every $h\in\N$ and $j\ge j(h)$,
\begin{eqnarray*}
\frac{1-m}2\le \V(u(t_j);\R^n\setminus A_j^h)
\le C\,(\beta_j^h)^{2/(n-1)}\,\int_{\R^n\setminus A_j^h}W(u(t_j))
\le C\,(\beta_j^h)^{2/(n-1)}\,\e\,\AC(u_0)\,,
\end{eqnarray*}
that is,  for every $h\in\N$ and $j\ge j(h)$,
\[
\beta_j^h
\ge\min\Big\{\frac12,\Big(\frac{1-m}{C\,\e\,\AC(u_0)}\Big)^{(n-1)/2}\Big\}=:\beta_*>0\,.
\]
In particular, up to extracting a subsequence in $j$, for each $j$ we can find $x_j^{k+1}\in\R^n$ such that
\[
\inf_{i=1,...,k}\,|x_j^{k+1}-x_j^i|\ge j\,,\qquad u_j(x_j^{k+1},t_j)\ge \beta_*\,,\qquad\forall j\,;
\]
and then, thanks to \eqref{reg1}, we can find $r_*>0$ such that
\[
u(t_j)\ge \frac{\beta_*}2\quad\mbox{on $B_{r_*}(x_j^{k+1})$}\,,\qquad\forall j\,.
\]
Hence, by \eqref{reg2}, there exists $\xi_{k+1}\in\cap_{p\ge 2}W^{3,p}(\R^n;[0,1])\setminus\{0\}$ such that
\begin{eqnarray}
  &&\label{weak conv k+1}
  \mbox{$\tau_{(-x_j^{k+1})}[u(t_j)]\weak\xi_{k+1}$ weakly in $W^{3,p}(\R^n)$ as $j\to\infty$}\,,
  \\
  \label{strong conv loc k+1}
  &&\lim_{j\to\infty}\big\|\tau_{(-x_j^{k+1})}[u(t_j)]-\xi_{k+1}\big\|_{W^{2,p}(B_R)}=0\,,\qquad\forall R>0\,,\forall p\ge 2\,,
\end{eqnarray}
and by the same argument of step two we see that $\xi_{k+1}$ satisfies
\[
2\,\e^2\,\Delta\xi_{k+1}=W'(\xi_{k+1})-\e\,\ell_\e\,V'(\xi_{k+1})\quad\mbox{on $\R^n$}\,,
\]
and thus
\begin{equation}\label{eq lower bound xi2}
	\V(\xi_{k+1}) \geq \frac{1}{C\, \ell_\e^n}\,.
\end{equation}
We have thus proved the desired induction step.

\medskip

\noindent {\it Step four}: We have so far proved the existence of $M\in\N$, $\ell_\e>0$, $\{(x_j^i)_j\}_{i=1}^M$ sequences in $\R^n$, and $\{\xi_i\}_{i=1}^M$ radially symmetric decreasing functions on $\R^n$ such that, up to extracting a subsequence, \eqref{thm21 leps utk to leps}, and \eqref{thm21 PDE xi} hold, together with part of \eqref{thm21 xi vol and energy} (i.e., $1=\sum_{i=1}^M\V(\xi_i)$), and
\begin{eqnarray}
  &&\nonumber
  \mbox{$\tau_{(-x_j^i)}[u(t_j)]\weak\xi_i$ weakly in $W^{3,p}(\R^n)$ as $j\to\infty$}\,,
  \\
  \label{conv conv}
  &&\lim_{j\to\infty}\big\|\tau_{(-x_j^i)}[u(t_j)]-\xi_i\big\|_{W^{2,p}(B_R)}=0\,,\qquad\forall R>0\,,p\ge 2\,,i=1,...,M\,.
\end{eqnarray}
To complete the proof of \eqref{thm21 xi vol and energy} we just need to notice that, by \eqref{conv conv}, for every $R>0$
\begin{eqnarray*}
\sum_{i=1}^M\AC(\xi_i;B_R)&\le&
\liminf_{j\to\infty}\sum_{i=1}^M\AC\big(\tau_{-x_j^i}[u(t_j)];B_R\big)
=
\liminf_{j\to\infty}\sum_{i=1}^M\AC\big(u(t_j);B_R(x_j^i)\big)
\\
&=&
\liminf_{j\to\infty}\AC\Big(u(t_j);\bigcup_{i=1}^MB_R(x_j^i)\Big)\le\AC(u_0)\,,
\end{eqnarray*}
where in the last equality we have used $|x_j^i-x_j^\ell|\to\infty$ as $j\to\infty$ if $i\ne\ell$, and in the last inequality we have used the monotonicity of $t\mapsto\AC(u(t))$. By arbitrariness of $R$, we have completed the proof of \eqref{thm21 xi vol and energy}. We are thus left to prove \eqref{thm21 utk limit to xii}. In this step, we shall prove a slightly weaker statement, namely
\begin{eqnarray}
\label{thm21 utk limit to xii p2}
&&\lim_{j\to\infty}\Big\| u(t_j)- \sum_{i=1}^M \tau_{x_i^j}[\xi_i]\Big\|_{W^{2,p}(\R^n)} =0\,,\qquad\forall p> 2\,.
\end{eqnarray}
We will then improve \eqref{thm21 utk limit to xii p2} to \eqref{thm21 utk limit to xii} in step five. (This improvement is of crucial importance in discussing convergence to equilibrium of the flow.) We begin the proof of \eqref{thm21 utk limit to xii p2} by showing that, with $n'=n/(n-1)$,
\begin{equation}
  \label{l2nprime conv}
  \lim_{j\to\infty}\|u(t_j)-v_j\|_{L^{2n'}(\R^n)}=0\,,\qquad\mbox{where}\,\, v_j=\sum_{i=1}^M\tau_{x_j^i}[\xi_i]\,.
\end{equation}
Indeed, for every $\s>0$ we can find $R>0$ such that
\begin{equation}
  \label{choice of R}
  \V(\xi_i;B_R)\ge (1-\s)\,\V(\xi_i)\,,\qquad\forall i=1,...,M\,.
\end{equation}
Setting $A_j^R=\R^n\setminus\cup_{i=1}^MB_R(x_j^i)$, for $j$ large enough to have $B_R(x_j^i)\cap B_R(x_j^\ell)=\emptyset$ ($i\ne\ell$), we  have
\begin{eqnarray}\nonumber
\frac1{C(M)}\,\int_{\R^n}|u(t_j)-v_j|^{2n'}\!\!&\le&\!\! \int_{ A_j^R}|u(t_j)|^{2n'}+|v_j|^{2n'}
+\sum_{i=1}^M\int_{B_R(x_j^i)}|u(t_j)-\tau_{x_j^i}[\xi_i]|^{2n'}
\\\label{fine0}
&&\!\!+\sum_{i=1}^M\int_{B_R(x_j^i)}\Big|\sum_{\ell\ne i}\tau_{x_j^\ell}[\xi_\ell]\Big|^{2n'}\,.
\end{eqnarray}
We first notice that, since $r^{2n'}\le C\,V(r)$ for all $r\in[0,1]$,,
\begin{eqnarray}\nonumber
\int_{A_j^R}|v_j|^{2n'}&\le&C(M)\,\sum_{i=1}^M\int_{A_j^R}|\tau_{x_j^i}[\xi_i]|^{2n'}
\le C(M)\,\sum_{i=1}^M\int_{\R^n\setminus B_R(x_j^i)}|\tau_{x_j^i}[\xi_i]|^{2n'}
\\\label{fine1}
&\le& C(M)\,\sum_{i=1}^M\int_{\R^n\setminus B_R}V(\xi_i)\le C(M)\,\s\,\sum_{i=1}^M\V(\xi_i)=C(M)\,\s;
\end{eqnarray}
similarly,
\begin{eqnarray*}
\frac1{C}\int_{A_j^R}|u(t_j)|^{2n'}&\le&\int_{A_j^R}V(u(t_j))=1-\sum_{i=1}^M\V(u(t_j);B_R(x_j^i))
\\
&=&\sum_{i=1}^M\V(\xi_i)-\V(u(t_j);B_R(x_j^i))
=\sum_{i=1}^M\V(\xi_i)-\V(\tau_{-x_j^i}[u(t_j)];B_R)
\end{eqnarray*}
so that, by \eqref{conv conv},
\begin{equation}
  \label{fine2}
  \limsup_{j\to\infty}\int_{A_j^R}|u(t_j)|^{2n'}\le C\,\sum_{i=1}^M \V(\xi_i;\R^n\setminus B_R)\le C\,M\,\s\,.
\end{equation}
Coming to the last term in \eqref{fine0} we see that
\begin{eqnarray}
\label{fine3}
  &&\sum_{i=1}^M\int_{B_R(x_j^i)}\Big|\sum_{\ell\ne i}\tau_{x_j^\ell}[\xi_\ell]\Big|^{2n'}
\le C(M)\,\sum_{i=1}^M\sum_{\ell\ne i}\int_{B_R(x_j^i)}\big|\tau_{x_j^\ell}[\xi_\ell]\big|^{2n'}
\\\nonumber
&&
\le C(M)\,\sum_{i=1}^M\sum_{\ell\ne i}\int_{B_R(x_j^i-x_j^\ell)}\big|\xi_\ell\big|^{2n'}\,,
\end{eqnarray}
where the last integral converges to zero since $|x_j^i-x_j^\ell|\to 0$ as $j\to\infty$ by $i\ne\ell$, and since $\xi_\ell\in L^{2n'}(\R^n)$. Finally, since $2n'\ge2$ we can use \eqref{conv conv} to address the second integral in \eqref{fine0}, and conclude from \eqref{fine0}, \eqref{fine1}, \eqref{fine2} and \eqref{fine3} that
\[
\limsup_{j\to\infty}\|u(t_j)-v_j\|_{L^{2n'}(\R^n)}\le C(M)\,\s\,,\qquad\forall\s>0\,,
\]
that is \eqref{l2nprime conv}. Now, by \eqref{eq l2 bound} we have that
\[
\|u(t_j)-v_j\|_{L^2(\R^n)}\le C(\e)\,\Big\{\AC(u(t_j))+\V(u(t_j))+\sum_{i=1}M\AC(\xi_i)+\V(\xi_i)\Big\}\le C(\e,u_0)\,,
\]
so that \eqref{l2nprime conv} implies
\begin{equation}
  \label{con lp}
  \lim_{j\to\infty}\|u(t_j)-v_j\|_{L^p(\R^n)}=0\,,\qquad\forall p>2\,.
\end{equation}
We finally recall that given $k,k_1,k_2\in\N$ and $p,p_1,p_2\in[1,\infty]$ related by $k=\theta\,k_1+(1-\theta)\,k_2$ and $(1/p)=(\theta/p_1)+(1-\theta)/p_2$ for some $\theta\in(0,1)$, then we have
\[
\|f\|_{W^{k,p}(\R^n)}\le C\,\|f\|_{W^{k_1,p_1}(\R^n)}^\theta\,\|f\|_{W^{k_2,p_2}(\R^n)}^{1-\theta}
\]
for every $f\in W^{k_1,p_1}(\R^n)\cap W^{k_2,p_2}(\R^n)$ and with a constant $C$ depending only on $n$, $p_1$, $p_2$, $k_1$ and $k_2$. By combining this interpolation inequality (with $k=1,2$, $k_1=0$, $k_2=3$, any $p_1>2$ and $p_2\ge2$) with \eqref{con lp} and the uniform $W^{3,p}$-bounds satisfied by $u(t_j)$ (recall \eqref{reg2}) and each of the $\xi_i$'s, we conclude the proof of \eqref{thm21 utk limit to xii p2}.

\medskip

\noindent {\it Step five}: To complete the proof of \eqref{thm21 utk limit to xii} we need to prove the $W^{1,2}$-convergence stated in \eqref{thm21 utk limit to xii p2}, that is
\begin{equation}
  \label{that is w12}
  \lim_{j\to\infty}\Big\|u(t_j)-\sum_{i=1}^M\tau_{x_j^i}[\xi_i]\Big\|_{W^{1,2}(\R^n)}=0\,.
\end{equation}
To this end we set, for the sake of brevity,
\[
u_j=u(t_j)\,,\qquad v_j=\sum_{i=1}^M\tau_{x_j^i}[\xi_i]\,,\qquad w_j=u_j-v_j=u(t_j)-\sum_{i=1}^M\tau_{x_j^i}[\xi_i]\,,
\]
as well as
\[
Z(r)=W'(r)-\e\,\ell_\e\,V'(r)\,,\qquad\forall r\in[0,1]\,.
\]
If we multiply by $w_j$ in $\e^2(\pa_tu_j-2\,\Delta u_j)=-Z(u_j)+\e\,(\ell_\e-\l_\e[u_j])\,V'(u_j)$ and in $2\,\e^2\,\Delta v_j=\sum_{i=1}^MZ(\tau_{x_j^i}[\xi_i])$, and we then add the resulting identities, we obtain
\begin{eqnarray}\label{PDE wj}
  -2\,\e^2\,w_j\,\Delta w_j&=&-\e^2\,w_j\,\pa_tu_j+\e\,(\ell_\e-\l_\e[u_j])\,w_j\,V'(u_j)
  \\\nonumber
  &&+w_j\Big\{-Z(u_j)+\sum_{i=1}^MZ(\tau_{x_j^i}[\xi_i])\Big\}
\end{eqnarray}
By \eqref{eq l2 bound}, $\V(u_j)=1$, $\AC(u_j)\le\AC(u_0)$, $\V(\xi_i)\le 1$, and $\AC(\xi_i)\le\AC(u_0)$ we have $\|w_j\|_{L^2(\R^n)}\le C(\e,u_0)$, so that $|w_j|\le M+1$ on $\R^n$ and \eqref{con lp} imply
\begin{equation}
  \label{wj in Lp}
  \|w_j\|_{L^2(\R^n)}\le C(\e,u_0,M)\,,\qquad \lim_{j\to\infty}\|w_j\|_{L^p(\R^n)}=0\,,\qquad\forall p>2\,.
\end{equation}
Combining \eqref{wj in Lp} and $0\le V'(u_j)\le C\,u_j^{2n'-1}$ with \eqref{PDE wj} and the H\"older inequality, we thus find
\begin{eqnarray}
  \nonumber
  \int_{\R^n}|\nabla w_j|^2&\le&C(\e,u_0,M)\,\|\pa_tu_j\|_{L^2(\R^n)}+C(\e,u_0)\,\big|\ell_\e-\l_\e[u_j]\big|\,\|w_j\|_{L^{2n'}(\R^n)}
  \\\label{w12 proof 1}
  &&+\int_{\R^n}w_j\,\Big\{-Z(u_j)+\sum_{i=1}^MZ(\tau_{x_j^i}[\xi_i])\Big\}\,.
\end{eqnarray}
where we have also used $\|u(t)\|_{L^{2n'}(\R^n)}\le C(\e,u_0)$ for all $t>0$. The first two terms on the right-hand side of \eqref{w12 proof 1} converge to zero as $j\to\infty$, respectively, thanks to Theorem \ref{thm existence and regularity}-(vi) and to \eqref{thm21 leps utk to leps} and \eqref{wj in Lp}. To deal with the third term on the right hand side of \eqref{w12 proof 1}, we consider $\s>0$ and pick $R>0$ as in \eqref{choice of R}, so that, as proved in \eqref{fine1} and \eqref{fine2}, we have
\begin{equation}
  \label{recall}
  \int_{A_j^R}\sum_{i=1}^M\tau_{x_j^i}[\xi_i]^{2n'}\le C(M)\,\s\,,\qquad
  \limsup_{j\to\infty}\int_{A_j^R}u_j^{2n'}\le C(M)\,\s\,,
\end{equation}
where $A_j^R=\R^n\setminus\cup_{i=1}^MB_R(x_j^i)$. Moreover, up to further increase the value of $R$ we can ensure that
\begin{equation}
  \label{xi in l2 out BR}
  \max_{i=1,...,M}\|\xi_i\|_{L^p(\R^n)}\le C(p)\,\s\,,\qquad\forall p\ge 2\,,
\end{equation}
as well as
\begin{equation}
  \label{altra}
  \mbox{$|\xi_i|\le \k(\e)$ on $\R^n\setminus B_R$, where $\k(\e)>0$ is s.t. $Z'(r)\ge \k(\e)$ for all $r\in[0,\k]$}\,.
\end{equation}
(Notice that the existence of $\k(\e)$ is guaranteed by $W''(0)>0$, $V''(0)=0$, and $|\ell_\e|\le C(u_0)$, which in turn follows from $\sup_j|\l_\e[u_j]|\le C(u_0)$ (recall Theorem \ref{thm existence and regularity}-(ii)) and \eqref{thm21 leps utk to leps}.

\medskip

Now, provided $j$ is large enough depending on $R$ (thus on $\sigma$), $\{A_j^R,B_R(x_i^j)\}_{i=1}^M$ is a partition of $\R^n$, that we can use to decompose the third term on the right hand side of \eqref{w12 proof 1}. We begin by working with the integration over $A_j^R$, that we decompose by assigning a privileged role to $\xi_1$, and writing
\begin{eqnarray}\nonumber
&&\int_{A_j^R}w_j\,\Big\{-Z(u_j)+\sum_{i=1}^MZ(\tau_{x_j^i}[\xi_i])\Big\}
=\int_{A_j^R}\big(u_j-\tau_{x_j^1}[\xi_1]\big)\,\big(Z(\tau_{x_j^1}[\xi_1])-Z(u_j)\big)
\\\nonumber
&&-\int_{A_j^R}\sum_{\ell=2}^M\,\tau_{x_j^\ell}[\xi_\ell]\,\Big\{-Z(u_j)+\sum_{i=1}^MZ(\tau_{x_j^i}[\xi_i])\Big\}
-\int_{A_j^R}\big(u_j-\tau_{x_j^1}[\xi_1]\big)\sum_{\ell=2}^M\,Z(\tau_{x_j^\ell}[\xi_\ell])
\\\label{di3}
&&\le\int_{A_j^R}\big(u_j-\tau_{x_j^1}[\xi_1]\big)\,\big(Z(\tau_{x_j^1}[\xi_1])-Z(u_j)\big)
\\\nonumber
&&+C(M)\,\max_{1\le i\le M}\|\xi_i\|_{L^2(\R^n\setminus B_R)}\,\max_{1\le i\le M}\big\{\|u_j\|_{L^2(\R^n)},\|\xi_i\|_{L^2(\R^n)}\big\}\,,
\end{eqnarray}
where we have used $|Z(r)|\le C\,r$ for all $r\in[0,1]$. Now by the properties of $W$ and $V$ we have
\begin{equation}
  \label{holder Z}
  Z(r)-Z(s)=Z'(s)\,(r-s)+{\rm O}(r-s)^{2n'-1}\,,\qquad\forall r,s\in[0,1]\,,
\end{equation}
so that by plugging $s=\xi_1$ and $r=u_j$ in \eqref{holder Z}, and by recalling that, thanks to \eqref{altra}, $Z'(\tau_{x_j^1}[\xi_1])\ge\k(\e)$ on $A_j^R$,
\begin{eqnarray}\label{di1}
&&\int_{A_j^R}\big(u_j-\tau_{x_j^1}[\xi_1]\big)\,\big(Z(\tau_{x_j^1}[\xi_1])-Z(u_j)\big)
\\\nonumber
&&\le-\int_{A_j^R}\big(u_j-\tau_{x_j^1}[\xi_1]\big)^2\,Z'(\tau_{x_j^1}[\xi_1])+
C\,\int_{A_j^R}|u_j-\tau_{x_j^1}[\xi_1]|^{2n'}
\\\nonumber
&&\le-\k(\e)\,\int_{A_j^R}\big(u_j-\tau_{x_j^1}[\xi_1]\big)^2+C\,\int_{A_j^R}|u_j-\tau_{x_j^1}[\xi_1]|^{2n'}\,.
\end{eqnarray}
Now, for every $p\ge 2$,
\begin{eqnarray}\label{di2}
\Big|\int_{A_j^R}|u_j-\tau_{x_j^1}[\xi_1]|^p-\int_{A_j^R}|w_j|^p\Big|
\le C(p)\,\sum_{i=2}^M\int_{A_j^R}\tau_{x_j^i}[\xi_i]^p\le C(p)\,\sum_{i=2}^M\|\xi_i\|_{L^p(\R^n\setminus B_R)}\,,
\end{eqnarray}
so that, combining \eqref{di1} and \eqref{di2} with \eqref{di3}, we finally conclude that
\begin{eqnarray*}
\int_{A_j^R}w_j\,\Big\{-Z(u_j)+\sum_{i=1}^MZ(\tau_{x_j^i}[\xi_i])\Big\}
\le   -\k(\e)\,\int_{A_j^R}w_j^2+ C(p,M)\,\s\,.
\end{eqnarray*}
Going back to \eqref{w12 proof 1} we have proved that
\begin{eqnarray}
  \label{w12 proof 2}
  \int_{\R^n}|\nabla w_j|^2+\k(\e)\,\int_{A_j^R}w_j^2\!\!&\le&\!\!
  \sum_{i=1}^M\int_{B_R(x_j^i)}w_j\,\Big\{-Z(u_j)+\sum_{i=1}^MZ(\tau_{x_j^i}[\xi_i])\Big\}
  \\\nonumber
  &&+\,{\rm o}_j+C(M)\,\s
\end{eqnarray}
where ${\rm o}_j\to 0$ as $j\to\infty$. Next we observe that
\begin{eqnarray*}
&&\int_{B_R(x_j^1)}w_j\,\Big\{-Z(u_j)+\sum_{i=1}^MZ(\tau_{x_j^i}[\xi_i])\Big\}
\\
&&\le C\, \int_{B_R(x_j^1)}\Big\{|u_j-\tau_{x_j^1}[\xi_1]|+\sum_{i=2}^M\tau_{x_j^1}[\xi_1]\Big\}\,\Big\{u_j+\sum_{k=1}^M|\tau_{x_j^k}[\xi_k]\Big\}
\\
&&\le C\, \max_{1\le i\le M}\big\{\|u_j\|_{L^2},\|\xi_i\|_{L^2}^2\big\}\Big\{ \int_{B_R(x_j^1)}|u_j-\tau_{x_j^1}[\xi_1]|^2+\sum_{k=2}^M\int_{B_R(x_j^1)}\tau_{x_j^k}[\xi_k]^2\Big\}
\\
&&\le C(\e,u_0,M)\,\Big\{\int_{B_R}|\tau_{-x_j^1}[u_j]-\xi_1|^2+\sum_{k=2}^M\int_{B_R(x_j^1-x_j^k)}\xi_k^2 \Big\}\,,
\end{eqnarray*}
where the last two integrals converge to zero as $j\to\infty$ thanks to \eqref{conv conv}, $\xi_k\in L^2$ and $|x_j^1-x_j^k|\to\infty$ as $j\to\infty$ if $k\ne 1$. Of course we can repeat this same argument with any other $B_R(x_j^i)$ with $i=2,...,M$, and go back to \eqref{w12 proof 2} to conclude that
\begin{eqnarray}
  \label{w12 proof 3}
  \int_{\R^n}|\nabla w_j|^2+\k(\e)\,\int_{A_j^R}w_j^2\le \,{\rm o}_j+C(M)\,\s\,.
\end{eqnarray}
We finally conclude the proof of \eqref{that is w12} since, for each $i=1,...,M$, we have
\begin{eqnarray*}
\int_{B_R(x_j^i)}w_j^2&\le&2\,\int_{B_R(x_j^i)}|u_j-\tau_{x_j^i}[\xi_i]|^2+2\,\sum_{k\ne i}\int_{B_R(x_j^i)}\tau_{x_j^k}[\xi_k]^2
\\
&=&2\,\|\tau_{-x_j^i}[u_j]-\xi_i\|_{L^2(B_R)}^2+2\,\sum_{k\ne i}\|\xi_k\|_{L^2(B_R(x_j^i-x_j^k))}^2\,,
\end{eqnarray*}
where the right hand side converges to zero as $j\to\infty$ by the same arguments used in the proof of \eqref{w12 proof 3}. This concludes the proof of \eqref{that is w12}.

\medskip

\noindent {\it Step six}: We finally prove that if $\spt\,u_0\cc\R^n$, then $M=1$. We use the moving plane method, following an argument in \cite[Theorem 1.2]{feireisl1997long}. Let $s_0>0$ be such that $\spt\,u_0\subset B_{s_0}(0)$. For $s\ge s_0$, let $H_s=\{x\in\R^{n+1}:x_1>s\}$, and define
\[
\rho_s[v](x)=v\big(2s-x_1,x_2,...,x_{n+1}\big)\,,\qquad\forall x\in\R^{n+1}\,,v:\R^{n+1}\to\R\,.
\]
In this way, $u_s(t)=\rho_s[u(t)]$ solves
\[
\e^2\,\big(\pa_t u_s-2\,\Delta u_s\big)=-W'(u_s)+\e\,\l_\e[u_s(t)]\,V'(u_s)\,,\qquad\mbox{on $\R^n\times(0,\infty)$}\,,
\]
with initial datum $u_s(0)=\rho_s[u(0)]=\rho_s[u_0]$. Since, trivially, $\l_\e[u_s(t)]=\l_\e[u(t)]$ for all $t>0$, we have that $v_s=u_s-u$ solves
\[
\e^2\,\big(\pa_t v_s-2\,\Delta v_s\big)=h(x,t)\,,\qquad\mbox{on $\R^n\times(0,\infty)$}\,,
\]
with $v_s(0)=\rho_s[u_0]-u_0$ and
\[
h(x,t)=W'(u)-W'(u_s)+\e\,\l_\e[u(t)]\,(V'(u_s)-V'(u))=c(x,t)\,v_s
\]
for some $c\in L^\infty(\R^n\times(0,\infty))$. Since $\spt\,u_0\subset B_{s_0}(0)$ and $s>s_0$ imply that $v_s(0)\ge0$ on $H_s$, by the parabolic maximum principle \cite[Lemma 2.3]{lieberman} we deduce that $v_s(t)\ge0$ on $H_s$ for all $t>0$. In particular, if $h>s$ and $(x_2,...,x_{n+1})\in\R^n$, then the non-negativity of $v_s(t)$ at $(h,x_2,...,x_{n+1})\in H_s$ implies that
\begin{eqnarray*}
  0\!\!&\le&\!\!\lim_{h\to s^+}\frac{v_s(h,x_2,...,x_{n+1},t)}{h-s}
  \\
  &=&\!\!
\lim_{h\to s^+}\frac{u(2s-h,x_2,...,x_{n+1},t)-u(h,x_2,...,x_{n+1},t)}{h-s}=-e_1\cdot\nabla u(s,x_2,...,x_{n+1},t)\,.
\end{eqnarray*}
We have thus proved that $e_1\cdot\nabla u(t)\le 0$ on $H_s$ for every $s>s_0$. By arbitrariness of $s>s_0$ and of the choice of the direction with respect to which we reflect $u(t)$, we conclude that
\begin{equation}
  \label{also seen}
\frac{x}{|x|}\cdot \nabla u(x,t)\le 0\,,\qquad\forall x\in\R^{n+1}\setminus B_{s_0}(0)\,,t>0\,.
\end{equation}
By \eqref{reg2}, for every $t>0$, $u(x,t)\to 0$ as $|x|\to\infty$. Now pick a sequence $t_j\to\infty$ such that \eqref{thm21 utk limit to xii p2} holds. Since $m=\max_{1\le i\le M}\sup_{\R^n}\xi_i>0$, if $M\ge 2$, then  the fact that for $i\ne k$ we have $|x_j^i-x_j^k|\to\infty$ as $j\to\infty$, combined with \eqref{thm21 utk limit to xii p2}, implies that, for every $R>0$, there is $j(R)\in\N$ such that if $j\ge j(R)$ then
\[
\sup_{\R^{n+1}\setminus B_R}u(t_j)\ge\frac{m}2\,.
\]
This leads to a contradiction with $u(x,t_j)\to 0$ as $|x|\to\infty$. Therefore $M=1$ and, by a similar argument, \eqref{also seen} is also seen to imply the boundedness of $(x_j^1)_j$, and thus its convergence modulo a further subsequence extraction.
\end{proof}

\section{Subsequential bubbling into diffused balls (Proof of Theorem \ref{theorem bubbling one})}\label{section theorem bubbling one}
\begin{proof}
[Proof of Theorem \ref{theorem bubbling one}] We start by noticing that, thanks to \eqref{isop lb}, we have
\[
N\,\Psi\Big(\e,\frac1N\Big)\ge 2\,n\,\om_n^{1/n}\,N^{1/n}\,,\qquad\forall N>0\,.
\]
In particular, if $M_0$ is the least integer such that $n\,\om_n^{1/n}\,M_0^{1/n}>\int_{\R^n}|\nabla u_0|$, then by
$\AC(u_0)\to 2\,\int_{\R^n}|\nabla u_0|$ we can find $\e_*>0$ (depending on $u_0$) such that
\begin{equation}
\label{ACu0 less than M0 balls proof}
M_0\,\Psi\Big(\e, \frac{1}{M_0}\Big)\ge \AC(u_0)\,,\qquad\forall \e\in(0,\e_*)\,.
\end{equation}
We shall prove the theorem for every solution $u(t)$ of \eqref{diffused VPMCF} corresponding to $\e<\e_0^*$, where
\begin{equation}
  \label{e0star}
\e_0^*:=\min\Big\{\e_*,\frac{\s_0}{M_0^{1/n}}\Big\}\,,
\end{equation}
for some suitably small universal constant $\s_0$ to be the determined below.

\medskip

\noindent {\it Step one}: We prove the existence of a universal constant $\s_0$ such that
\[
\Psi(\s,1)\ge\Psi(\e,1)\,,\qquad\forall \e<\min\{\s_0,\s\}\,.
\]
We shall take $\s_0\le\e_0$, so that, thanks to \eqref{Psi strict increasing in eps}, we will be able to focus directly on the case when $\e<\s_0<\s$. Should the the claim fail in this case, we could then find sequences $(\e_j)_j$ and $(\s_j)_j$ with $\e_j\to 0^+$ as $j\to\infty$ and
\begin{equation}
  \label{sj ej}
\s_j>\e_j\,,\qquad \Psi(\s_j,1)<\Psi(\e_j,1)\,,\qquad\forall j\in\N\,.
\end{equation}
For each $j$ we denote by  $u_j$ a minimizer of $\Psi(\s_j,1)$, and notice that, by \eqref{Psi limit as eps to zero} and up to extracting a subsequence in $j$, there is $\s_*\in[0,\infty]$ such that
\begin{equation}
  \label{psi sj 1}
\lim_{j\to\infty}\s_j=\s_*\in[0,\infty]\,,\qquad \lim_{j\to\infty}\Psi(\s_j,1)=\lim_{j\to\infty}\ac_{\s_j}(u_j)=2\,c_{\rm iso}(n)\,.
\end{equation}
We can immediately exclude that $\s_*=0$ thanks to \eqref{Psi strict increasing in eps} and \eqref{sj ej}. If $\s_*\in(0,\infty)$, then a minor modification of the elementary compactness argument of \cite[Proof of Theorem A.1, steps one and two]{maggi2023hierarchy} shows that, up to extracting subsequences, there is a minimizer $u_*$ in $\Psi(\s_*,1)$ such that $u_j\to u_*$ in $L^1_{\rm loc}(\R^n)$. In particular, by combining \eqref{Psi isop lower bound} with \eqref{psi sj 1} we find that
\[
2\,c_{\rm iso}(n)<\Psi(\s_*,1)=\ac_{\s_*}(u_*)\le\liminf_{j\to\infty}\ac_{\s_j}(u_j)=\lim_{j\to\infty}\Psi(\s_j,1)=2\,c_{\rm iso}(n)\,,
\]
a contradiction. We are thus left with the possibility that $\s_*=+\infty$. To obtain a contradiction, we notice that, by $1=\V(u_j)=\|\Phi(u_j)\|_{L^{n/(n-1)}(\R^n)}$, we have
\[
c_{\rm iso}(n)\le|D[\Phi(u_j)]|(\R^n)<\frac{\ac_{\s_j}(u_j)}2\to c_{\rm iso}(n)\,,\qquad\mbox{as $j\to\infty$}.
\]
In particular, by a standard compactness argument (see, e.g., \cite[Theorem A.1]{fmpraBV})), for some $a,r>0$ and up to extracting a subsequence, $\Phi(u_j)\to a\,1_{B_r}$ in $L^{n/(n-1)}(\R^n)$ as $j\to\infty$. Setting $b=\Phi^{-1}(a)$ we have $u_j\to b\,1_{B_r}$ in $L^1_{\rm loc}(\R^n)$, so that, for every $\vphi\in C^\infty_c(\R^n)$,
\[
\int_{\R^n}\,b\,1_{B_r}\,\nabla\vphi=\lim_{j\to\infty}\int_{\R^n}\,u_j\,\nabla\vphi\le\|\vphi\|_{L^2(\R^n)}\,\|\nabla u_j\|_{L^2(\R^n)}
\]
where $\|\nabla u_j\|_{L^2(\R^n)}\le \ac_{\s_j}(u_j)/\s_j\to 0$ as $j\to\infty$ thanks to $\s_j\to+\infty$. In summary, $|D(b\,1_{B_r})|(\R^n)=0$, so that $0=b=\Phi^{-1}(a)$, and thus $a=\Phi(0)=0$, against $a>0$.

\medskip

\noindent {\it Step two}: Given a sequence $t_j\to\infty$ as $j\to\infty$, we want to show that, up to extracting a subsequence, there is $M \in \N$ (with $M\leq M_0$) and sequences $\{(x_j^i)_j\}_{i=1}^M$ with $|x_j^i-x_j^k|\to\infty$ as $j\to\infty$ ($i\ne k$), such that, for all $p>2$,
\begin{equation}
\label{convergence to bubbles M0 proof}
\lim_{j\to\infty}\Big\| u(t_j)-  \sum_{i=1}^M \tau_{x_j^i}\big[\zeta_{\e,1/M}\big]\Big\|_{(W^{2,p}\cap W^{1,2})(\R^n)}=0\,.
\end{equation}
By Theorem \ref{thm bubbling general}, we know that \eqref{convergence to bubbles M0 proof} holds if in place of $\zeta_{\e,1/M}$ we have some $\xi_i$, $i=1,...,M$, solving, for a same $\ell_\e>0$,
\begin{equation}
  \label{bella xi}
2\,\e^2\,\Delta \xi_i = W'(\xi_i)-\e\,\ell_\e\, V'(\xi_i)\quad\mbox{on $\R^n$},
\end{equation}
with $\AC(u_0)\ge\sum_{i=1}^M\AC(\xi_i)$ and $1=\sum_{i=1}^M\V(\xi_i)$. We set for brevity $m_i=\V(\xi_i)$ and assume without loss of generality that $m_i\ge m_{i+1}$. Our goal is thus proving the existence of $z_i\in\R^n$ such that $\xi_i=\tau_{z_i}[\zeta_{\e,1/M}]$ for each $i$.

\medskip

We first notice that it must be $m_1\ge 1/M_0$. Indeed, should this not be the case, then using also step one and the fact that $\e<\e_0^*$ implies $\e\,M_0^{1/n}<\s_0$, we would have
\[
\frac1{m_i^{1/n}}> M_0^{1/n}\,,\qquad \Psi\Big(\e\big/m_i^{1/n},1\Big)\ge\Psi\Big(\e\,M_0^{1/n},1\Big)\,,\qquad\forall i\,.
\]
We could then combine these inequalities with $\AC(u_0)\ge\sum_{i=1}^M\AC(\xi_i)$ and \eqref{Psi scaling} to find that
\begin{eqnarray*}
\AC(u_0)&\ge&\sum_{i=1}^M\,\Psi(\e,m_i)=\sum_{i=1}^Mm_i^{(n-1)/n}\,\Psi\Big(\e\big/m_i^{1/n},1\Big)
\\
&>&M_0^{1/n}\,\sum_{i=1}^Mm_i\,\Psi\Big(\e\big/m_i^{1/n},1\Big)\ge M_0^{1/n}\,\Psi(\e\,M_0^{1/n},1)\,\sum_{i=1}^Mm_i
\\
&=&M_0^{1/n}\,\Psi(\e\,M_0^{1/n},1)=M_0^{1/n}\,M_0^{(n-1)/n}\,\Psi(\e,1/M_0)=M_0\,\Psi(\e,1/M_0)\,,
\end{eqnarray*}
thus leading to a contradiction with \eqref{ACu0 less than M0 balls proof}.

\medskip

We can thus assume that $m_1\ge 1/M_0$. By arguing as in step one in the proof of Theorem \ref{thm bubbling general} (see, in particular, the proof of \eqref{eq formulalagrange}), we see that \eqref{bella xi} implies
\[
n\,\ell_\e\,m_i=n\,\ell_\e\,\V(\xi_i) =n\, \AC(\xi_i)-2\,\e\,\int_{\R^n} |\nabla \xi_i|^2\,.
\]
By \eqref{ACu0 less than M0 balls proof}, we find
\[
\e\,\ell_\e\le\e\,\sum_{i=1}^M\,\AC(\xi_i)\le\e\,\AC(u_0)\le \e\,M_0^{1/n}\,\Psi(\e\,M_0^{1/n},1)\,.
\]
By $\e\,M_0^{1/n}<\s_0$ and by step one we also have $\Psi(\e\,M_0^{1/n},1)\le\Psi(\s_0,1)$, so that
\[
\e\,\ell_\e\le \s_0\,\Psi(\s_0,1)\,.
\]
Since $\s\,\Psi(\s,1)\to 0$ as $\s\to 0^+$, we conclude that, up to further decreasing the value of $\s_0$, we can guarantee $\e\,\ell_\e<\nu_0$ with $\nu_0$ as in Theorem $\Psi$-(ii). In particular, for each $i$ there exists $z_i\in\R^n$ such that
\[
\xi_i=\tau_{z_i}[\zeta_{\e,m_i}]\,,\qquad \ell_\e=\Lambda_{\e,m_i}\,.
\]
Since $\ell_\e$ is independent of $i$, $\ell_\e=\Lambda_{\e,m_i}$ implies that $m_i=m_1$ for all $i$. Therefore $1=\sum_{i=1}^Mm_i$ gives $m_1=1/M$, and the proof that $\xi_i=\tau_{z_i}[\zeta_{\e,1/M}]$ for each $i$ is complete.

\medskip

\noindent {\it Conclusion}: We are left to prove that $M$ is uniquely determined by the relation
\begin{equation}
  \label{unch}
M\,\Psi(\e,1/M)=\lim_{t\to\infty}\AC(u(t))\,.
\end{equation}
Since $M\le M_0$ and $\e\,M_0^{1/n}<\s_0\le\e_0$ we know that $M\in(0,(\e_0/\e)^n)$. To conclude that \eqref{unch} uniquely characterizes $M$ we are going to prove that
\[
x\mapsto x\,\Psi\Big(\e,\frac1x\Big)\,\,\mbox{is strictly increasing on}\,\, \Big(0,\Big(\frac{\e_0}{\e}\Big)^n\Big)\,,
\]
which, in turn, is equivalent to showing that $x\mapsto \Psi(\e,x)/x$ is strictly decreasing on $((\e/\e_0)^n,\infty)$.

\medskip

Let us set $f(x)=\Psi(\e,x)$ so that, by \eqref{Psi strict concave in m}, $f$ is concave on $(0,\infty)$ and strictly concave on $(a,\infty)$ (with $a=(\e/\e_0)^n$). Since $f(0^+)\ge0$ we deduce from $f(0^+)\le f(x)+f'(x)(0-x)$ for all $x>0$ that $x\,f'(x)\le f(x)$ for all $x>0$. In particular,
\begin{equation}
  \label{gio}
\frac{f(a)}a\ge f'(a)>f'(x)\qquad\forall x>a\,.
\end{equation}
If we now use first $f(a)<f(x)+f'(x)(a-x)$ for all $x>a$, and then \eqref{gio}, we find that
\[
x\,f'(x)-f(x)<a\,f'(x)-f(a)<0\,,\qquad\forall x>a\,.
\]
This shows that $x\mapsto f(x)/x$ is strictly decreasing on $(a,\infty)$, as claimed.
\end{proof}

\section{Strict stability of the diffused isoperimetric problem (Proof of Theorem \ref{theorem eps stability})}\label{section stability second variation}

\begin{proof}[Proof of Theorem \ref{theorem eps stability}] We aim at proving the theorem by combining two results. The first one is the radial case of Theorem \ref{theorem eps stability}, which was proved in \cite[Lemma 4.4]{maggi-restrepo}: if $\e\in(0,\e_0)$, then
\begin{equation}\label{eq stability radial}
\QQ_\e[\zeta_\e](\vphi)\geq \frac{1}{C}\int_{\R^n} \e\, |\nabla \vphi|^2+\frac{\vphi^2}\e\,,
\end{equation}
for every $\vphi\in W^{1,2}(\R^n)$ which is radial with respect to the origin and such that $\int_{\R^n} V'(\zeta_\e)\,\vphi=0$. The second one is the strict stability of the second variation of the (volume-constrained) area functional on the unit sphere. By the latter, we mean the following stability result: Let $\{A_i\}_{i\in\N}$ denote the normalized eigenfunctions of the Laplacian on $\SS^{n-1}$ and let $\{\mu_i\}_{i\in\N}$ denote the corresponding eigenvalues listed in increasing order, so that $-\Delta^{\SS^{n-1}}A_i = \mu_i\,A_i$ on $\SS^{n-1}$ for each $i$, $A_0$ is constant with $\mu_0=0$, $A_i(\theta)=\theta_i$ with $\mu_i=(n-1)$ for $i=1,...,n$,  $\mu_{n+1}>\mu_n$ and $\int_{\SS^{n-1}}A_i^2=1$. It is well known (see, e.g. \cite[Lemma 4.2]{CicaleseLeonardi}) that, for every $\a\in W^{1,2}(\SS^{n-1})$,
\begin{equation}
	\label{nonneg sphere}
	\int_{\SS^{n-1}}|\nabla^{\SS^{n-1}} \a|^2-(n-1)\,\a^2\ge 0\,,\qquad\mbox{if}\,\,\int_{\SS^{n-1}}\a=0\,,
\end{equation}
(stability of the sphere with respect to volume-preserving variations) and that
\begin{equation}
	\label{stability sphere}
	\int_{\SS^{n-1}}|\nabla^{\SS^{n-1}} \a|^2-(n-1)\,\a^2\ge c(n)\,\int_{\SS^{n-1}}|\nabla^{\SS^{n-1}}\a|^2+\a^2\,,
\end{equation}
if
\begin{equation}
	\label{stability sphere ortho cond}
	\int_{\SS^{n-1}}\a=0\,,\qquad \int_{\SS^{n-1}}\theta\,\a(\theta)\,=0\,,
\end{equation}
for some positive constant $c(n)$ (strict stability of the sphere with respect to volume-preserving variations orthogonal to translations). Combining these two stability results will require some a careful decompositions of functions $\vphi$ satisfying \eqref{stability scnd var ortho conditions}.

\medskip

\noindent {\it Step one}: In this first step we introduce a convenient way to rewrite
\begin{equation}
  \label{form}
\QQ_\e[\zeta_\e](\vphi)=\int_{\R^n}2\,\e\,|\nabla\vphi|^2+\Big\{\frac{W''(\zeta_\e)}\e-\Lambda_\e\,V''(\zeta_\e)\Big\}\,\vphi^2\,,
\end{equation}
by means of the transformation $\psi=\vphi/\zeta_\e'$, which allows one to relate $\QQ_\e[\zeta_\e]$ to the second variation of the volume-constrained area functional on the sphere. (This kind of transformation is inspired by similar computations found in \cite{tonegawa2005stable,le2011second,gaspar2020second}.) More precisely, setting $\hat{x}=x/|x|$ and denoting by $\zeta_\e'$, $\zeta_\e''$, etc. the derivatives of the radial profile of $\zeta_\e$, we show that if $\vphi\in W^{1,2}(\R^n)$ and, correspondingly, we define $\psi\in  W^{1,2}_{\loc}(\R^n\setminus\{0\})$ by
\begin{equation}
	\label{tone transform radial}
	\vphi=\psi\,\zeta_\e'\,,
\end{equation}
then
\begin{equation}
	\label{eq identity inner and outer}
	\QQ_\e[\zeta_\e](\vphi)=2\,\e\,\int_{\R^n}(\zeta_\e'^2)\,\Big\{|\nabla\psi|^2-\frac{(n-1)}{|x|^2}\,\psi^2\Big\}\,.
\end{equation}
Indeed, under \eqref{tone transform radial}, we have $\nabla\varphi=\zeta_\e'\,\nabla\psi+ \zeta_\e''\,\psi\,\hat{x}$, and integrating by parts\footnote{This is easily justified since $\zeta_\e$ and its derivatives all decay exponentially at infinity, see \cite[Theorem 3.1]{maggi-restrepo}.}
the mixed term in
\begin{equation}
	\label{eq change of var}
	|\nabla \varphi|^2
	=
	(\zeta_\e')^2\, |\nabla \psi|^2
	+ 2 \,\zeta_\e'\,\zeta_\e''\,\psi\, (\hat{x}\cdot\nabla \psi)
	+ (\zeta_\e'')^2\,\psi^2\,,
\end{equation}
we find
\begin{eqnarray}\nonumber
	2\,\int_{\R^n}\zeta_\e'\,\zeta_\e''\,\psi\,(\hat{x}\cdot\nabla \psi)
	&=&-\int_{\R^n}\psi^2\,\Div\big(\zeta_\e'\,\zeta_\e''\hat{x}\big)
	\\
	\label{hola}
	&=&-\int_{\R^n}\psi^2\, (\zeta_\e'')^2-\int_{\R^n} \psi^2\, \zeta_\e'\,\Big\{\zeta_\e'''+\frac{n-1}{|x|}\,\zeta_\e''\Big\}\,.
\end{eqnarray}
To rewrite the term with $\zeta_\e'''$ we make use of $2\,\e^2\,\Delta\zeta_\e=W'(\zeta_\e)-\e\,\Lambda_\e\,V'(\zeta_\e)$ on $\R^n$, which, writing $\Delta \zeta_\e$ in radial coordinates, takes the form
\begin{equation}
	\label{eq ODE}
	2\,\e^2\,\Big\{\zeta_\e''+ \frac{(n-1)}{|x|}\,\zeta_\e'\Big\}= W'(\zeta_\e)-\e\,\Lambda_\e\, V'(\zeta_\e)\,.
\end{equation}
Differentiating \eqref{eq ODE} in the radial direction we thus find
\[
2\,\e^2\,\Big\{\zeta_\e'''+ \frac{n-1}{|x|}\,\zeta_\e''-\frac{n-1}{|x|^2}\,\zeta_\e'\Big\}
= \zeta_\e'\,\Big\{W''(\zeta_\e)-\e\,\Lambda_\e\, V''(\zeta_\e)\Big\}\,,
\]
which can be combined into \eqref{hola} to obtain
\begin{eqnarray}\nonumber
	2\,\int_{\R^n}\zeta_\e'\,\zeta_\e''\,\psi\,(\hat{x}\cdot\nabla \psi)
	\!\!&=&\!\!-\int_{\R^n}\psi^2\, (\zeta_\e'')^2-\int_{\R^n}\psi^2\, (\zeta_\e')^2\,
	\Big\{\frac{W''(\zeta_\e)-\e\,\Lambda_\e\, V''(\zeta_\e)}{2\,\e^2}+\frac{n-1}{|x|^2}\Big\}\,.
\end{eqnarray}
On combining this last identity with \eqref{eq change of var} and the definition of $\QQ_\e[\zeta_\e]$ we thus find
\begin{eqnarray*}
	\QQ_\e[\zeta_\e](\vphi)&=&
	2\,\e\,\int_{\R^n}|\nabla\vphi|^2+\int_{\R^n}\Big\{\frac{W''(\zeta_\e)}\e-\Lambda_\e\,V''(\zeta_\e)\Big\}\,\vphi^2
	\\
	&=&
	2\,\e\,\int_{\R^n}(\zeta_\e')^2\, |\nabla \psi|^2+ (\zeta_\e'')^2\,\psi^2
	\\
	&&-2\,\e\,\int_{\R^n}\psi^2\, (\zeta_\e'')^2
	-\int_{\R^n} \psi^2\, (\zeta_\e')^2\,\Big\{\frac{W''(\zeta_\e)}\e-\Lambda_\e\, V''(\zeta_\e)+2\,\e\,\frac{n-1}{|x|^2}\Big\}
	\\
	&&+\int_{\R^n}\Big\{\frac{W''(\zeta_\e)}\e-\Lambda_\e\,V''(\zeta_\e)\Big\}\,\vphi^2
\end{eqnarray*}
which boils down to \eqref{eq identity inner and outer} thanks to $\vphi=\zeta_\e'\,\psi$.

\medskip

\noindent {\it Step two}: We prove that conclusion \eqref{stability scnd var W12} follows from conclusion \eqref{stability scnd var vphi only}. Indeed, arguing by contradiction, should \eqref{stability scnd var ortho conditions} imply \eqref{stability scnd var vphi only} but not \eqref{stability scnd var W12}, then we could find a sequence $(\vphi_j)_j$ in $W^{1,2}(\R^n)$ such that
\begin{eqnarray}
\label{ortos}
&&\int_{\R^n}\vphi_j\,V'(\zeta_\e)=0\,,\qquad \int_{\R^n}\vphi_j\,\nabla\zeta_\e=0\,,\qquad\forall j\,,
\\
\label{p1}
&&\int_{\R^n}|\nabla\vphi_j|^2+\vphi_j^2=1\,,\qquad\forall j\,,
\\
\label{p2}
&&\lim_{j\to\infty}\QQ_\e[\zeta_\e](\vphi_j)=0\,.
\end{eqnarray}
Having assumed that \eqref{stability scnd var ortho conditions} implies \eqref{stability scnd var vphi only}, we could deduce from \eqref{ortos} that
\begin{equation}
  \label{tesi2}
  \QQ_\e[\zeta_\e](\vphi_j)\ge\frac\e{C}\,\int_{\R^n}\vphi_j^2\,,\qquad\forall j\,.
\end{equation}
By combining \eqref{tesi2} with \eqref{p2} we would then find
\begin{equation}
  \label{p3}
  \lim_{j\to\infty}\int_{\R^n}\vphi_j^2=0\,,\quad\mbox{and, hence, by \eqref{p1},}\,\, \lim_{j\to\infty}\int_{\R^n}|\nabla\vphi_j|^2=1\,.
\end{equation}
But then, taking into account that $(W''(\zeta_\e)/\e-\Lambda_\e\,V''(\zeta_\e))\in L^\infty(\R^n)$, and combining \eqref{form} and \eqref{p3}, we could conclude that $\QQ_\e[\zeta_\e](\vphi_j)\to 2\,\e$ as $j\to\infty$, in contradiction with \eqref{p2}.

\medskip

\noindent {\it Step three}: By step two, we are left to prove that \eqref{stability scnd var ortho conditions} implies \eqref{stability scnd var vphi only}. In this step, we present an additional reduction. Let $L^2_{\rm rad}(\R^n)$ denote the set of radial functions with respect to the origin belonging to $L^2(\R^n)$, and let
\begin{eqnarray*}
W^{1,2}_{\rm rad}(\R^n)\!\!&=&\!\!L^2_{\rm rad}(\R^n)\cap W^{1,2}(\R^n)\,,
\\
Z\!\!&=&\!\!L^2_{\rm rad}(\R^n)^\perp\cap W^{1,2}(\R^n)\,,
\\
Z^*\!\!&=&\!\!\Big\{\vphi\in Z:\int_{\R^n}\vphi\,\nabla\zeta_\e=0\Big\}\,.
\end{eqnarray*}
We prove that if
\begin{equation}
\label{stability scnd var vphi only zetastar}
\QQ_\e[\zeta_\e](\vphi)\geq \frac{\e^2}C\,\int_{\R^n}\varphi^2\,,\qquad\forall\vphi\in Z^*\,,
\end{equation}
then \eqref{stability scnd var ortho conditions} implies \eqref{stability scnd var vphi only} (and the theorem is proved).

\medskip

To begin with, denoting by $\vphi_{\rm rad}$ and $\vphi_{\rm rad}^\perp$ the $L^2$-projections of $\vphi\in W^{1,2}(\R^n)$ on, respectively, $W^{1,2}_{\rm rad}(\R^n)$ and $Z$, so that $\vphi=\vphi_{\rm rad}+\vphi_{\rm rad}^\perp$, we notice that
\begin{equation}
  \label{q deco}
  \QQ_\e[\zeta_\e](\vphi)=\QQ_\e[\zeta_\e](\vphi_{\rm rad})+\QQ_\e[\zeta_\e](\vphi_{\rm rad}^\perp)\,,\qquad\forall\vphi\in W^{1,2}(\R^n)\,.
\end{equation}
Indeed,
\begin{eqnarray*}
\QQ_\e[\zeta_\e]\big(\vphi_{\rm rad},\vphi_{\rm rad}^\perp\big)
&=&
\int_{\R^n}2\,\e\,\nabla\vphi_{\rm rad}\cdot \nabla\vphi_{\rm rad}^\perp
+\Big(\frac{W''(\zeta_\e)}\e-\Lambda_\e\,V''(\zeta_\e)\Big)\vphi_{\rm rad}\,\vphi_{\rm rad}^\perp
\\
&=&
\int_{\R^n}
\Big\{-2\,\e\,\Delta\vphi_{\rm rad}+\Big(\frac{W''(\zeta_\e)}\e-\Lambda_\e\,V''(\zeta_\e)\Big)\vphi_{\rm rad}\Big\}\,
\vphi_{\rm rad}^\perp=0\,,
\end{eqnarray*}
where in the last identity we have used that $\vphi_{\rm rad}^\perp\in Z$ and the fact that the function in the curly bracket is radial.

\medskip

Next, we observe that the orthogonality relations in \eqref{stability scnd var ortho conditions} can be equivalently reformulated as follows for any $\vphi\in W^{1,2}(\R^n)$:
\begin{eqnarray}\label{ortos 1}
  &&\int_{\R^n}V'(\zeta_\e)\,\vphi=0\qquad\mbox{if and only if}\qquad\int_{\R^n}V'(\zeta_\e)\,\vphi_{\rm rad}=0\,,
  \\\label{ortos 2}
  &&\int_{\R^n}\vphi\,\nabla\zeta_\e=0\qquad\mbox{if and only if}\qquad\int_{\R^n}\vphi_{\rm rad}^\perp\,\nabla\zeta_\e=0\,,
\end{eqnarray}
since $V'(\zeta_\e)\in L^2_{\rm rad}(\R^n)$, $\nabla\zeta_\e=\zeta_\e'\,\hat{x}$, and $\int_{\R^n}\hat x\,\psi=0$ for all $\psi\in L^2_{\rm rad}(\R^n)$. In summary, by combining in the order \eqref{q deco}, \eqref{eq stability radial} applied to $\vphi_{\rm rad}$ (as we can do thanks to \eqref{ortos 1}) and \eqref{stability scnd var vphi only zetastar} to $\vphi_{\rm rad}^\perp$ (as we can do since $\vphi_{\rm rad}^\perp\in Z^*$ by \eqref{orto 2}) we deduce that
\[
\QQ_\e[\zeta_\e](\vphi)=\QQ_\e[\zeta_\e](\vphi_{\rm rad})+\QQ_\e[\zeta_\e](\vphi_{\rm rad}^\perp)
\ge\frac1{C\,\e}\,\int_{\R^n}(\vphi_{\rm rad})^2+\frac{\e}C\,\int_{\R^n}(\vphi_{\rm rad}^\perp)^2\ge \frac{\e}C\,\int_{\R^n}\vphi^2\,,
\]
that is \eqref{stability scnd var vphi only}.

\medskip

\noindent {\it Step four}: We now begin the proof of \eqref{stability scnd var vphi only zetastar} by introducing a Fourier decomposition of $\vphi\in W^{1,2}(\R^n)$ that is particularly convenient when $\vphi\in Z^*$. Since the radial Laplacian on $\R^n$, i.e. the map $B(r)\mapsto B''(r)+(n-1)\,B'(r)/r$ , defines a self-adjoint operator on $L^2((0,\infty),r^{n-1}\,dr)$, we can consider an orthonormal basis $\{B_j\}_{j=0}^\infty$ of $L^2((0,\infty),r^{n-1}\,dr)$ made up of its eigenfunctions. Since $L^2(\SS^{n-1})$ and $L^2((0,\infty),r^{n-1}\,dr)$ are separable Hilbert spaces, it follows that
$\{A_i\otimes B_j\}_{i,j=0}^\infty$ is an orthonormal basis of $L^2(K;\k)$, where $K=\SS^{n-1}\times(0,\infty)$ and $\k=(\H^{n-1}\llcorner\SS^{n-1})\times(r^{n-1}\,dr)$. In particular, since $\Phi:K\to\R^n$, $\Phi(\theta,r)=\theta\,r$, naturally induces an isometry between $L^2(\R^n)$ and $L^2(K,\k)$, we conclude that $\{(A_i\otimes B_j)\circ\Phi\}_{i,j=0}^\infty$ is an orthonormal basis of $L^2(\R^n)$. Moreover, each $a_{ij}=(A_i\otimes B_j)\circ\Phi$ is an eigenfunction of the Laplacian on $\R^n$, so that the orthogonality of $a_{ij}$ and $a_{hk}$ in $L^2(\R^n)$ (which holds true for $i\ne h$ or $j\ne k$) implies the orthogonality of $a_{ij}$ and $a_{hk}$ in $W^{1,2}_0(\R^n)$. In summary, whenever $\vphi\in W^{1,2}(\R^n)$ we have
\begin{eqnarray*}
&&\vphi=\sum_{i,j=0}^\infty\vphi_{ij}\,a_{ij}\qquad\mbox{in $W^{1,2}(\R^n)$}\,,
\\
&&\mbox{where}\,\,\vphi_{ij}=\int_{\R^n}a_{ij}\,\vphi=\int_{\R^n}A_i(\hat{x})\,B_j(|x|)\,\vphi(x)\,dx\,.
\end{eqnarray*}
If we define $\vphi_i^*\in W^{1,2}_{\rm rad}(\R^n)$ and $\vphi_i\in W^{1,2}(\R^n)$ by
\[
\vphi_i^*(x)=\sum_{j=0}^\infty\vphi_{ij}\,B_j(|x|)\,,\qquad \vphi_i(x)=A_i(\hat{x})\,\vphi_i^*(x)\,,\qquad x\in\R^n\,,
\]
then we have
\begin{equation}
\label{eq decomposition angular}
\vphi=\sum_{i=0}^\infty\vphi_i\qquad\mbox{in $W^{1,2}(\R^n)$}\,,
\end{equation}
as well as
\begin{equation}
\label{eq decomposition gradient}
\int_{\R^n} |\nabla \varphi|^2=\sum_{i=0}^\infty \int_{\R^n} |\nabla \varphi_i|^2\,,\qquad
\int_{\R^n} f\,\varphi^2 = \sum_{i=0}^\infty \int_{\R^n} f\,\varphi_i^2\,,
\end{equation}
for every $\vphi\in W^{1,2}(\R^n)$ and $f\in L^\infty(\R^n)$ that is either radial or angular (i.e., $f\circ\Phi$ depends either on $r$ or on $\theta$ only). We can also notice that
\begin{equation}
  \label{when in Z}
  \vphi_0=\vphi_0^*=0\qquad\forall\vphi\in Z\,,
\end{equation}
and then, since $A_0$ is a constant, if $\vphi\in Z$, then for every $j\in\N$
\[
\vphi_{0j}=\int_{\R^n}A_0(|\hat x|)\,B_j(|x|)\,\vphi(x)\,dx=A_0\,\int_{\R^n}B_j(|x|)\,\vphi(x)\,dx=0\,.
\]
In particular, by applying \eqref{eq decomposition gradient} with $f=(W''(\zeta_\e)/\e)-\Lambda_\e\,V''(\zeta_\e)$ and \eqref{when in Z} we find
\begin{equation}
\label{eq decomp quadratic form}
	\QQ_\e[\zeta_\e](\vphi) = \sum_{i=1}^\infty \QQ_\e[\zeta_\e](\vphi_i)\,,\qquad\forall\vphi\in Z\,.
\end{equation}
We make two claims:

\medskip

\noindent {\it Claim one}: if $\vphi\in Z^*$ and $i\ge n+1$, then
\begin{equation}
  \label{claim one}
  \QQ_\e[\zeta_\e](\vphi_i)\ge \frac{\e}{C}\,\int_{\R^n}\frac{\vphi_i^2}{|x|^2}\,.
\end{equation}

\medskip

\noindent {\it Claim two}: if $\vphi\in Z^*$ and $i=1,...,n$, then
\begin{equation}
  \label{claim two}
  \QQ_\e[\zeta_\e](\vphi_i)\ge \frac1{C\,\e}\,\int_{\R^n}\vphi_i^2\,.
\end{equation}

\medskip

We first show how to complete the proof of the theorem starting from these two claims, and then we prove the claims themselves.

\medskip

\noindent {\it Conclusion of the theorem from the claims}: Since $W''(0)>0$, $\zeta_\e\to 0$ as $|x|\to\infty$ (uniformly on $\e\in(0,\e_0)$), and $\e\,\Lambda_\e\to 0$ as $\e\to 0^+$, we see that there are universal constants $\k$ and $R_1$ such that
\begin{equation}
  \label{ffff}
  W''(\zeta_\e)-\e\,\Lambda_\e\,V''(\zeta_\e)\ge\k\,,\qquad\mbox{on $\R^n\setminus B_{R_1}$}\,,
\end{equation}
for every $\e\in(0,\e_0)$. By \eqref{claim one},
\begin{equation}
  \label{finefine one}
  \int_{B_{R_1}}\vphi_i^2\le \frac{C\,R_1^2}\e\,\QQ_\e[\zeta_\e](\vphi_i)\,,\qquad\forall i\ge n+1\,,
\end{equation}
so that, combining \eqref{when in Z}, \eqref{finefine one} and \eqref{claim two} we conclude that
\begin{equation}
  \label{finefine}
  \int_{B_{R_1}}\vphi^2\le \frac{C}\e\,\QQ_\e[\zeta_\e](\vphi)\,,\qquad\forall\vphi \in Z^*\,.
\end{equation}
Now let $L=\max_{B_{R_1}}|W''(\zeta_\e)-\e\,\Lambda_\e\,V''(\zeta_\e)|$. If
\[
\frac\k{2\,L}\,\int_{\R^n\setminus B_{R_1}}\vphi^2\le\int_{B_{R_1}}\vphi^2\,,
\]
then we deduce $\int_{\R^n}\vphi^2\le (C/\e)\,\QQ_\e[\zeta_\e](\vphi)$ by \eqref{finefine}; if, instead,
\begin{equation}
\label{finefine2}
\frac\k{2\,L}\,\int_{\R^n\setminus B_{R_1}}\vphi^2>\int_{B_{R_1}}\vphi^2\,,
\end{equation}
then by \eqref{ffff} and \eqref{finefine2} we see that
\begin{eqnarray*}
  \QQ_\e[\zeta_\e](\vphi)&\ge&\int_{\R^n}\Big\{\frac{W''(\zeta_\e)}\e-\Lambda_\e\,V''(\zeta_\e)\Big\}\,\vphi^2
  \ge\frac{\k}\e\,\int_{\R^n\setminus B_{R_1}}\vphi^2-\frac{L}\e\,\int_{B_{R_1}}\,\vphi^2
  \\
  &\ge&\frac\k{2\,\e}\,\int_{\R^n\setminus B_{R_1}}\vphi^2\,,
\end{eqnarray*}
which combined with \eqref{finefine} gives again $\int_{\R^n}\vphi^2\le (C/\e)\,\QQ_\e[\zeta_\e](\vphi)$, as desired. We are thus left to prove the two claims.

\medskip

\noindent {\it Proof of claim one}: Setting $\psi_i=\vphi_i/\zeta_\e'$ on $\R^n$, we see that $\psi_i(r\,\theta)=A_i(\theta)\,\xi_i(r)$ (where $\xi_i=\vphi_i^*/\zeta_\e'$), with
\[
|\nabla\psi_i|^2(r\,\theta)=|\nabla^{\SS^{n-1}}A_i(\theta)|^2\,\frac{\xi_i(r)^2}{r^2}+|A_i(\theta)|^2\,\xi_i'(r)^2\,.
\]
Since $i\ge n+1$ we can exploit \eqref{stability sphere} which, combined with \eqref{eq identity inner and outer}, gives
\begin{eqnarray*}
&&\frac{\QQ_\e[\zeta_\e](\vphi_i)}{2\,\e}=\int_{\R^n}(\zeta_\e')^2\,\Big\{|\nabla\psi_i|^2-\frac{(n-1)}{|x|^2}\psi_i^2\Big\}
\\
&\ge&\int_0^\infty\,\zeta_\e'(r)^2\,
\Big\{\Big(\frac{\xi_i(r)^2}{r^2}\,\int_{\SS^{n-1}}|\nabla^{\SS^{n-1}}A_i|^2-(n-1)\,A_i^2\Big)+\xi_i'(r)^2\int_{\SS^{n-1}}A_i^2\Big\}\,
r^{n-1}\,dr
\\
&\ge&c(n)\,
\int_0^\infty\,\zeta_\e'(r)^2\,
\Big\{\Big(\frac{\xi_i(r)^2}{r^2}\,\int_{\SS^{n-1}}|\nabla^{\SS^{n-1}}A_i|^2+\,A_i^2\Big)+\xi_i'(r)^2\int_{\SS^{n-1}}A_i^2\Big\}\,
r^{n-1}\,dr
\\
&=&c(n)\,\int_{\R^n}(\zeta_\e')^2\,\Big\{|\nabla\psi_i|^2+\frac{\psi_i^2}{|x|^2}\Big\}
\ge c(n)\,\int_{\R^n}(\zeta_\e')^2\,\frac{\psi_i^2}{|x|^2}=c(n)\,\int_{\R^n}\frac{\vphi_i^2}{|x|^2}\,,
\end{eqnarray*}
that is \eqref{claim one}.

\medskip

\noindent {\it Proof of claim two}: Since $\vphi\in Z^*$ we have, for each $i=1,...,n$,
\[
0=\int_{\R^n}\hat{x}_i\,\zeta_\e'\,\vphi=\sum_{k=1}^\infty\int_{\R^n}\hat{x}_i\,\zeta_\e'\,\vphi_k\,,
\]
which, combined with $\vphi_k(x)=A_k(\hat{x})\,\vphi_k^*(|x|)=c(n)\,\hat{x}_k\,\vphi_k^*(|x|)$, gives
\begin{equation}
  \label{orthowrong}
  \int_0^\infty\,\vphi_i^*(r)\,\zeta_\e'(r)\,r^{n-1}\,dr=0\,,\qquad\forall i=1,...,n\,.
\end{equation}
We now prove, if $\vphi\in Z^*$, then
\begin{equation}
  \label{claim two star}
  \QQ_\e[\zeta_\e](\vphi_i)\ge \frac1{C\,\e}\,\int_{\R^n}(\vphi_i^*)^2\,,\qquad\forall i=1,...,n\,.
\end{equation}
Notice that \eqref{claim two star} implies \eqref{claim two} since $\vphi_i\le c(n)\,\vphi_i^*$ for $i=1,...,n$.

\medskip

We prove \eqref{claim two star} by contradiction, following closely the proof of \cite[Lemma 4.4]{maggi-restrepo}. Indeed, should \eqref{claim two star} fail, then there would be sequences $\e_j\to 0^+$ as $j\to\infty$ and $(\vphi_j)_j$ in $Z^*$ such that (up to rotations taking $i_j\in\{1,...,n\}$ to $i_j=1$ for all $j$)
\begin{eqnarray}
\label{gr1}
  &&\frac1{\e_j}\,\int_{\R^n}((\vphi_j)_1^*)^2=1\,,\qquad\forall j\,,
  \\
\label{gr2}
  &&\int_0^\infty\,(\vphi_j)_1^*(r)\,\zeta_{\e_j}'(r)\,r^{n-1}\,dr=0\,,\qquad\forall j\,,
  \\
\label{gr3}
  &&\lim_{j\to\infty}\QQ_{\e_j}[\zeta_{\e_j}]\big((\vphi_j)_1\big)=0\,.
\end{eqnarray}
Setting $R_0=1/\om_n^{1/n}$ and
\[
\beta_j(s)=(\vphi_j^*)_1(R_0+\e_j\,s)\,,\qquad \eta_j(s)=\zeta_{\e_j}(R_0+\e_j\,s)\,,\qquad s\in\R\,,
\]
we can rewrite \eqref{gr1} and \eqref{gr2} as
\begin{eqnarray}\label{gr1 rescaled}
  &&\int_{I_j}\beta_j(s)^2\,(R_0+\e_j\,s)^{n-1}\,ds=1\,,\qquad\forall j\,,
  \\\label{gr2 rescaled}
  &&\int_{I_j}\beta_j(s)\,\eta_j'(s)\,(R_0+\e_j\,s)^{n-1}\,ds=0\,,\qquad\forall j\,,
\end{eqnarray}
where $I_j=(-R_0/\e_j,\infty)$. Concerning \eqref{gr3} we notice that by $(\vphi_j)_1(x)=A_1(\hat x)\,(\vphi_j)_1^*(|x|)$ it follows that
\[
|\nabla (\vphi_j)_1|^2=|\nabla^{\SS^{n-1}}\,A_1|^2\,\big((\vphi_j)_1^*\big)^2+A_1^2\,|\nabla(\vphi_j)_1^*|^2
\]
so that, as $\|A_1\|_{L^2(\SS^{n-1})}=1$,
\[
\e_j\,\int_{\R^n}|\nabla (\vphi_j)_1|^2\ge\e_j\,\int_{\SS^n}A_1^2\,\int_0^\infty\,|\nabla(\vphi_j)_1^*|^2(r)\,r^{n-1}\,dr
=\int_{I_j}\beta_j'(s)^2\,(R_0+\e_j\,s)^{n-1}\,ds\,.
\]
Again by Fubini's theorem and thanks to  $\|A_1\|_{L^2(\SS^{n-1})}=1$ we find
\[
\int_{\R^n}\Big\{\frac{W''(\zeta_{\e_j})}{\e_j}-\Lambda_{\e_j}\,V''(\zeta_{\e_j})\Big\}\,(\vphi_j)_1^2
=\int_{I_j} \big\{W''(\eta_j)-\e_j\,\Lambda_{\e_j}\,V''(\eta_j)\big\}\,\beta_j^2\,(R_0+\e_j\,s)^{n-1}\,ds
\]
We can thus deduce from \eqref{gr3} that
\begin{equation}
  \label{gr3 rescaled}
  \lim_{j\to\infty}\int_{I_j}\Big((\beta_j')^2+\big\{W''(\eta_j)-\e_j\,\Lambda_{\e_j}\,V''(\eta_j)\big\}\,\beta_j^2\Big)\,
  (R_0+\e_j\,s)^{n-1}\,ds=0\,;
\end{equation}
and, in fact, by taking into account that $\e_j\,\Lambda_{\e_j}\to 0^+$ as $j\to\infty$ and that $|V''|\le C$ on $[0,1]$, we see from \eqref{gr1 rescaled} that \eqref{gr3 rescaled} is equivalent to
\begin{equation}
  \label{gr3 rescaled 2}
  \lim_{j\to\infty}\int_{I_j}\Big((\beta_j')^2+W''(\eta_j)\,\beta_j^2\Big)\,(R_0+\e_j\,s)^{n-1}\,ds=0\,.
\end{equation}
In turn, since $|W''|\le C$ on $[0,1]$, by combining \eqref{gr1 rescaled} with \eqref{gr3 rescaled 2} we see that $(\beta_j)_j$ is bounded in $W^{1,2}_{\rm loc}(\R)$. Hence, up to extracting a subsequence, we can find $\beta\in W^{1,2}_{\rm loc}(\R)$ such that $\beta_j\weak\beta$ weakly in $W^{1,2}(\R)$.

\medskip

In this position, we can repeat {\it verbatim} two arguments contained in the proof of \cite[Lemma 4.4]{maggi-restrepo}. The first argument shows that the sequence of probability measures $(\mu_j)_j$ defined by $\mu_j=\beta_j^2\,(R_0+\e_j\,s)^{n-1}\,ds$ is in the compactness case of the concentration-compactness principle, and thus satisfies
\begin{equation}
  \label{cc holds}
  \lim_{s\to\infty}\,\sup_j\,\mu_j\big(\R\setminus(-s,s)\big)=0\,.
\end{equation}
The second argument shows that
\begin{equation}
  \label{second argument}
  \lim_{j\to\infty}\int_{I_j}W''(\eta_j)\,\beta_j^2\,(R_0+\e_j\,s)^{n-1}\,ds=R_0^{n-1}\,\int_\R\,W''(\eta_0)\,\beta^2\,,
\end{equation}
where $\eta_0(s)=\eta(s-\tau_0)$, $\eta$ is the unique solution of $-\eta'=\sqrt{W(\eta)}$ on $\R$ with $\eta(0)=1/2$, and $\tau_0=\int_\R\,V'(\eta(s))\,\eta(s)\,s\,ds$.

\medskip

By \eqref{cc holds}, \eqref{gr1 rescaled}, \eqref{second argument} and \eqref{gr3 rescaled 2} we thus find
\begin{eqnarray}\label{pg1}
&&R_0^{n-1}\,\int_\R\,\beta^2=1\,,
\\\label{pg2}
&&\int_\R\,2\,(\beta')^2+W''(\eta_0)\,\beta^2\le 0\,.
\end{eqnarray}
By \cite[Lemma 4.3]{maggi-restrepo}, \eqref{pg2} implies that $\beta(s+\tau_0)=t\,\eta'(s)$ for some $t\ne 0$ (the case $t=0$ is ruled out by \eqref{pg1}). In other words, $\beta=t\,\eta_0'$.

\medskip

We now claim that
\begin{equation}
  \label{finek}
  \lim_{j\to\infty}\int_{I_j}\beta_j\,\eta_j'\,(R_0+\e_j\,s)^{n-1}\,ds=R_0^{n-1}\,\int_\R\,\beta\,\eta_0'\,.
\end{equation}
Indeed, by $|\eta_j'(s)|\le C\,e^{-|s|/C}$ for $s\in\R$ and by \eqref{cc holds} we see that
\begin{eqnarray*}
&&\Big|\int_{I_j\setminus(-s_0,s_0)}\beta_j\,\eta_j'\,(R_0+\e_j\,s)^{n-1}\,ds\Big|
\\
&&\le\Big(\int_{I_j}(\eta_j')^2\,(R_0+\e_j\,s)^{n-1}\,ds\Big)^{1/2}\,\mu_j\big(I_j\setminus(-s_0,s_0)\big)^{1/2}\le\om(s_0)\,,
\end{eqnarray*}
for some function $\om$, independent of $j$, such that $\om(s)\to 0^+$ as $s\to\infty$. Similarly
\[
\Big|\int_{\R\setminus(-s_0,s_0)}\,\beta\,\eta_0'\Big|\le\om(s_0)\,,
\]
and therefore \eqref{finek} follows since, as $j\to\infty$, $\beta_j\to\beta$ in $L^2_{\rm loc}(\R)$ and by $\eta_j'\to\eta'_0$ locally uniformly on $\R$. On combining $\beta=t\,\eta_0'$ with \eqref{gr2 rescaled} and \eqref{finek} we conclude that
\[
0=t\,R_0^{n-1}\,\int_\R\,(\eta_0')^2\,,
\]
and thus, that $\eta\equiv{\rm constant}$, reaching a contradiction.
\end{proof}

\section{Exponential convergence to a single diffused bubble (Proof of Theorem \ref{theorem convergence without bubbling})}\label{section convergence without bubbling}

\begin{proof}[Proof of Theorem \ref{theorem convergence without bubbling}] We are proving the theorem by showing the existence if $\e\in(0,\e_0)$ and $u_0$ is as in the statement, then
\begin{eqnarray}
\label{eq convergence rate energy0}
\AC(u(t))-\Psi (\e,1)\!\!&\le&\!\! C(\e,u_0)\, e^{-t/C_*(\e)}\,,
\\
\label{eq single ball}
\big\| u(t)-\tau_{x_0}[\zeta_\e]\big\|_{L^2(\R^n)}\!\!&\le&\!\! C(\e,u_0)\, e^{-t/C(\e)}\,,\qquad\forall t>1/C(\e,u_0)\,.
\end{eqnarray}
By the assumptions on $u_0$, Theorem \ref{theorem bubbling one} holds with $M=1$. In particular, the only accumulation points for the sequences $(\AC(u(t_j)))_j$ and $(\lambda_\e[u(t_j)])_j$ corresponding to any $t_j\to\infty$ as $j\to\infty$ are, respectively, $\Psi(\e,1)$ and $\Lambda_\e$, so that we have
\begin{equation}
\label{eq convergence of the energy}
\lim_{t\to \infty} \AC(u(t)) = \Psi(\e,1)\,,\qquad\lim_{t\to\infty}\l_\e[u(t)]=\Lambda_\e\,.
\end{equation}

\medskip

\noindent {\it Step one}: We prove that
\begin{equation}\label{eq min in L2}
\lim_{t\to\infty}\big\| u(t)-\tau_{x(t)}[\zeta_{\e}]\big\|_{(W^{1,2}\cap C^0)(\R^n)}=0\,.
\end{equation}
where, for each $t>0$, we have defined $x(t)\in\R^n$ so that
\[
\big\|u(t)-\tau_{x(t)}[\zeta_\e]\big\|_{L^2(\R^n)}\le \big\|u(t)-\tau_{x}[\zeta_\e]\big\|_{L^2(\R^n)}\,,\qquad\forall x\in\R^n\,.
\]
Indeed, by \eqref{eq convergence of the energy}, if $(t_j)_j$ is an arbitrary sequence with $t_j\to\infty$ as $j\to\infty$, then $((u(t_j))_j$ is a minimizing sequence of $\Psi(\e,1)$. Now, in \cite[step two, proof of Theorem 2.1]{maggi-restrepo} it is proved that if $u_j$ is a minimizing sequence of $\Psi(\e,1)$, then, up to extracting subsequences, there is $x\in\R^n$ such that
\[
\lim_{j\to\infty}\big\| u_j-\tau_{x}[\zeta_{\e}]\big\|_{(W^{1,2}\cap C^0)(\R^n)}=0\,.
\]
By combining this fact with the definition of $x(t)$ we conclude the proof of \eqref{eq min in L2}.

\medskip

\noindent {\it Step two}: We introduce the Fisher information of the flow
\begin{equation}\label{eq Fisher}
\mathcal{I}_\e(t)=\e\,\int_{\R^n}\big(\pa_tu(t)\big)^2\,,\qquad\forall t>0\,,
\end{equation}
and prove that
\begin{equation}\label{eq fisher decay}
\mathcal{I}_\e (t) \leq C(\e,u_0)\, e^{-t/C(\e)}\,,\qquad\forall t>1/C(\e,u_0)\,.
\end{equation}
Indeed, by Theorem \ref{thm existence and regularity}-(vi) (see, in particular, \eqref{dissipation formula}) we have that $\I_\e\in W^{1,1}(a,\infty)$ for every $a>0$, with
\begin{equation}\label{eq bakryemery}
\frac{d}{dt}\,\frac{\I_\e(t)}2= -\int_{\R^n} \Big\{2\,\e\,|\nabla(\pa_tu)|^2
+\Big( \frac{W''(u)}\e-\lambda_\e[u(t)]\,V''(u)\Big)\,(\pa_tu)^2\Big\}\,.
\end{equation}
By $W''\in\Lip[0,1]$, $V''\in C^{0,\a(n)}[0,1]$, and $0<\Lambda_\e<C$ for all $\e\in(0,\e_0)$ (recall \eqref{limit of Lambda eps}) we find that, pointwise on $\R^n$,
\begin{eqnarray*}
&&\Big|\Big( \frac{W''(u(t))}\e-\lambda_\e[u(t)]\,V''(u(t))\Big)-\Big( \frac{W''(\tau_{x(t)}[\zeta_\e])}\e-\Lambda_\e\,V''(\tau_{x(t)}[\zeta_\e])\Big)\Big|
\\
&&\le C\,\Big\{\frac{|u(t)-\tau_{x(t)}[\zeta_\e]|}\e+\big|\lambda_\e[u(t)]-\Lambda_\e\big|+\big|u(t)-\tau_{x(t)}[\zeta_\e]\big|^{\a(n)}\Big\}\,.
\end{eqnarray*}
Combining this inequality with \eqref{eq bakryemery} we find that
\begin{eqnarray*}
&&-\frac{d}{dt}\,\frac{\I_\e(t)}2\ge\QQ_\e[\tau_{x(t)}\big[\zeta_\e]\big]\big(\pa_tu(t)\big)
\\\nonumber
&&-C\,\Big\{\frac{\|u-\tau_{x(t)}[\zeta_\e]\|_{C^0(\R^n)}}\e+\big|\lambda_\e[u(t)]-\Lambda_\e\big|
+\big\|u-\tau_{x(t)}[\zeta_\e]\big\|_{C^0(\R^n)}^{\a(n)}\Big\}\,\int_{\R^n}(\pa_tu(t))^2\,,
\end{eqnarray*}
so that \eqref{eq convergence of the energy} and \eqref{eq min in L2} imply the existence, for every $\eta>0$, of a positive constant $t_*=t_*(\e,u_0,\eta)$, such that
\begin{eqnarray}
\label{eq errorbakryemery}
&&-\frac{d}{dt}\,\frac{\I_\e(t)}2\ge\QQ_\e[\tau_{x(t)}\big[\zeta_\e]\big]\big(\pa_tu(t)\big)
-\eta\,\I_\e(t)\,,\qquad\forall t>t_*\,.
\end{eqnarray}
Now let $P_t$ denote the projection operator of $L^2(\R^n)$ onto its closed subspace
\[
Y_t=\Big\{\vphi\in L^2(\R^n):\int_{\R^n}V'\big(\tau_{x(t)}[\zeta_\e]\big)\,\vphi=\int_{\R^n}\pa_i\big(\tau_{x(t)}[\zeta_\e]\big)\,\vphi=0\Big\}\,.
\]
By Theorem \ref{theorem eps stability} we have
\begin{equation}
  \label{presta}
  \QQ_\e[\tau_{x(t)}\big[\zeta_\e]\big]\big(\pa_tu(t)\big)\ge\frac1{C(\e)}\,\int_{\R^n}\big(P_t[\pa_tu(t)]\big)^2\,.
\end{equation}
To get a control on $\I_\e(t)$ we thus need to estimate $\|\pa_tu(t)-P_t[\pa_tu(t)]\|_{L^2(\R^n)}$.

\medskip

To this end, let us begin by noticing that, since $\int_{\R^n}V'(u(t))\,\pa_tu(t)=0$ for all $t>0$, for every $\de>0$ there is $t_1=t_1(\de,u_0)>0$ such that
\begin{eqnarray}\nonumber
\Big|\int_{\R^n}V'(\tau_{x(t)}[\zeta_\e])\,\pa_tu(t)\Big|&\le& C\,\|u(t)-\tau_{x(t)}[\zeta_\e]\|_{L^2(\R^n)}\,\|\pa_tu(t)\|_{L^2(\R^n)}
\\\label{orto 1}
&\le&\de\,\|\pa_tu(t)\|_{L^2(\R^n)}\,,\qquad\forall t>t_1\,.
\end{eqnarray}
On the other hand, testing \eqref{diffused VPMCF} with $\pa_iu(t)\in W^{2,2}(\R^n)$ we find that
\begin{eqnarray*}
 \e^2\, \int_{\R^n}\pa_tu(t)\,\nabla_{e_i}u(t)=2\,\e^2\,\int_{\R^n}\Delta u(t)\,\pa_iu(t)
 -\int_{\R^n}\pa_i\big(W(u(t))-\l_\e[u(t)]\,V(u(t))\big)\,,
\end{eqnarray*}
(notice that $\l_\e[u(t)]$ is a function of $t$ alone, and is not affected by differentiation along $e_i$ here), and since $W(u(t)),V(u(t))\in W^{1,2}(\R^n)$ for all $t>0$, we conclude that
\[
\int_{\R^n}\pa_tu(t)\,\pa_iu(t)=2\,\int_{\R^n}\Delta u(t)\,\pa_iu(t)
=-2\,\int_{\R^n}\nabla u(t)\cdot\pa_i(\nabla u(t))=-\int_{\R^n}\pa_i|\nabla u(t)|^2=0\,,
\]
for all $t>0$. Combining this identity with \eqref{eq min in L2} we thus conclude that, up to further increase the value of $t_1=t_1(\de,u_0)$, we have
\begin{eqnarray}\nonumber
\Big|\int_{\R^n}\pa_i\big(\tau_{x(t)}[\zeta_\e]\big)\,\pa_tu(t)\Big|&\le& C\,\|\pa_iu(t)-\pa_i\tau_{x(t)}[\zeta_\e]\|_{L^2(\R^n)}\,\|\pa_tu(t)\|_{L^2(\R^n)}
\\\label{orto 2}
&\le&\de\,\|\pa_tu(t)\|_{L^2(\R^n)}\,,\qquad\forall t>t_1\,,i=1,...,n\,.
\end{eqnarray}
Therefore, by choosing $\de=\de(\e)>0$ small enough in terms of the constant $C(\e)$ appearing in \eqref{presta}, we conclude from \eqref{orto 1} and \eqref{orto 2} that, if $t>t_1=t_1(\de(\e),u_0)=t_1(\e,u_0)$, then, for a positive constant $C_*(\e)$ depending only on $\e$,
\begin{equation}
  \label{coe}
  \QQ_\e\big[\tau_{x(t)}\zeta_\e\big](\pa_tu(t))\geq \frac{\I_\e(\pa_tu(t))}{C_*(\e)}\,,\qquad\forall t>t_1\,.
\end{equation}
If we now choose $\eta=1/(2\,C_*(\e))$ in \eqref{eq errorbakryemery}, and correspondingly set $t_0=\min\{t_1,t_*\}$, then, we deduce from  \eqref{coe} that
\[
-\frac{d}{dt}\,\frac{\I_\e(t)}2\ge \frac{\I_\e(t)}{2\,C_*(\e)}\,,\qquad\forall t>t_0\,,
\]
from which \eqref{eq fisher decay} immediately follows.

\medskip

\noindent {\it Conclusion}: By Theorem \ref{thm existence and regularity}-(iv),
\[
\AC(u(T))-\AC(u(t)) = -\e\,\int_{t}^T\,ds\int_{\R^n} (\pa_tu(s))^2\,,\qquad\forall T>t>0\,.
\]
Combining this identity with \eqref{eq convergence of the energy} and \eqref{eq fisher decay} we find that, if $t>1/C(\e,u_0)$, then
\[
\AC(u(t))-\Psi(\e,1)= \int_t^\infty\,\I_\e(s)\,ds\le C(\e,u_0)\,\int_t^\infty e^{-s/C(\e)}\,ds\le C(\e,u_0)\,e^{-t/C(\e)}\,,
\]
thus proving \eqref{eq convergence rate energy0}. Next we notice that if $1/C(\e,u_0)<t<T$, then by combining the fundamental theorem of Calculus with the Minkowski inequality, the H\"older inequality and then with \eqref{eq fisher decay}, we obtain
\begin{eqnarray*}
  &&\|u(T)-u(t)\|_{L^2(\R^n)}\le\Big(\int_{\R^n}\Big|\int_t^T\pa_tu(s)ds\Big|^2\Big)^{1/2}
  \le\int_t^T\Big(\int_{\R^n}|\pa_tu(s)|^2\Big)^{1/2}\,ds
  \\
  &\le&\sum_{k=0}^\infty\int_{t+k}^{t+k+1}\Big(\int_{\R^n}|\pa_tu(s)|^2\Big)^{1/2}\,ds
  \le\sum_{k=0}^\infty\Big(\int_{t+k}^{t+k+1}ds\int_{\R^n}|\pa_tu(s)|^2\Big)^{1/2}
  \\
  &\le&\sum_{k=0}^\infty\Big(\int_{t+k}^\infty ds\int_{\R^n}|\pa_tu(s)|^2\Big)^{1/2}
  \le C(\e,u_0)\,\sum_{k=0}^\infty e^{-(t+k)/C(\e)}\le C(\e,u_0)\,e^{-t/C(\e)}\,,
\end{eqnarray*}
that is
\begin{equation}
  \label{sinnerberrettini}
  \|u(T)-u(t)\|_{L^2(\R^n)}\le C(\e,u_0)\,e^{-t/C(\e)}\,,\qquad\forall T>t>1/C(\e,u_0)\,.
\end{equation}
Now let $t_j\to\infty$ as $j\to\infty$: since $(u(t_j))_j$ is a minimizing sequence of $\Psi(\e,1)$, then, by the argument in step one and up to extracting a subsequence, there is $x_0\in\R^n$ such that $\|u(t_j)-\tau_{x_0}[\zeta_\e]\|_{L^2(\R^n)}\to 0$ as $j\to\infty$. By taking $T=t_j$ in \eqref{sinnerberrettini} and letting $j\to\infty$ in the corresponding inequality we thus complete the proof of \eqref{eq single ball}, and thus, of the theorem.
\end{proof}

\bibliographystyle{alpha}
\bibliography{references_mod}	
\end{document}